\documentclass[12pt]{amsart}

\usepackage{a4wide}

\usepackage{graphicx}

\usepackage{amssymb}
\usepackage{amsthm}

\usepackage{color}

\usepackage{arydshln}

\usepackage{chngcntr}

\usepackage{appendix}

\usepackage{arydshln}

\counterwithin{table}{section}
\counterwithin{figure}{section}

\newtheorem{thm}{Theorem}[section]
\newtheorem{prop}[thm]{Proposition}
\newtheorem{lem}[thm]{Lemma}

\newtheorem{dfn}[thm]{Definition}
\newtheorem{cor}[thm]{Corollary}
\theoremstyle{remark}
\newtheorem{rk}[thm]{Remark}
\newtheorem*{rule123*}{{\bf Side-pairing selection process (123-rule)}}

\newtheorem{hyp}{Assumption}

\def\cqfd{\mbox{}\nolinebreak\hfill$\Box$\medbreak\par}
\newenvironment{pf}{\noindent\textbf{Proof:}}{\cqfd}

\newcommand{\tr}{\textrm{Tr}}
\newcommand{\re}{\mathfrak{Re}}

\newcommand{\br}{\textrm{br}}

\renewcommand{\r}{\rho}
\newcommand{\s}{\sigma}
\renewcommand{\t}{\tau}

\newcommand{\p}{p}

\title[Non-arithmetic complex hyperbolic lattices]{New non-arithmetic complex hyperbolic lattices II}

\date{Dec 10, 2019}

\author{Martin Deraux, John R. Parker and Julien Paupert}

\newcommand{\pu}{{\rm PU}}
\newcommand{\su}{{\rm SU}}
\newcommand{\C}{\mathbb{C}}
\newcommand{\Q}{\mathbb{Q}}
\newcommand{\Z}{\mathbb{Z}}
\newcommand{\N}{\mathbb{N}}
\newcommand{\R}{\mathbb{R}}
\newcommand{\B}{\mathcal{B}}
\newcommand{\RH}[1]{{{\bf H}^{#1}_{\R}}}
\newcommand{\CH}[1]{{{\bf H}^{#1}_{\C}}}
\newcommand{\CP}[1]{{{\bf P}^{#1}_{\C}}}
\newcommand{\CHB}[1]{{{\bf \overline{H}}^{#1}_{\C}}}

\renewcommand{\S}{\mathcal{S}}

\newcommand{\T}{\mathcal{T}}

\begin{document}

\begin{abstract}
  We describe a general procedure to produce fundamental domains for
  complex hyperbolic triangle groups. This allows us to produce new
  non-arithmetic lattices, bringing the number of known non-arithmetic
  commensurability classes to 22.
\end{abstract}


\maketitle

\section{Introduction}

The goal of this paper is to give a unified construction of various
families of lattices in the isometry group $\pu(2,1)$ of the complex
hyperbolic plane $\CH 2$. We describe a systematic manner to produce
fundamental domains that works for all known triangle group lattices,
with minor modifications for some pathological cases.

The groups we consider turn out to produce all previously known
examples of non-arithmetic lattices in $\pu(2,1)$. Until recently, all
such groups were contained, up to commensurability, in the list of
lattices that appears in work of Deligne-Mostow/Thurston,
see~\cite{delignemostow},~\cite{mostowihes},~\cite{thurstonshapes}. In
fact, a lot of these groups were discovered over a century ago by
Picard~\cite{picard}, and studied by several people including
Terada~\cite{terada}. These groups give nine commensurability classes
of non-arithmetic lattices, but the determination of the precise
number of commensurability classes required a significant amount of
work
(see~\cite{sauter},~\cite{delignemostowbook},~\cite{kappesmoller},~\cite{mcmullengaussbonnet}).

In~\cite{dpp2}, we announced the construction of 12 lattices, giving
at least 9 new non-arithmetic commensurability classes. The most
difficult part of the result is the proof that the groups are
lattices.  Indeed, the fact that they are not commensurable to any
Deligne-Mostow lattice can be proved by the somewhat rough
commensurability invariant given by the field generated by traces in
the adjoint representation. The fact that they are not arithmetic
follows from a standard application of the complex reflection version
of the Vinberg arithmeticity criterion, see for
instance~\cite{paupert}.

The proof of discreteness relies on the construction of an explicit
fundamental domain for each group. There are general ways to produce
such fundamental domains, for instance Dirichlet domains, but these
often turn out to give overly complicated combinatorics, see for
instance~\cite{deraux4445}.

The domains used in~\cite{dpp2} are quite simple and natural. Their
vertices are all given by (well-chosen) intersections of mirrors of
reflections in the group, their 1-faces are all geodesic arcs, and
2-faces are as natural as possible in the context of the non-constant
curvature geometry of the complex hyperbolic plane, as they lie on
complex lines or Giraud disks.

The combinatorial construction of the domains was inspired by the
fundamental domains constructed by Rich Schwartz in~\cite{rhochi}, and
work related to James Thompson's thesis~\cite{thompson}. The general
procedure turns out to be quite elementary, and a lot of it can be
described by hand, even though a lot of the computations are much
easier to perform with a computer.

The computational heart of our argument is the proof that the
geometric realization gives an embedding of our combinatorial
fundamental domain into $\CHB 2$. This relies on interval arithmetic
in conjunction with the rational univariate representation (RUR) for
0-dimensional polynomial systems (see~\cite{rouillier}). Software that
performs these verifications is publicly available~\cite{spocheck}; it
is based on the implementation of the RUR that was developed by
B. Parisse in giac, see~\cite{giac}.

We will review and clarify the construction, and show that it applies
to a wide class of complex hyperbolic triangle groups. As a result, we
get new fundamental domains for many groups that appeared previously
in the literature, including many of the Deligne-Mostow lattices.

Our methods also allow us to treat the 6 sporadic triangle groups left
over from~\cite{dpp2}. We denote by $\S(\p,\tau)$ the sporadic
triangle group generated by a complex reflection $R_1$ with rotation
angle $2\pi/\p$ and an order 3 isometry $J$ with $\tr(R_1J)=\tau$;
  see Table~\ref{tab:tauvalues} for the meaning of the notation
  $\sigma_j$. The family of groups
    $\S(\p,\bar\sigma_4)$ was studied in \cite{dpp2}.
\begin{thm} \label{thm:sporadicNA}
The groups $\S(\p,\sigma_1)$ are non-arithmetic lattices for
$\p=3,4,6$.  The groups $\S(\p,\sigma_5)$ are non-arithmetic lattices
for $\p=3,4$. They are not commensurable to any Deligne-Mostow
lattice, nor to any lattice of the form $\S(\p,\bar\sigma_4)$.
\end{thm}
\begin{thm} \label{thm:sporadicA}
The group $\S(2,\sigma_5)$ is an arithmetic lattice, and so are the
groups $\S(\p,\sigma_{10})$ for $\p=3,4,5,10$.
\end{thm}

We also consider a slightly different family of lattices $\T(\p,{\bf
  T})$, that comes out of James Thompson's thesis (see
Table~\ref{tab:Tvalues}). We prove that some of them are
non-arithmetic and also that they are new, in the
sense that they are not commensurable to any Deligne-Mostow lattice,
nor to any sporadic triangle group.
\begin{thm} \label{thm:thompson}
  The groups $\T(\p,{\bf S}_2)$ for $\p=4,5$ and $\T(3,{\bf H_2})$ are 
  non-arithmetic lattices. They are not commensurable to each other, to any 
  Deligne-Mostow lattice, nor to any sporadic triangle group.
\end{thm}
This statement follows from the analysis of their adjoint trace fields
(see section~\ref{sec:tracefields}) and their non-arithmeticity index,
see section~\ref{sec:spectrum}.  A more detailed analysis, requiring
more subtle arguments, shows the following (see
section~\ref{sec:commensurability}, Table~\ref{tab:commclasses} in
particular).
\begin{thm}\label{thm:numberofclasses}
  The currently known non-arithmetic lattices in $\pu(2,1)$ come in 22
  commensurability classes.
\end{thm}
It was recently observed~\cite{derauxklein} that some of these
non-arithmetic lattices actually appear in a list of lattices
constructed by Couwenberg, Heckman and Looijenga~\cite{CHL}, that gave
a common generalization of work of
Barthel-Hirzebruch-H\"ofer~\cite{BHH} and
Deligne-Mostow~\cite{delignemostow}. We refer to these lattices as CHL
lattices. Note that, apart from Deligne-Mostow lattices, the CHL
lattices contain three families of 2-dimensional lattices,
corresponding to line arrangements in $\CP 2$ of type $H_3$, $G_{24}$
and $G_{26}$. Using the same techniques as the ones
in~\cite{derauxklein}, one verifies that these three families
correspond to our families $\S(p,\sigma_{10})$,
$\S(p,\overline{\sigma}_4)$ and $\T(p,{\bf S_2})$, respectively (see
also~\cite{derauxabel}).

Using the analysis in section~\ref{sec:commensurability}
(Table~\ref{tab:commclasses}), we see that our lattices
$\S(p,\sigma_{1})$ ($p=3,4,6$), $\S(p,\sigma_5)$ ($p=3,4$) and
$\T(3,{\bf H_2})$ are non-arithmetic lattices that are not
commensurable to any CHL lattice.

We assume the reader is familiar with basic notions of hyperbolic
geometry over some base field, and with Coxeter groups. To get a quick
idea of the main differences between real and complex hyperbolic
geometry, the reader can consult~\cite{chengreenberg}. We will freely
use the classification of isometries into elliptic, parabolic and
loxodromic elements, sometimes with slight refinements, e.g. a regular
elliptic isometry is an elliptic isometry whose matrix representatives
have distinct eigenvalues.  We refer to~\cite{goldman} for background
on complex hyperbolic geometry and bisectors, see also section~2
of~\cite{dpp2} for a quick review.

{\bf Acknowledgements:} Part of this work took place while the authors
were visiting ICERM, Arizona State University, Universit\'e
Grenoble-Alpes, Durham University and Tokyo Institute of Technology.
The authors would like to thank these institutions for their
hospitality during these visits. The first author would like to thank
Bernard Parisse for his help in the development of
spocheck~\cite{spocheck}, and his flexibility in adapting giac.  The
second author was partly supported by a JSPS Invitation Fellowship
L16517. The third author was partly supported by Simons Foundation
Collaboration Grant for Mathematicians 318124 and National Science
Foundation grant DMS 170846.

\section{Groups generated by two complex reflections}

\subsection{Subgroups of $\pu(1,1)$ generated by two elliptic elements}

Let $b$ and $c$ be two elliptic elements in $\pu(1,1)$, which we
assume to be primitive of the same order, i.e. they rotate in $\CH 1$
by an angle $2\pi/\p$, $\p\in\N$, $\p\geq 2$. It is a well known fact
that the discreteness of the group generated by $b$ and $c$ is
controlled by the product $bc$, in the following sense.
\begin{prop}\label{prop:knapp}
  If $\langle b,c\rangle$ is a lattice, then $bc$ is
  non-loxodromic. If $bc$ is elliptic, then the group is a lattice if
  and only if $bc$ rotates by an angle $4\pi/n$ for some $n\in\N^*$,
  or by $8\pi/\p$.
\end{prop}
The first part follows from a straightforward application of the
Poincar\'e polyhedron theorem. The second one is more subtle, it is a
special case of Knapp's theorem, see~\cite{knapp}.

Proposition~\ref{prop:knapp} will serve as a model for higher
dimensional analogues (we will look for simple words in the generators
whose behavior determines whether or not the group is a lattice), and
it is also important because it explains the behavior of subgroups
generated by two complex reflections in $\CH 2$ (by looking at the
projective line of lines through the intersection of the mirrors,
possibly in projective space).

A natural analogue of the elliptic elements of $\pu(1,1)$ for higher
dimensions is given by complex reflections in $\pu(n,1)$, whose
representative matrices have an eigenvalue of multiplicity
$n$. Geometrically, such an isometry fixes pointwise a complex
projective hyperplane called its mirror, and rotates about it by a
certain angle. In the next section, we discuss groups generated by two
such complex reflections.

\subsection{Subgroups of $\su(2,1)$ generated by two complex reflections}
\label{sec:pair-cx-ref}

Let $A,B\in \su(2,1)$ be complex reflections with angle $2\pi/\p$,
with distinct mirrors. We assume they each have eigenvalues $u^2,\bar
u, \bar u$, where $u=e^{2\pi i/3\p}$. Let ${\bf a}$ and ${\bf b}$ be
polar vectors to the mirrors of $A$ and $B$ respectively; that is
${\bf a}$ and ${\bf b}$ are $u^2$-eigenvectors.  Note that $\bar u^2$
is an eigenvalue of $AB$, corresponding to the intersection of the
multiple eigenspaces of $A$ and $B$.  Indeed using formulae
in~\cite{pratoussevitch}, see also~\cite{parkertraces}, we can write
down the trace of $AB$.

\begin{lem}\label{lem:tr-AB}
Let $A$ and $B$ be as above. Then
$$
{\rm tr}(AB)=\left(2-\frac{|u^3-1|^2\bigl|\langle{\bf a},{\bf b}\rangle\bigr|^2}
{\langle{\bf a},{\bf a}\rangle\langle{\bf b},{\bf b}\rangle}\right)u+\bar u^2.
$$
\end{lem}

We are interested in the case where $AB$ is elliptic of finite order. The 
following proposition follows easily from Lemma~\ref{lem:tr-AB}. 

\begin{prop}\label{prop:pair-cx-ref}
Let $A$ and $B$ be as above. Then the following are equivalent.
\begin{enumerate}
\item[(1)] 
$$
{\rm tr}(AB)=\bigl(2-4\cos^2(\theta)\bigr)u+\bar u^2=-2\cos(2\theta)u+\bar u^2,
$$
\item[(2)] $AB$ has eigenvalues $-ue^{2i\theta}$, $-ue^{-2i\theta}$, $\bar u^2$,
\item[(3)] 
$$
\frac{|u^3-1|^2\bigl|\langle{\bf a},{\bf b}\rangle\bigr|^2}
{\langle{\bf a},{\bf a}\rangle\langle{\bf b},{\bf b}\rangle}=4\cos^2(\theta).
$$
\end{enumerate}
In particular, if $AB$ has finite order then $\theta$ is a rational multiple of $\pi$.
\end{prop}

If the $\bar u^2$-eigenspace of $AB$ is spanned by a negative vector,
it corresponds to the intersection of the mirrors of $A$ and $B$ and
is a fixed point of $AB$ in ${\bf H}^2_{\mathbb C}$. 
If the $\bar u^2$-eigenspace of $AB$ is spanned 
by a positive vector then it is polar to a complex line preserved by $AB$
which is orthogonal to the mirrors of $A$ and $B$.

\subsection{Braid length} \label{sec:braiding}

Throughout the paper, we will use the following terminology for {\bf
  braid relations} between group elements (see Section 2.2 of Mostow
  \cite{mostowpacific}). If $G$ is a group and
$a,b\in G$, we say that $a$ and $b$ satisfy a braid relation of length
$n\in\N^*$ if
\begin{equation}\label{eq:braiding}
(ab)^{n/2}=(ba)^{n/2},
\end{equation}
where powers that are half integers should be interpreted as saying
that the corresponding alternating product of $a$ and $b$ should have
$n$ factors.  For instance, $(ab)^{3/2}=aba$, $(ba)^2=baba$,
$(ab)^{5/2}=ababa$, etc. 
For short, we will sometimes write the sentence ``$a$ and
  $b$ satisfy a braid relation of length $n$'' simply as ``$\br_n(a,b)$''.

If $a$ and $b$ satisfy some braid relation, the smallest $n$ such
that~\eqref{eq:braiding} holds will be called the {\bf braid length}
of the pair $a,b$, which we will denote by $\br(a,b)$.
\begin{rk}\label{rk:braiding}
  \begin{itemize}
    \item A braid relation of length 2 simply means $a$ and $b$ commute.
    \item The classical braid relation $aba=bab$ is a braid relation
      of length 3.
    \item If $a$ and $b$ both have order 2, $\br(a,b)=n$ if and only
      if their product has order $n$.
    \item If $\br_n(a,b)$ holds for some integer $n$, then clearly the
      relation $\br_{kn}(a,b)$ also holds for every integer $k>1$. In
      particular, the relation $\br_n(a,b)$ does not imply
      $\br(a,b)=n$, but it does imply $\br(a,b)$ divides $n$.
  \end{itemize}
\end{rk}

It will be useful later in the paper to consider in some detail the
case where ${\rm tr}(AB)=-2u\cos(2\pi/q)+\bar u^2$ for some $q\in\N^*$. 
That is, we take $\theta=\pi/q$ in Proposition \ref{prop:pair-cx-ref}.

\begin{prop}\label{prop:braiding}
Let $u=e^{2\pi i/3\p}$ for some integer $p>1$.  
If ${\rm tr}(AB)=-2u\cos(2\pi/q)+\bar u^2$ for some integer $q>1$, then
  $\br(A,B)=q$. Moreover, the group generated by $A$ and
  $B$ in $\pu(2,1)$ is a central extension of the rotation subgroup of
  a triangle group. 
  \begin{enumerate}
  \item 
  If $q$ is odd, then its center is generated by
  $(AB)^q$, which is a complex reflection with angle $\frac{(q-2)p-2q}{p}\pi$. 
  The corresponding quotient is a $(2,p,q)$-triangle group.
  \item 
  If $q$ is even, the center is generated by $(AB)^{q/2}$, which is a complex
  reflection with angle  $\frac{(q-2)p-2q}{2p}\pi$.
  The quotient is a $(\frac{q}{2},p,p)$ triangle group.
  \end{enumerate}
  In particular, if $r=\frac{2pq}{(q-2)p-2q}$ is an integer then
  the order of $AB$ is the least common multiple of $q$ and $r$ when
  $q$ is odd and it is the least common multiple of $q/2$ and $r$ when $q$
  is even.
\end{prop}

We remark that in most cases which we consider, the least common multiple in 
the last part of this result is $r$. However, when $p=12$ and $q=6$ we have 
$r=4$ and the order of $AB$ is 12. This will arise for the group $\T(12,{\bf E}_2)$ 
below.

When mentioning $(k,l,m)$-triangle groups, we always assume
$k,l,m\geq2$ are integers. By the rotation subgroup of a
$(k,l,m)$-triangle group, we mean the index two subgroup of
orientation preserving isometries in the group generated by real
reflections in the sides of a triangle with angles $\pi/k$, $\pi/l$,
$\pi/m$ (note that such a triangle lives in $\RH 2$, $\R^2$ or $S^2$
depending on $k,l,m$). In other words, it is generated by rotations
around the vertices of the triangle, with respective angles $2\pi/k$,
$2\pi/l$, $2\pi/m$.  We note that the triangle group in
Proposition~\ref{prop:braiding} is spherical, Euclidean or hyperbolic
whenever $1/p+1/q-1/2$ is positive, zero or negative
respectively. This happens if and only if the angle
$\frac{(q-2)p-2q}{p}\pi$ or $\frac{(q-2)p-2q}{2p}\pi$ is negative,
zero or positive respectively.

The equation $r=\frac{2pq}{(q-2)p-2q}$ is equivalent to
$2/p+2/q+2/r=1$. We are interested in solutions with
$p,\,q,\,r\in\Z\cup\{\infty\}$ (and the usual convention that
$1/\infty=0$).  Since this is symmetric when we permute $p$, $q$ and
$r$, it suffices to give the set $\{p,q,r\}$ and allow
permutations. Hence we give solutions with $1/r\le 1/q\le
1/p<1$. There is one infinite family of solutions, namely $(2,q,-q)$,
as well as the following finite list:
$$
\begin{array}{lllll}
(3,3,-6),\ & (3,4,-12),\ & (3,5,-30),\ & (3,6,\infty),\ & (3,7,42), \\
(3,8,24),\ & (3,9,18),\ & (3,10,15),\ & (3,12,12), \ & (4,4,\infty), \\
(4,5,20),\ & (4,6,12),\ & (4,8,8),\ & (5,5,10), & (6,6,6).
\end{array}
$$

\section{Subgroups of $\pu(2,1)$ generated by three complex reflections} \label{sec:notation}

We now wish to analyze groups generated by three complex reflections
$R_1$, $R_2$ and $R_3$ in $\pu(2,1)$.  Throughout, we will consider
triangle groups whose generators have the same rotation angle, given
by $2\pi/\p$. If the triangle group is equilateral, i.e. there is an
elliptic isometry cyclically permuting the mirrors of the generators,
we write $J$ for that isometry, and order the reflections so that
$R_2=JR_1J^{-1}$, $R_3=JR_2J^{-1}$. We then write
$$
P=R_1J,\quad Q=R_1R_2R_3.
$$ 
It is straightforward to check that, in the equilateral case,
$Q=P^3$.

For reasons that will become clear later, we assume that $Q$ has an
isolated fixed point. This assumption may seem somewhat unnatural, but
the discussion in the previous section should make it more natural in
the search for lattices (rather than simply discrete groups).

The central motivating question of this paper is the following:
\begin{center}
 \textbf{
  When is the group generated by $R_1$, $R_2$ and $R_3$ a lattice?}
\end{center}
It is a folklore belief that the discreteness of the group should be
controlled by explicit short words in the generators. In the special
case where the $R_j$ are involutions, a precise conjectural statement
was given by Rich Schwartz in~\cite{schwartzICM}, where the
conjectural control words actually depend on the triangle. In his
Ph.D. thesis, James Thompson gave a conjectural list of the triangle
groups (with involutive generators) that were not only discrete, but
actually lattices (his work was partly motivated by the example
in~\cite{deraux4445}).

A guiding principle (which is at this stage far from justified rigorously)
is that, if the group is to be a lattice, then
\begin{itemize}
  \item for all $j=1,2,3$, $R_j$ and $R_{j+1}$ should generate a
    lattice in $\pu(1,1)$ (or in $\pu(2)$), in particular $R_1R_2$, 
    $R_2R_3$ and $R_3R_1$ should all be
    should be non-loxo\-dro\-mic;
  \item $R_1R_2R_3$ should be non-loxodromic;
  \item $R_1R_2R_3R_2^{-1}$, $R_1R_3^{-1}R_2R_3$ and
    $R_3R_1R_2R_1^{-1}$ should be non-loxo\-dro\-mic.
\end{itemize}

Throughout the paper, we will use word notation in the generators
$R_1$, $R_2$, $R_3$, and denote these group elements simply by
$1,2,3$. Hoping that no confusion with complex conjugation occurs, we
will also denote their inverses by $\bar1, \bar2, \bar3$. In
particular, the above control words read $12$, $23$, $31$, $123$, $123\bar2$, 
$1\bar323$, $312\bar1$, etc.

\subsection{Equilateral triangle groups} \label{sec:equilateral}
The idea in the above guiding principle was used to give a rough sieve
of the lattice candidates in~\cite{parkerunfaithful},~\cite{parkerpaupert}, whose results we now briefly recall. The basic
point is that equilateral triangle groups can be parametrized by the
order $\p$ of the generators and the complex parameter
$$
  \tau=\tr(R_1J).
$$  
Writing ${\bf n}_j$ for a polar vector to the mirror of $R_j$ and
$u=e^{2\pi i/3\p}$, an equivalent definition of $\tau$ is
$$
\tau=(u^2-\bar{u})\frac{\langle{\bf n}_{j+1},{\bf n}_j\rangle}
{\Vert{\bf n}_{j+1}\Vert\,\Vert{\bf n}_j\Vert}.
$$

The precise statement about parametrizing groups by the pair $\p,\tau$
is the following.
\begin{prop}
  Let $\p\in\mathbb{N}$, $\p\geq 2$ and $\tau\in\mathbb{C}$. We write
  $u=e^{2\pi i/3\p}$, $\alpha=2-u^3-\overline{u}^3$ and
  $\beta=(\overline{u}^2-u)\tau$.  Then there exists a complex
  reflection $R_1$ with rotation angle $2\pi/p$ and a regular elliptic
  element $J$ in $\su(2,1)$ such that $\tr(R_1J)=\tau$ if and only if
  \begin{equation}
    \alpha^3+2\re(\beta^3) - 3 \alpha|\beta|^2 < 0. \label{eq:signature}
  \end{equation}
\end{prop}
In fact, using a basis for $\C^3$ consisting of vectors polar
to the mirrors of the reflections $R_j$, we can write
$$
H=\left(\begin{matrix}
  \alpha & \beta & \overline{\beta}\\
  \overline{\beta}&\alpha&\beta\\
  \beta&\overline{\beta}&\alpha
\end{matrix}\right)
,\quad 
R_1=\left(\begin{matrix}
  u^2 & \tau & -u\overline{\tau}\\
  0 & \overline{u} & 0\\
  0 & 0 & \overline{u}
\end{matrix}\right)
,\quad 
J=\left(\begin{matrix}
  0 & 0 & 1\\
  1 & 0 & 0 \\
  0 & 1 & 0
\end{matrix}\right),
$$ 
and the expression that appears in equation~\eqref{eq:signature} is
simply the determinant of $H$.  We denote by $\S(\p,\tau)$ the
corresponding group (we will always assume that~\eqref{eq:signature}
is satisfied). Note that the generating pair is almost uniquely
determined by $p$ and $\tau$, in the following sense.
\begin{prop}
  Let $R_1$, $R_1'$ be complex reflections of angle $2\pi/p$, let $J$,
  $J'$ be regular elliptic elements of $\su(2,1)$. Denote by
  $\tau=\tr(R_1J)$, $\tau'=\tr(R_1'J')$. If the pairs $(R_1,J)$ and
  $(R_1',J')$ are conjugate in $\pu(2,1)$, then there exists a cube
  root of unity $\omega$ such that $\tau'=\omega\tau$, or $p=2$ and
  there is a cube root of unity $\omega$ such that
  $\tau'=\omega\bar\tau$.
\end{prop}
Beware that the groups $\S(p,\tau)$ and $\S(p',\tau')$ may well be
conjugate in $\pu(2,1)$ even when the corresponding generating pairs
$(R_1,J)$, $(R_1',J')$ are not.

It is difficult to determine the values of the parameters for which 
the group $\S(\p,\tau)$ is lattice, even though, as mentioned above,
it is likely that this implies that the pairwise product of generators
should be non-loxodromic (see~\cite{schwartzICM},~\cite{thompson}).

In particular, we search for groups such that the eigenvalues of
$R_1J$ and $R_1R_2$ are all roots of unity (recall that
$R_2=JR_1J^{-1}$). Note that
\begin{eqnarray}
  &\tr(R_1J)=\tau\\
  &\tr(R_1R_2)=u(2-|\tau|^2)+\overline{u}^2
\end{eqnarray}
Using Proposition \ref{prop:pair-cx-ref} we see that when $R_1R_2$ is
elliptic then $|\tau|=2\cos(\theta)$, or equivalently
$|\tau|^2-2=2\cos(2\theta)$, for some $\theta$.

Now we search for $\p,\tau$ such that 
\begin{eqnarray}
  &\tau=e^{i\alpha}+e^{i\beta}+e^{-i(\alpha+\beta)}\label{eq:rational-1}
  \\
  \label{eq:theta}
  &|\tau|^2-2=2\cos2\theta,
\end{eqnarray}
where $\alpha$, $\beta$ and $\theta$ are all rational multiples of
$\pi$.  This allows us to make the crucial observation that our set of
equations is in fact equivalent to one that does not involve $\p$.  In
other words, we need only find the values of $\tau$ such that there
exist $\alpha$, $\beta$ and $\theta$ rational multiples of $\pi$
satisfying~\eqref{eq:rational-1} and~\eqref{eq:theta}. 
For each such value of $\tau$, any value of $\p\geq 2$ gives a group
preserving a Hermitian form, but the signature of this form depends
on $p$ and $\tau$. We are interested in the case where this signature is
$(2,1)$.

Eliminating $\tau$ from~\eqref{eq:rational-1} and~\eqref{eq:theta} yields
\begin{equation}
  \cos(2\theta)-\cos(\alpha-\beta)-\cos(2\alpha+\beta)-\cos(\alpha+2\beta)=\frac{1}{2},\label{eq:conwayjones}
\end{equation}
so the question is now reduced to a problem about finding all possible
sets of rational multiples of $\pi$ that satisfy the rational
relation~\eqref{eq:conwayjones}; as explained
in~\cite{parkerunfaithful}, this problem was stated and solved by
Conway and Jones (see Theorem~7 of~\cite{conwayjones}).

Note that $\tau$ determines the angles, so we can list the solutions
only by giving the values of $\tau$. Moreover, if $\tau$ corresponds
to a solution, then clearly so do $\omega\tau$ and
$\overline{\omega}\tau$, where $\omega=(-1+i\sqrt{3})/2$ is a
primitive cube root of unity; in terms of our geometric motivaion,
this corresponds to multiplying the group by a scalar matrix of order
3. Also, if $\tau$ is a solution, then so is $\overline{\tau}$, so in
the list below we will only list one representative for complex
conjugate pairs, and avoid repetitions coming from multiplying a given
trace $\tau$ by a cube root of unity.

Because of the fact that there are many solutions, Conway and Jones
only list them up to obvious symmetry in the angles. As a consequence,
the application of~\cite{conwayjones} in this context requires a lot of
bookkeeping, and it is quite difficult to achieve it by hand.

It turns out there are two continuous families of solutions, given by
\begin{equation}\label{eq:taumostow}
  \tau=-e^{i\phi/3}, \ {\rm and}
\end{equation}
\begin{equation}\label{eq:tausauter}
  \tau=e^{i\phi/6}\cdot 2\cos(\phi/2).
\end{equation}
These are of course only seemingly continuous, since we are only
interested in solutions where $\phi$ is a rational multiple of $\pi$.

As mentioned in~\cite{parkerunfaithful}, the first family corresponds
to Mostow groups, whereas the second family corresponds to certain
subgroups of Mostow groups (note that some values of $\tau$ lie in
both families). We refer to the corresponding (parametrized) curves in
the complex plane as the Mostow curve and the Sauter curve,
respectively.

For groups with $\tau$ on the Mostow or Sauter curves, the list of
lattices can be deduced from work of Deligne-Mostow
(see~\cite{mostowdiscontinuous},~\cite{parkerunfaithful}
and~\cite{parkerpaupert}). In order to refer to these groups, we will
use the same notation as Mostow, namely 
$$
\Gamma(\p,t)
$$ 
denotes the group generated by reflections of order $\p$ and
phase-shift $t\in\Q$. This group can also be described as
$\S(\p,\tau)$ where $\tau=e^{\pi
  i(\frac{3}{2}+\frac{1}{3p}-\frac{t}{3})}$.

There are also a finite number of solutions that lie neither on the
Mostow curve nor on the Sauter curve, which are given in Table~\ref{tab:tauvalues}.
{\scriptsize
\begin{table}[htbp]
\begin{eqnarray*}
  & \sigma_1 = -1+i\sqrt{2};\\
  & \sigma_2 = -1+i(\sqrt{5}+1)/2;
      \quad \sigma_3 = -1+i(\sqrt{5}-1)/2\\
  & \sigma_4 = (-1+i\sqrt{7})/2\\
  & \sigma_5 = e^{-\pi i/9}(-\bar\omega-(1-\sqrt{5})/2);
      \quad \sigma_6 = e^{-\pi i/9}(-\bar\omega-(1+\sqrt{5})/2)\\
  & \sigma_7 = -e^{-\pi i/9}(\bar\omega+2\cos\frac{2\pi}{7});
   \quad \sigma_8 = -e^{-\pi i/9}(\bar\omega+2\cos\frac{4\pi}{7});
   \quad \sigma_9 = -e^{-\pi i/9}(\bar\omega+2\cos\frac{6\pi}{7})\\
  & \sigma_{10} = (1+\sqrt{5})/2
  \quad \sigma_{11} = (1-\sqrt{5})/2
\end{eqnarray*}
\caption{The list of isolated values of $\tau$ that give $R_1J$ and
  $R_1R_2$ of finite order (or possibly parabolic). The list is given
  only up to complex conjugation, and up to multiplication by a cube
  root of unity.}\label{tab:tauvalues}
\end{table}
}
Note that the last two values were missing
in~\cite{parkerunfaithful}, but this has essentially no bearing on the
results in~\cite{dpp2}, since the corresponding lattices turn out to
be arithmetic (see the commensurability invariants given in the
appendix).  

Groups with $\tau=\tr(R_1J)$ in Table~\ref{tab:tauvalues} are called {\bf sporadic
  triangle groups}. A conjectural list of sporadic triangle groups
that are lattices was given in~\cite{dpp}, and a significant part of
that conjecture was proved in~\cite{dpp2}.
The goal of the present paper is to extend the methods of~\cite{dpp2}
to a wider class of groups. For one thing, the general method should
make some of the ad hoc constructions in~\cite{dpp2} more transparent.
In particular, we complete the proof of the conjectures from~\cite{dpp}.

For the other, we exhibit a larger number of lattices, some of them
giving new non-arithmetic commensurability classes of lattices (some
are not commensurable to any Deligne-Mostow/Thurston groups, nor to
any sporadic triangle group).
\begin{itemize}
\item We prove that all 12 groups mentioned in~\cite{dpp2} are indeed
  lattices (the proof given there covered six out of the twelve), as
  well as the four extra sporadic groups.
\item We propose an extension of the construction to some
  non-equi\-lateral lattices, and handle the groups that come out of the
  analysis in James Thompson's thesis~\cite{thompson}.
\end{itemize}

For concreteness, in Table~\ref{tab:sporadicLattices} we list the
relevant values of the order $p$ of complex reflections, for sporadic
families of groups that do indeed contain lattices.
\begin{table}
  \begin{tabular}{|c|c|}
    \hline
    $\tau$                   & Lattice for $p=$\\
    \hline
    $\sigma_1$               & 3,4,6\\
    $\overline{\sigma}_4$    & 3,4,5,6,8,12\\
    $\sigma_5$               & 2,3,4\\
    $\sigma_{10}$            & 3,4,5,10\\
    \hline
  \end{tabular}
\caption{Values of $p,\tau$ such that $\S(p,\tau)$ are lattices.}\label{tab:sporadicLattices}
\end{table}

\subsection{Non equilateral triangle groups}\label{sec:noneq}

In this section, we describe the groups that come from
Thompson's thesis, since they do not appear anywhere in the
literature (in~\cite{thompson} and \cite{kamiyaparkerthompson}
mainly involutive generators were considered). 

The non equilateral triangle groups that appear in this paper will be
parametrized by a triple of complex numbers, denoted by 
${\bf  T}=(\rho,\sigma,\tau)$. These three complex numbers generalize $\tau$
  in the sense that when the triangle is equilateral they are all equal to the
  parameter $\tau$ given above. 
  As before, we assume the three generators
rotate by the same angle $2\pi/\p$, and denote
$u=e^{2i\pi/3\p}$. Then
$$
\rho=(u^2-\bar{u})\frac{\langle{\bf n}_2,{\bf n}_1\rangle}
{\Vert{\bf n}_2\Vert\,\Vert{\bf n}_1\Vert},\quad
\sigma=(u^2-\bar{u})\frac{\langle{\bf n}_3,{\bf n}_2\rangle}
{\Vert{\bf n}_3\Vert\,\Vert{\bf n}_2\Vert},\quad
\tau=(u^2-\bar{u})\frac{\langle{\bf n}_1,{\bf n}_3\rangle}
{\Vert{\bf n}_1\Vert\,\Vert{\bf n}_3\Vert}.
$$

We denote the corresponding group by
$\T(\p,{\bf T})$. Its generators are given by { \small
\begin{equation*}
  R_1=\left(\begin{matrix}
  u^2 & \r & -u\overline{\t}\\
  0 & \bar u & 0\\
  0 & 0 & \bar u
  \end{matrix}\right);\ 
  R_2=\left(\begin{matrix}
  \bar u & 0 & 0\\
  -u\bar \r & u^2 & \s\\
  0 & 0 & \bar u
\end{matrix}\right);\ 
  R_3=\left(\begin{matrix}
  \bar u & 0 & 0\\
  0 & \bar u & 0\\
  \t & -u\bar \s & u^2
\end{matrix}\right)
\end{equation*}
}
which preserve the Hermitian form
$$
H=\left(\begin{matrix}
  \alpha & \beta_1 & \overline{\beta}_3\\
  \overline{\beta}_1&\alpha&\beta_2\\
  \beta_3&\overline{\beta}_2&\alpha
\end{matrix}\right)
$$ 
where $\alpha=2-u^3-\bar u^3$, $\beta_1=(\bar u^2-u)\r$,
$\beta_2=(\bar u^2-u)\s$, $\beta_3=(\bar u^2-u)\t$. Note that
putting $u=-1$ gives the formulae in Section 2.3 of
\cite{thompson} except that $H$ is multiplied by $2$.

The triple $(R_1,R_2,R_3)$ is determined up to conjugacy by 
$|\rho|$, $|\sigma|$, $|\tau|$ and $\arg(\rho\sigma\tau)$; 
see~\cite{pratoussevitch} and~\cite{parkertraces}:

\begin{prop}
  For $j=1,\,2,\,3$, let $R_j$, $R_j'$ be complex reflections of 
  angle $2\pi/p$ in $\su(2,1)$. Let $(\rho,\sigma,\tau)$ and 
  $(\rho',\sigma,\tau')$ be defined as above. If the triples 
  $(R_1,R_2,R_3)$ and $(R_1',R_2',R_3')$ are conjugate in 
  $\pu(2,1)$, then 
  $$
  |\rho'|=|\rho|,\quad |\sigma'|=|\sigma|,\quad |\tau'|=|\tau|,\quad
  \arg(\rho'\sigma'\tau')=\arg(\rho\sigma\tau)
  $$
  or $p=2$ and 
  $$
  |\rho'|=|\rho|,\quad |\sigma'|=|\sigma|,\quad |\tau'|=|\tau|,\quad
  \arg(\rho'\sigma'\tau')=-\arg(\rho\sigma\tau).
  $$
\end{prop}

Even though the triangle is not equilateral, we take complex
reflections that rotate by the same angle, and an important
consequence of this is that the condition corresponding to the
requirement that short words ($123$, $123\bar2$, etc) be
non-loxodromic turns out to be independent of that angle.

In particular, in order to determine the relevant values of
$(\r,\s,\t)$, one can restrict to considering groups generated by
reflections of order 2. In that case, the triangle is determined by
its angles together with a Cartan angular invariant
(see~\cite{schwartzICM} or~\cite{thompson}), and it has become
customary to label this triangle according to the orders of
$23$, $31$, $12$ and $1\bar{3}23$ and Schwartz uses 
$p$, $q$, $r$, $n$ for these orders respectively. Because of the 
conflict of this notation with the order $\p$ of the complex reflections, 
we choose to write $(a,b,c;d)$ instead. Specifically, we write 
$(a,b,c;d)$ for the group generated by complex reflections in a triangle with 
angles $\pi/a$, $\pi/b$, $\pi/c$ such that the element corresponding 
to $1\bar323$ has order $d$ (more specifically the triangle with 
sides the mirrors of $R_1$, $R_3$, $R_3^{-1}R_2R_3$ has 
angles $\pi/a$, $\pi/b$, $\pi/d$). 

Note that the above discussion makes sense only when $a,b,c\geq 3$,
since the $(2,b,c)$ triangle groups are rigid in $\pu(2,1)$. In fact,
some of these rigid groups turn out to produce lattices as well, when
replacing involutions by reflections of order larger than 2; we will
come back to this below (see Table~\ref{tab:Trigid} for instance).

The traces of the relevant products of reflections are 
\begin{eqnarray*}
  \tr(R_1R_2)&=&u(2-|\r|^2)+\overline{u}^2, \\
  \tr(R_2R_3)&=&u(2-|\s|^2)+\overline{u}^2, \\
  \tr(R_3R_1)&=&u(2-|\tau|^2)+\overline{u}^2, \\
  \tr(R_1R_3^{-1}R_2R_3)&=&u(2-|\s\t-\bar\r|^2)+\overline{u}^2, \\
  \tr(R_1R_2R_3)&=&3-|\r|^2-|\s|^2-|\t|^2+\r\s\t. 
\end{eqnarray*}
Therefore, the analogues of equation \eqref{eq:theta} are:
\begin{eqnarray*}
|\r|^2-2 & = & 2\cos(2\pi/c), \\
|\s|^2-2 & = & 2\cos(2\pi/a), \\
|\t|^2-2 & = & 2\cos(2\pi/b), \\
|\s\t-\bar\r|^2 -2& = & 2\cos(2\pi/d).
\end{eqnarray*}

The analogue of the equation~\eqref{eq:conwayjones} in this context
turns out to be much harder to solve (it involves a sum of eight
cosines rather than four). Rather than solving that equation, Thompson
used a computer search to list $(a,b,c;d)$ triangles (still with
$a,b,c\geq 3$) such that the short words mentioned above are all
elliptic, assuming that $a$, $b$, $c$ and $d$ are no larger than 2000.

The corresponding groups are listed in Table~\ref{tab:Tvalues} in
terms of $\r$, $\s$, $\t$. Note that the presence of the symmetries
described in~\cite{kamiyaparkerthompson} allows us to assume that
$a\leq b\leq c\leq d$.
\begin{table}[htbp]
$$
\begin{array}{|l|rrrr|c|ccc|c|}
\hline
& a & b & c & d & o(123) & \rho & \sigma & \tau & \textrm{Lattice for }p=\\
\hline
{\bf S}_1 & 3 & 3 & 4 & 4 & 7 & \frac{1+i\sqrt{7}}{2} & 1 & 1                         & 3,4,5,6,8,12\\
{\bf S}_2 & 3 & 3 & 4 & 5 & 5 & 1+\omega\frac{1+\sqrt{5}}{2} & 1 & 1                  & 3,4,5\\
{\bf E}_1 & 3 & 3 & 4 & 6 & 8 & i\sqrt{2} & 1 & 1                                     & 3,4,6\\
{\bf E}_2 & 3 & 4 & 4 & 4 & 6 & \sqrt{2} & -\bar{\omega} & \sqrt{2}                   & 3,4,6,12\\
{\bf H}_1 & 3 & 3 & 4 & 7 & 42 & \frac{-1+i\sqrt{7}}{2} & e^{-4i\pi/7} & e^{-4i\pi/7} & 2,-7\\
{\bf H}_2 & 3 & 3 & 5 & 5 & 15 & -1-e^{-2i\pi/5} & e^{4i\pi/5} & e^{4i\pi/5}          & 2,3,5,10,-5\\
\hline
\end{array}
$$
\caption{Thompson's list of parameters (up to complex conjugation). In
  the table $\omega$ denotes $(-1+i\sqrt{3})/2$. Negative values of
  $p$ can also be replaced by their absolute value $|p|$, provided we
  take the complex conjugate value of the corresponding parameter
  ${\bf T}$, since $\T(p,{\bf T})=\T(-p,{\bf \bar
    T})$.}\label{tab:Tvalues}
\end{table}

In Table~\ref{tab:Trigid}, we list the corresponding groups coming
from rigid triangle groups (these were not considered
in~\cite{thompson}, but they produce lattices as well).
\begin{table}[htbp]
$$
\begin{array}{|l|rrrr|c|ccc|c|}
\hline
& a & b & c & d & o(123) & \rho & \sigma & \tau                   & \textrm{Lattice for }p=\\
\hline
{\bf S}_3 & 2 & 3 & 3 & 3 & 4      & 1 & 0 & 1                    & 5,6,7,8,9,10,12,18\\
{\bf S}_4 & 2 & 3 & 4 & 4 & 3      & \sqrt{2} & 0 & 1             & 4,5,6,8,12\\
{\bf S}_5 & 2 & 3 & 5 & 5 & 5      & \frac{1+\sqrt{5}}{2} & 0 & 1 & 3,4,5,10\\
{\bf E}_3 & 2 & 3 & 6 & 6 & \infty & \sqrt{3} & 0 & 1             & 3,4,6\\
\hline
\end{array}
$$
\caption{Non-equilateral triangle groups, coming from rigid triangle groups.}\label{tab:Trigid}
\end{table}

\begin{rk}\label{rk:char}
\begin{enumerate}
\item If ${\bf T}=(\rho,\sigma,\tau)$ is a valid parameter, then so is
  ${\bf \bar T}=(\bar\rho,\bar\sigma,\bar\tau)$, and the braid lengths
  corresponding to control words are the same for both groups (but for
  the same value of $\p>2$, the groups are usually not conjugate in
  $\pu(2,1)$). In the table, we list only one representative for each
  complex conjugate pair.
\item It follows from the above discussion that among triangle
  groups with non-loxodromic control words (of orders less than 2000),
  the ones in Table~\ref{tab:Tvalues} are characterized up to complex
  conjugation by 
  \begin{itemize}
    \item the order of $R_1R_2R_3$ and
    \item the braid lengths of pairs of reflections corresponding to
      control words (provided the braid relation is not too large,
      i.e. at most 2000), namely the braid length of the pairs
      $(R_j,R_k)$, $(R_1,R_2R_3R_2^{-1})$, $(R_1,R_3^{-1}R_2R_3)$,
      $(R_3,R_1R_2R_1^{-1})$.
  \end{itemize}
  \item Even though the triangles associated to these groups
  are not equilateral, for some values of ${\bf T}=(\rho,\sigma,\tau)$ the groups
  do possess extra symmetries. For example, when $a=b$ and $c=d$, that
  is the pairs $(R_1,R_2)$, 
  $(R_2R_3)$ and the pairs $(R_3,R_1)$, $(R_1,R_3^{-1}R_3)$ braid to the
  same length, as in the case of ${\bf H}_2$, we may 
  adjoin a square root of $Q=R_1R_2R_3$ conjugating $R_1$ to $R_1R_2R_1^{-1}$,
  $R_2$ to $R_3$ and $R_3$ to $R_3^{-1}R_1R_3$. A more interesting symmetry
  arises for ${\bf E}_2$. Consider the map
  $$
  S=\bar\omega^{1/3}\left(\begin{matrix}
  \omega & 0 & 0 \\ 0 & 0 & u\bar\omega \\ 0 & -\bar{u}\omega & -1
  \end{matrix}\right).
  $$
  The map $S$ has the following effect on the generators:
  $$ 
  SR_1S^{-1}=R_1,\quad SR_2S^{-1}=R_3,\quad SR_3S^{-1}=R_3^{-1}R_2R_3.
  $$
  Therefore $\langle R_1,R_2,R_3\rangle$ is a normal subgroup of
  $\langle R_1,R_2,S\rangle$. Moreover, $S$ is a complex reflection of order 3 whose 
  mirror is orthogonal to the mirror of $Q^3$. In particular, $S$ fixes $p_0$, the fixed 
  point of $Q$.
\end{enumerate}
\end{rk}

Later in the paper, we will not consider the
(non-equilateral) groups for ${\bf S_1}$ and ${\bf E_1}$, since these
are actually conjugate to sporadic (hence equilateral) triangle
groups, see section~\ref{sec:iso-nonrigid}.

We will not give much detail about the rigid Thompson groups. Indeed,
we will check that $\T(p,{\bf S_3})$ are Livn\'e lattices, $\T(p,{\bf
  S_4})$ are all isomorphic to some specific Mostow lattices,
$\T(p,{\bf S_5})$ are isomorphic to the corresponding sporadic groups
$\S(p,\sigma_{10})$ (see section~\ref{sec:iso-rigid}).

We will not consider the groups of the form $\T(p,{\bf E_3})$ either,
because of the following result.
\begin{prop}
  The lattices $\T(p,{\bf E_3})$, $p=3,4,6$ are arithmetic.
\end{prop}
\begin{pf}
  One verifies that their adjoint trace fields are $\Q$ (see
  section~\ref{sec:tracefields}), from which arithmeticity follows
  (see sections~\ref{sec:tracefields} and~\ref{sec:spectrum}).
\end{pf}

\section{Description of the algorithm} \label{sec:description}

\subsection{Combinatorial construction}

The general goal of this section is to describe the basic building
blocks of our fundamental domains, which should be bounded by
spherical shells that surround the fixed point of $P=R_1J$ (or
$Q=R_1R_2R_3$ in the non-equilateral case).

By a spherical shell, we mean that the corresponding cell complex
should be an embedded (piecewise smooth) copy of $S^3$, so that it
bounds a well-defined 4-ball. Surrounding a point means that we want
that point to be in the ball component of the complement of that copy
of $S^3$.

We will first discuss the construction on the combinatorial level, and
defer geometric realization to later in the paper
(section~\ref{sec:realization}). Both at the combinatorial and the
geometrical level, we will refer to 0-faces as {\bf vertices}, 1-faces
as {\bf edges}, 2-faces as {\bf ridges}, and 3-faces as {\bf
  sides}. In section~\ref{sec:pyramid}, we explain how sides of our
combinatorial domain are obtained from ordered triangles of complex
lines. We will then explain how to find a suitable list of triangles,
so that the corresponding pyramids form a spherical shell.

\begin{rk}
  Calling sides the 3-dimensional faces of our polytopes induces a
  slight conflict of terminology, since it is customary to talk about
  the sides of a triangle, which are one (real or complex)-dimensional
  in nature. From this point on, we will use the word ``side''
  exclusively for 3-dimensional facets, hence we will replace the word
  ``side'' by the word ``edge'' when referring to the 1-dimensional
  subsets attached to a triangle.
\end{rk}

\subsubsection{Pyramid associated to an ordered triangle} \label{sec:pyramid}

The basic building blocks for our fundamental domain will be pyramids
in bisectors. We start with a simple procedure to build a pyramid with a
given triangle as one of its faces, relying on as little geometric
information as possible.

Let us start with an ordered triangle in complex hyperbolic space,
which we will think of as encoded by its complex edges. We write the
complex edges as $\bf a$, $\bf b$, and $\bf c$, and we denote by $a$,
$b$ and $c$ the corresponding complex reflections (all of the same
rotation angle $2\pi/\p$). In a slight abuse of notation, we will
often use the same notation ${\bf d}$ for a complex line, its
extension to projective space or its polar vector.

We will call ${\bf a}$ the {\bf base} of the triangle, and we call the
intersection point between ${\bf b}$ and ${\bf c}$ the {\bf apex} of
the triangle. Note that the apex may or may not lie in complex
hyperbolic space, but this will be unimportant until we try to realize
the pyramids geometrically.

The action of $b$ and $c$ on the projective line of complex lines
through the apex is depicted in Figure~\ref{fig:triangles}.
\begin{figure}[htbp]
  \includegraphics[width=0.4\textwidth]{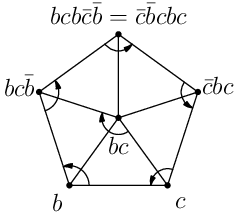}
  \caption{Triangle group picture, seen in the projective line
    through the intersection of ${\bf b}$ and ${\bf
      c}$.}\label{fig:triangles}
\end{figure}
Both $b$ and $c$ act as rotations by angle $2\pi/\p$, and their product
$bc$ acts as a rotation as well; see Proposition~\ref{prop:braiding}. 
We assume that the latter rotation
has finite order, or in other words that $b$ and $c$ satisfy a braid
relation of some finite length $n\in \N$ (we will of course assume
$n>1$). 

Note that when going around the picture in Figure~\ref{fig:triangles}
counter-clockwise, the product of any two successive rotations is
equal to the product $bc$, which gives a mnemonic device for some of
the formulas below. 

Inspired by the picture in Figure~\ref{fig:triangles}, if $b$ and $c$
braid to order $n$, the pyramid associated to the above triangle will
have an $n$-gon as its base, given by the intersection of the base of
the triangle with the mirrors of
$$
\dots,bcbc^{-1}b^{-1},\ bcb^{-1},\ b,\ c,\ c^{-1}bc,\ c^{-1}b^{-1}cbc,\dots
$$ 
Hoping that no confusion will arise, we will often use bars to
denote inverses, so the above sequence also reads
$$
\dots,bcb\bar c\bar b,\ bc\bar b,\ b,\ c,\ \bar cbc,\ \bar c\bar bcbc,\dots
$$ 
The fact that $b$ and $c$ braid to order $n$ says that the above
sequence has period $n$. The case $n=5$ is illustrated in Figure~\ref{fig:pyramid}.
In particular, the relation $bcb\bar{c}\bar{b}=\bar{c}\bar{b}cbc$ is a 
consequence of the generalized braid relation $bcbcb=cbcbc$.
\begin{figure}[htbp]
  \includegraphics[width=0.4\textwidth]{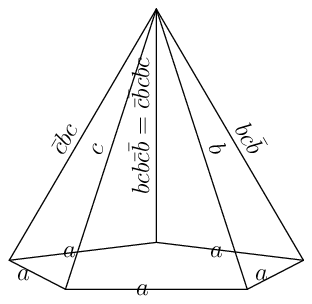}
  \caption{Pyramid with pentagonal base, corresponding to the braid
    relation $(bc)^{5/2}=(cb)^{5/2}$. The edges labelled $a$ are
    referred to as \emph{base edges}, the other ones as \emph{lateral
    edges}.}\label{fig:pyramid}
\end{figure}

Note that the sequence has the property that any successive terms in
the sequence multiply to the same product $bc$. This implies that the
pyramid would be the same (up to rotational symmetry) if we had
started with, say, $a,bc\bar b,b$ instead of $a;b,c$ (more generally
with $a;b_k,b_{k+1}$ with $b_k$, $b_{k+1}$ consecutive lateral edges
of the pyramid). Note in particular that we consider two pyramids the
same precisely when their base labels are the same isometry, and the
ordered labels of lateral edges are the same up to cyclic permutation.

\begin{rk}\label{rk:flatpyramids}
In principle, we allow a slightly degenerate kind of pyramid in the
construction, namely when $b$ and $c$ commute, the pyramid has only
two lateral edges, or equivalently two base vertices. These ``flat'' pyramids
will actually get discarded from the shell when checking that ridges
are on precisely two pyramids, see Section~\ref{sec:algo}.
\end{rk}

One can of course shift a given triangle ${\bf a},{\bf b},{\bf c}$ to
two other ordered triangles with the same orientation, namely ${\bf
  b},{\bf c},{\bf a}$ and ${\bf c},{\bf a},{\bf b}$, but these will, in
general, produce pyramids that differ combinatorially, since the pairs
$(a,b)$, $(b,c)$, $(c,a)$ need not braid with the same order. This will
be exploited in Section~\ref{sec:cycles}.

\subsubsection{Side pairing maps} \label{sec:pairing}

We now think of the pyramid $a;b,c$ associated to the triangle 
${\bf a},{\bf b},{\bf c}$, see Section~\ref{sec:pyramid}, as encoding a
side (i.e. a 3-face) of a fundamental domain for our group. In
particular, the sides should come in pairs, so there should be another
side isometric to it.

We would like to use the reflection $a$, or its inverse, as a
side-pairing map, and construct a side that has the same base as
$a;b,c$. Recall that, by construction, the base of that pyramid is the
mirror of the reflection $a$, so it is fixed by $a$. There are two
natural candidates to create an opposite face, namely those associated
to ${\bf a},{\bf ab\bar a},{\bf a c\bar a}$ and 
${\bf a},{\bf \bar ab a},{\bf \bar a ca}$. 

In order to decide which of the two triangles we choose, we will use
the fact that we want to build a spherical shell around the fixed
point of $R_1R_2R_3$.

\begin{rule123*}\label{123rule}
  We only include the pyramid corresponding to a triangle $a,b,c$
  provided either $abc$ or $cba$ is equal to $123$. If $abc=123$, then
  the corresponding side-pairing map will be $a$, and if $bca=123$,
  the side-pairing map will be $\bar a$.
\end{rule123*}

The equality $abc=123$ is to be understood in the triangle group (not
in the free group in three letters). In order to check such a
relation, it is enough to reduce the corresponding words according to
the braid relations between $a$, $b$ and $c$.

The most basic example is the initial pyramid $1;2,3$, which is paired
by $1$ to $1;12\bar1,13\bar1$ (but we do not use $1;\bar1 21,\bar1
31$).

\subsubsection{Forcing invariance} \label{sec:invariance}

We want our spherical shell to be $P$-invariant, so whenever a pyramid
from a triangle $a,b,c$ is included, we want to include all its
conjugates by either powers of $P=R_1J$ when the triangle group is
symmetric or powers of $Q=R_1R_2R_3$ in the non-symmetric case.

This is easily done using word notation, note that
$$
P1\bar P = 1J1\bar J\bar 1 = 12\bar 1;\quad
P2\bar P = 1J2\bar J\bar 1 = 13\bar 1;\quad
P3\bar P = 1J3\bar J\bar 1 = 1.
$$
and similarly
$$
\bar P1 P = 3;\quad
\bar P2 P = \bar 3 13;\quad
\bar P3 P = \bar 3 23.
$$

\subsubsection{Forcing ridge cyles} \label{sec:cycles}

The discussion in this section is related to the fact that we want our
set of pyramids to form a spherical shell. In particular, for each
pyramid, its ridges (i.e. 2-faces) should lie on precisely two
different pyramids (i.e. 3-faces) in the shell.

At least on the combinatorial level, ridges from two different
pyramids are considered the same provided they have the same
(cyclically ordered) sets of labels. If we ensured that the existence
of side-pairing maps by the selection process explained in
section~\ref{sec:pairing}, then the base ridges are on at least two
pyramids; in the sequel, we will assume they are on precisely those
two.

In fact we want \emph{all} ridges to be on precisely two pyramids of
our invariant shell. Applying this to lateral ridges gives a strong
restriction to produce the shell. If $a;b,c$ appears in the shell,
then it is natural to consider the shifted pyramid $b;c,a$ and
$c;a,b$, but only one of them well satisfy the 123-rule (see
page~\pageref{123rule}). Indeed, if $a;b,c$ has been included, then
either $abc$ or $bca$ is equal to $123$. In the first case, we need to
select $c;a,b$, in the second we select $b;c,a$. Of course, for these
to yield well-defined pyramids, we need $a$ and $b$ (or $c$ and $a$,
respectively) to braid to some finite order.

For example, we could shift the initial pyramid $1;2,3$ to either
$2;3,1$ or $3;1,2$. The first shift gets discarded, since the
corresponding products are $231$ and $312$, neither of which is
$123$. The second one is kept, since $(1\cdot 2)\cdot 3=123$ and its
side-pairing map is $\bar 3$, which maps $3;1,2$ to
$3;\bar313,\bar323$.

Another example is the pyramid $2;1,23\bar2$. We discard
$23\bar2;2,1$, but we keep $1;23\bar2,2$, whose side-pairing map is
$1$. More examples appear in section~\ref{sec:results}.

Provided we use the 123-rule and the corresponding pyramids all have
finite braiding order, all lateral ridges in the shell will lie on
precisely two pyramids in the shell.

\subsubsection{Building an invariant spherical shell}\label{sec:algo}

The previous sections suggest a procedure for building an invariant
spherical shell. We denote by $p_0$ the isolated fixed point of
$P=R_1J$ (or of $Q=R_1R_2R_3$ in the non-symmetric case).

We say a pyramid $a;b,c$ surrounds $p_0$ provided $abc=123$ or
$bca=123$. Note that if $a;b,c$ surrounds $p_0$, then so do all of its
$P$-images.

Now start with a set $\mathcal{P}$ of pyramids that all surround
$p_0$, and force its faces to be paired (see
Section~\ref{sec:pairing}), and invariant (see
Section~\ref{sec:invariance}).

Consider the ridges of pyramids of $\mathcal{P}$ that lie only on one
pyramid; then shift the corresponding triangle according to the rule
in Section~\ref{sec:cycles}; if the corresponding apex isometries
braid to finite order, enlarge $\mathcal{P}$ to contain the
corresponding shifted pyramid.
\begin{hyp} \label{hyp:shell}
  The above process never fails (i.e. apex isometries always braid to
  some finite order), and at some finite stage we get a paired
  $P$-invariant shell $\mathcal{P}$, such that every ridge is on
  precisely two pyramids of the shell.
\end{hyp}
This may seem like a lot to ask, but this hypothesis holds in many
cases, as discussed in Section~\ref{sec:results} below. In particular,
see Theorem~\ref{braid-for-apex}.

In order to obtain the condition that every ridge is on precisely two
pyramids, we need to discard all flat pyramids, i.e. those of the form
$a;b,c$ where $b$ and $c$ commute; see Remark~\ref{rk:flatpyramids}.
Such a pyramid collapses to a single triangle, which is also a lateral
ridge of pyramids of the form $b;c,a$ and $c;b,a$. Since $b$ and $c$
commute these pyramids both satisfy the 123-rule.

\subsection{Geometric realization}\label{sec:realization}

\subsubsection{Realizing vertices}

The first point is that we want to realize vertices of our pyramids in
complex hyperbolic space. Throughout this section, $a;b_1,b_2$ denotes
a given pyramid in the invariant shell, one of whose ridges is the
triangle with sides ${\bf a}$, ${\bf b_1}$ and ${\bf b_2}$. We will
denote by ${\bf b_1},\dots,{\bf b_n}$ the ordered set of lateral edges of
the pyramid.

Note that the lateral ridges of the pyramids are complex triangles,
and two complex lines in $\CH 2$ may or may not intersect in $\CH
2$. The basic idea is that the corresponding projective lines always
intersect in $\CP 2$, and that point is unique provided the
corresponding complex lines are distinct. This brings forward a
genericity assumption:
\begin{hyp} \label{hyp:irreducible}
For every side $a;b,c$ of a pyramid in $\mathcal{P}$, the mirrors of
$a$, $b$ and $c$ are in general position, by which we mean they are
pairwise distinct, and their intersection points are distinct.
\end{hyp}
If that is the case, the vertices of the pyramids have a natural
realization in $\CP 2$. Of course this is not completely satisfactory
in terms of complex hyperbolic geometry, we now explain how to realize
our shell in $\CH 2$.

Each lateral edge will contribute two or three vertices, depending on
where various projective lines intersect (inside or outside
$\CHB 2$). 

Recall that complex lines in $\CH 2$ can be described by a polar
vector $v$ in $\C^3$, in which case the complex line corresponds to
the set of negative lines in $v^\perp$. Moreover, two lines with
disctinct polar vectors $v$ and $w$ respectively meet in a unique
point in $\CP 2$ denoted by $u=v\boxtimes w$, which is inside $\CH 2$
if and only if $\langle u,u\rangle<0$. If they intersect outside
$\CHB 2$ (i.e. if $\langle u,u\rangle>0$), then they have a unique
common perpendicular complex line, which is simply the complex line
polar to $u$ (see~\cite{goldman} for details).

\noindent
\textbf{Top vertices}

\begin{itemize}
\item If ${\bf b_1}$ and ${\bf b_2}$ intersect inside 
$\CHB 2$ 
the pyramid will have a single top vertex, given by their
intersection point. 
\item If not, then the intersection point is polar to a complex line
  ${\bf d}$, which with abuse of notation we write as ${\bf d}={\bf
  b}_1\boxtimes{\bf b}_2$ (in fact ${\bf b}_k\boxtimes{\bf b}_l$ is
  actually independent of $k$ and $l$, of course with $k\neq l$). In
  the latter case, there will be $n$ top vertices, given by the
  intersection of ${\bf d}$ with the mirrors corresponding to the $n$
  lateral edges of the pyramid.
\end{itemize}

\noindent
\textbf{Base and mid vertices}

For each $k$, the $k$-th lateral edge of the pyramid will contribute
two or three vertices, depending on the position in $\CP 2$ (inside or
outside $\CHB 2$) of ${\bf d}_k={\bf a}\boxtimes {\bf b}_k$.
\begin{itemize}
\item If ${\bf d}_k$ is in $\CHB 2$, then there are only
  two vertices on ${\bf b}_k$, namely the top vertex described
  previously, and ${\bf d}_k$ which we call a {\bf bottom vertex}.
\item If ${\bf d}_k$ is outside $\CHB 2$, then it is polar
  to the common perpendicular complex line to ${\bf a}$ and 
  ${\bf b}_k$. In that case, this lateral edge will actually contribute three
  vertices, namely the top vertex, and the two feet of the common
  perpendicular complex line, which are ${\bf d}_k\boxtimes {\bf b}_k$
  and ${\bf d}_k\boxtimes{\bf a}$. The point ${\bf d}_k\boxtimes {\bf b}_k$
  will be called a {\bf mid vertex} of the pyramid, and 
  ${\bf d}_k\boxtimes{\bf a}$ a {\bf bottom vertex}.
\end{itemize}

\subsubsection{Realizing edges}

The 1-skeleton of the realization of the pyramid is obtained by
joining suitable pairs of vertices by geodesic arcs.
\begin{itemize}
\item \textbf{Top edges} The realization of a pyramid has top edges if and
only if the geometric realization of the apex of the pyramid lies
outside $\CHB 2$. 
One simply includes a geodesic arc
between ${\bf d}\cap {\bf b}_k$ and ${\bf d}\cap {\bf b}_{k+1}$, for
$k$ modulo $n$.

\item \textbf{Top to bottom vertices} If ${\bf a}\boxtimes{\bf b}_k$ is in
$\CHB 2$, then we join it either to ${\bf b}_k\boxtimes{\bf
  b}_{k+1}$ (if this point is in $\CHB 2$), or to ${\bf
  d}\boxtimes {\bf b}_k$ where ${\bf d}$ is polar to ${\bf b}_k\boxtimes {\bf
  b}_{k+1}$.

\item \textbf{Top to mid vertices} If ${\bf a}\boxtimes {\bf b}_k$ is
\emph{not} in $\CHB 2$, let ${\bf d}_k$ denote its polar
complex line. Then we join the mid vertex ${\bf d}_k\boxtimes {\bf
  b}_k$ to the top vertex ${\bf b}_k\boxtimes {\bf b}_{k+1}$ (if this
point is in $\CHB 2$), or to the top vertex ${\bf
  d}\boxtimes {\bf b}_k$ where ${\bf d}$ is polar to ${\bf
  b}_k\boxtimes {\bf b}_{k+1}$.

\item \textbf{Mid to bottom vertices} If ${\bf a}\boxtimes {\bf b}_k$ is not
in $\CHB 2$, let ${\bf d}_k$ denote its polar complex
line. Then we join the mid vertex ${\bf d}_k\boxtimes {\bf b}_k$ to
the bottom vertex ${\bf d}_k\boxtimes {\bf a}$.

\item \textbf{Bottom edges} One includes a geodesic arc between 
${\bf  a}\boxtimes {\bf b}_k$ (or ${\bf a}\boxtimes {\bf d}_k$ if the
previous point is outside $\CHB 2$) and ${\bf a}\boxtimes {\bf
  b}_{k+1}$ (or ${\bf a}\boxtimes {\bf d}_{k+1}$), for $k$ modulo $n$.
\end{itemize}

\subsubsection{Realizing ridges}

We make the following
\begin{hyp}\label{hyp:1d-embedded}
The (ordered) polygon obtained by taking the bottom edges joining
the bottom vertices
${\bf  a}\boxtimes {\bf b}_k$ or ${\bf a}\boxtimes {\bf d}_k$ 
is an embedded (piecewise smooth)
topological circle in the (closure in $\CHB 2$ of the)
complex line ${\bf a}$, equivalently this polygon bounds a disk in
that (closed) complex line. 
\end{hyp}
This allows us to define the bottom ridge.

If ${\bf d}={\bf b_1}\boxtimes{\bf b_2}$ is outside complex hyperbolic
space, then there is a similar $n$-gon in ${\bf d}$, which will be a
ridge as well, which we refer to as the \emph{top ridge}. Just as for
the bottom ridge, we assume embeddedness of the top polygon in order
to be able to define a top ridge. 

The lateral ridges are slightly more difficult to describe, since their
combinatorial type depends on the position of intersections of edges
in $\CP 2$. We list the eight possibilities for the combinatorics of a
ridge that contain lateral edges ${\bf b_k}$ and ${\bf b_{k+1}}$ in Figure~\ref{fig:ridgetypes} (we take
indices mod $n$, so when $k=n$, we have $k+1=1$).
\begin{figure}[htbp]
\includegraphics[height=2.2cm]{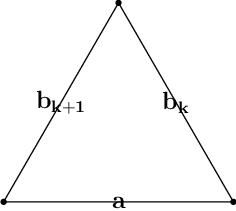}\hfill
\includegraphics[height=2.2cm]{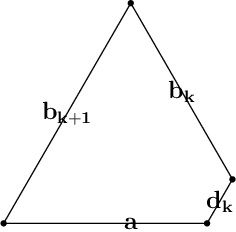}\hfill
\includegraphics[height=2.2cm]{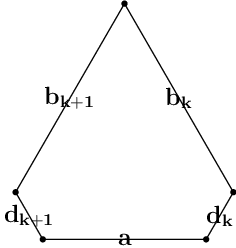}\hfill
\includegraphics[height=2.2cm]{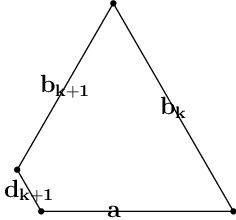}\\[0.4cm]
\includegraphics[height=2cm]{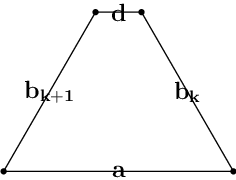}\hfill
\includegraphics[height=2cm]{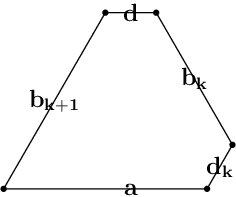}\hfill
\includegraphics[height=2cm]{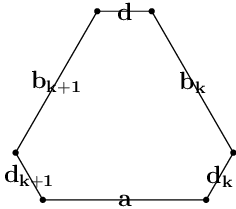}\hfill
\includegraphics[height=2cm]{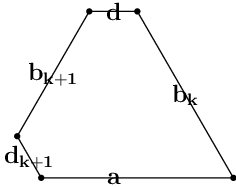}
\caption{Combinatorial types of lateral ridges.}\label{fig:ridgetypes}
\end{figure}

We will need to consider triangles with vertices outside complex
hyperbolic space, so we start by establishing some
terminology. Consider a triple ${\bf e_1},{\bf e_2},{\bf e_3}$ of
pairwise distinct complex lines (as before, ${\bf e_j}$ denotes either
a vector in $\C^3$ which is positive with respect to the Hermitian
form, or its polar complex line). The \emph{vertices} of the triangle
are the intersection points in projective space of its edges, which
are given by ${\bf v_i}={\bf e_j}\boxtimes{\bf e_k}$, where the
indices $i,j,k$ are pairwise distinct. We call ${\bf v_i}$ the vertex
opposite to the edge ${\bf e_i}$.

A \emph{complex height} of the triangle through ${\bf v_i}$ is a
complex geodesic in $\CH 2$ that is orthogonal to one of the complex
edges, and whose extension to projective space contains the opposite
vertex. If a complex height through a given vertex ${\bf v_i}$ exists,
then it is unique (in fact it is given by the complex line polar to
${\bf v_i}\boxtimes {\bf e_i}$).

If the complex height through ${\bf v_i}$ exists, we call its
intersection with the edge ${\bf e_i}$ the \emph{foot} of the complex
height. The foot of the complex height through ${\bf v_i}$ is given by
${\bf f_i}={\bf v_i}-\frac{\langle{\bf v_i},{\bf
      e_i}\rangle}{\langle{\bf e_i},{\bf e_i}\rangle}{\bf e_i}$, and
  one easily checks that the triangle has a complex height through
  ${\bf v_i}$ if and only if ${\bf f_i}$ is a negative vector.
\begin{dfn}\label{dfn:triangle}
  A \emph{complex hyperbolic triangle} is a triple of pairwise
  distinct complex lines that admits three complex heights.
\end{dfn}
From now on, all triangles are assumed to be complex hyperbolic
triangles, by which we mean the edges are pairwise distinct, and there
are three well-defined complex heights.

The basic fact that allows us to construct lateral ridges is the
following.
\begin{prop}\label{prop:ambient}
  Given a complex hyperbolic triangle ${\bf a},{\bf b},{\bf c}$
  in complex hyperbolic space, there is a unique bisector
  $\B_{\bf a}$ such that
  \begin{enumerate}
  \item ${\bf a}$ is a complex slice of $\B_{\bf a}$ and
  \item the extended real spine of $\B_{\bf a}$ contains
    ${\bf b}\boxtimes{\bf c}$.
  \end{enumerate}
\end{prop}
This result follows from the fact that a bisector is uniquely
determined by its real spine. The complex spine of the bisector
$\B_{\bf a}$ in the proposition must be orthogonal to the
base ${\bf a}$ and it must contain the vertex ${\bf b}\boxtimes{\bf
  c}$, so it must be the complex height through that vertex. Its real
spine simply joins the foot of the complex height and the
corresponding vertex.
\begin{prop}\label{prop:ridge-realization}
  The bisector $\B_{\bf a}$ from
  Proposition~\ref{prop:ambient} contains the 1-skeleton of the
  geometric realization for $a;b,c$.
\end{prop}
In order to prove this, we review the following fact, which appears as
Lemma 2.3 in~\cite{dpp}.
\begin{lem}\label{lem:orthslice}
  Let $L$ be a complex line orthogonal to a complex slice of a
  bisector $\B$. Then $L\cap\B$ is a geodesic,
  contained in a meridian of $\B$.
\end{lem}

\begin{pf} (of Prop.~\ref{prop:ridge-realization})
  This is only slightly tedious because of the diversity of cases for
  the combinatorial types of lateral ridges, see
  Figure~\ref{fig:ridgetypes}. The bottom edges are in the bisectors
  because by construction $\B_{\bf a}$ has ${\bf a}$ as one of its
  slices. The top edges, if any, are also in a slice of $\B_{\bf a}$,
  polar to the apex. The fact that the other edges are in the bisector
  follows from Lemma~\ref{lem:orthslice}.
\end{pf}

Note that the choice of the base ${\bf a}$ of the triangle is of
course artificial. In fact, as discussed in Section~\ref{sec:cycles},
there should be two sides containing a given ridge, and the other one
should be constructed by the same process, but using ${\bf b}$ or
${\bf c}$ as the base.

We call the bisectors $\B_{\bf a}$, $\B_{\bf b}$ and $\B_{\bf c}$ the
\emph{natural bisectors} associated to the triangle. We will say that
the triangle ${\bf a},{\bf b},{\bf c}$ is \emph{real} if it is the
complexification of a triangle in a copy of $\RH 2$, or equivalently
if the three corresponding polar vectors can be scaled so that their
pairwise inner products are real.

\begin{prop} \label{prop:giraud}
  Let ${\bf a},{\bf b},{\bf c}$ be a \emph{non-real} complex hyperbolic
  triangle. The natural bisectors satisfy the following properties.
  \begin{enumerate}
  \item $\B_{\bf a}\cap \B_{\bf b}= \B_{\bf
    b}\cap \B_{\bf c}= \B_{\bf c}\cap
    \B_{\bf a}$. 
  \item The above intersections have at most two connected components,
    and each component is a proper smooth disk in $\CH 2$.  
  \item The 1-skeleton of the corresponding ridge is contained in
    (the closure of) only one of the connected components.
  \end{enumerate}
\end{prop}

\begin{rk} \label{rk:real_triangles}
  Among the groups studied in this paper, only the sporadic groups with
$\tau=\sigma_{10}$ require real triangles. In fact, for
  $\S(p,\sigma_{10})$ with $p=3,4,5,10$, one checks by direct computation
  that the relevant real spines actually intersect inside $\CH 2$, so
  the natural bisectors are not cospinal either, and
  Lemma~9.1.5 of \cite{goldman} applies. The situation is in fact simpler then, as the
  intersection is connected; indeed both bisectors are then linear in
  coordinates centered on the common point of their real spines.
\end{rk}

Note also that by construction the 1-skeleton of the corresponding ridge is
embedded in a (non-totally geodesic) disk, so it bounds a piecewise
smooth disk. In other words, Proposition~\ref{prop:giraud} gives a
well-defined realization of the lateral ridges.

\medskip

\begin{pf} (of Prop.~\ref{prop:giraud})
Since the complex lines are pairwise distinct, the vectors polar to
  the edges are linearly independent, and we choose them as a basis of
  $\C^3$. Consequently, we denote by ${\bf e_i}$, $i=1,2,3$ the
  standard basis vectors of $\C^3$, and take these as polar to the
  edges of the triangle. We may assume $\langle{\bf e_j},{\bf
    e_j}\rangle=1$ for all $j=1,2,3$, and $\langle {\bf e_1},{\bf
    e_2}\rangle = a_{12}\varphi$, $\langle {\bf e_2},{\bf e_3}\rangle
  = a_{23}\varphi$, $\langle {\bf e_3},{\bf e_1}\rangle =
  a_{31}\varphi$, with $a_{jk}$ real and $|\varphi|=1$.

  We write
  $$
  \varphi^3+\bar\varphi^3=2r
  $$ for some $r\in[-1,1]$.

  Let $H$ denote the matrix of the relevant Hermitian form in the
  standard basis, which is given by
  $$
  H = \left(
  \begin{matrix}
    1                 & a_{12}\varphi     & a_{31}\bar\varphi\\
    a_{12}\bar\varphi &       1           &   a_{23}\varphi \\
    a_{31}\varphi     & a_{23}\bar\varphi &        1
  \end{matrix}
  \right).
  $$ 

  Since the Hermitian form must have signature $(2,1)$, writing
  $d=\det H$, we must have
  \begin{equation}\label{eq:d}
    d=2ra_{12}a_{23}a_{31}-a_{12}^2-a_{23}^2-a_{31}^2+1<0,
  \end{equation}
  which we assume in what follows.

  As above, we denote by ${\bf e_i}$ the vectors polar to the edges, by
    ${\bf v_i}$ the vertex opposite to ${\bf e_i}$, by ${\bf f_i}$ the
    corresponding foot of the complex height, and by $\B_i$ the
    corresponding natural bisector.

  One computes that the vertices are given by
  {\small
  $$
  {\bf v_1}=
  \left(
  \begin{matrix}
    1-a_{23}^2\\
    a_{23}a_{31}\varphi^2-a_{12}\bar\varphi\\
    a_{12}a_{23}\bar\varphi^2-a_{31}\varphi\\
  \end{matrix}
  \right),
  \quad
  {\bf v_2}=
  \left(
  \begin{matrix}
    a_{23}a_{31}\bar\varphi^2-a_{12}\varphi\\
    1-a_{31}^2\\
    a_{12}a_{31}\varphi^2-a_{23}\bar\varphi\\
  \end{matrix}
  \right),
  \quad
  {\bf v_3}=
  \left(
  \begin{matrix}
    a_{12}a_{23}\varphi^2-a_{31}\bar\varphi\\
    a_{12}a_{31}\bar\varphi^2-a_{23}\varphi\\
    1-a_{12}^2
  \end{matrix}
  \right).
  $$
  }
  and the feet of the complex heights are given by
  {\tiny
  $$
  {\bf f_1}=
  \left(
  \begin{matrix}
     -2 r a_{12}a_{23}a_{31} + a_{12}^2+a_{31}^2\\
     a_{23}a_{31}\varphi^2-a_{12}\bar\varphi\\
     a_{12}a_{23}\bar\varphi^2-a_{31}\bar\varphi\\
  \end{matrix}
  \right),
  \quad
  {\bf f_2}=
  \left(
  \begin{matrix}
     a_{23}a_{31}\bar\varphi^2-a_{12}\varphi\\
     -2 r a_{12}a_{23}a_{31} + a_{12}^2+a_{23}^2\\
    a_{12}a_{31}\varphi^2-a_{23}\bar\varphi\\
  \end{matrix}
  \right),
  \quad
  {\bf f_3}=
  \left(
  \begin{matrix}
    a_{12}a_{23}\varphi^2-a_{31}\bar\varphi\\
    a_{12}a_{31}\bar\varphi^2-a_{23}\varphi\\
     -2 r a_{12}a_{23}a_{31} + a_{23}^2+a_{31}^2\\    
  \end{matrix}
  \right).
  $$
  }

  The condition that the complex heights be well-defined translates
  into the following inequalities:
  \begin{equation}\label{eq:key_inequalities}
  \begin{array}{c}
    2 r a_{12}a_{23}a_{31} - a_{12}^2 - a_{23}^2 < 0\\
    2 r a_{12}a_{23}a_{31} - a_{23}^2 - a_{31}^2 < 0\\
    2 r a_{12}a_{23}a_{31} - a_{31}^2 - a_{12}^2 < 0
  \end{array}
  \end{equation}

  We denote by ${\bf s_j}$ a vector polar to the complex spine of
  $\B_j$ (which is also the complex height through ${\bf v_j}$).
  We have
  $$
  {\bf s_1} = 
  \left(
  \begin{matrix}
    -a_{23}\bar\varphi\\
    a_{31}\varphi\\
    0
  \end{matrix}
  \right),
  \quad
  {\bf s_2} = 
  \left(
  \begin{matrix}
    0\\
    -a_{31}\bar\varphi\\
    a_{12}\varphi
  \end{matrix}
  \right),
  \quad
  {\bf s_3} = 
  \left(
  \begin{matrix}
    a_{23}\varphi\\
    0\\
    -a_{12}\bar\varphi
  \end{matrix}
  \right).
  $$

  Since the real spine of $\B_j$ contains ${\bf v_j}$ and ${\bf f_j}$,
  and these two vectors are not orthogonal, the vector ${\bf
    z}=\left(\begin{matrix} z_1\\z_2\\z_3
  \end{matrix}\right)$ is on $\B_j$ if and only if the triple Hermitian inner product 
  $\langle {\bf z}, {\bf v_j}, {\bf f_j}\rangle = \langle {\bf z},
  {\bf v_j}\rangle \langle {\bf v_j}, {\bf f_j}\rangle \langle {\bf
    f_j}, {\bf z}\rangle$ is real.

  The natural bisectors can then be described by the following equations
  $$
  \begin{array}{c}
  \B_1: \Im((a_{31}v_1\bar v_3 - a_{12}v_2\bar v_1)\varphi)=0,\\
  \B_2: \Im((a_{12}v_2\bar v_1 - a_{23}v_3\bar v_2)\varphi)=0,\\
  \B_3: \Im((a_{23}v_3\bar v_2 - a_{31}v_1\bar v_3)\varphi)=0.
  \end{array}
  $$
  It follows that $\B_1\cap \B_2$ is contained in $\B_3$, which proves (1).

  We now prove (2), which follows from the fact that the natural
  bisectors are not cotranchal (see Lemma~9.1.5 in~\cite{goldman}).
  In order to prove non-cotranchality, we want to prove that ${\bf u}$
  is not on the real spine of $\B_1$ nor of $\B_2$. Since the
  intersection of any bisector with its complex spine is precisely its
  real spine, it is equivalent to show that ${\bf u}$ is not on $\B_1$
  nor on $\B_2$.

  Now consider the intersection in projective space of the extended
  complex spines of $\B_1$ and $\B_2$, which is represented by ${\bf u}
  = {\bf s_1}\boxtimes {\bf s_2}$. One computes
  $$
  {\bf u} = 
  \left(
  \begin{matrix}
  (a_{12}a_{23}\bar\varphi^2-a_{31}\varphi)(a_{23}a_{31}\bar\varphi^2-a_{12}\varphi) \\
  (a_{31}a_{12}\varphi^2-a_{23}\bar\varphi)(a_{23}a_{31}\varphi^2-a_{12}\bar\varphi) \\
  (a_{12}a_{23}\bar\varphi^2-a_{31}\varphi)(a_{31}a_{12}\varphi^2-a_{23}\bar\varphi)
  \end{matrix}
  \right),
  $$
  and the triple Hermitian inner products
  {\small
  \begin{equation}\label{eq:test}
  \begin{array}{c}
  \langle {\bf u}, {\bf v_1}, {\bf f_1}\rangle = 
  d(a_{12}a_{23}\bar\varphi^2-a_{31}\varphi)(a_{12}a_{31}\bar\varphi^2-a_{23}\varphi)(a_{23}a_{31}\bar\varphi^2-a_{12}\varphi)(2ra_{12}a_{23}a_{31}-a_{12}^2-a_{31}^2)^2,\\
  \langle {\bf u}, {\bf v_2}, {\bf f_2}\rangle = 
  d(a_{12}a_{23}\varphi^2-a_{31}\bar\varphi)(a_{12}a_{31}\varphi^2-a_{23}\bar\varphi)(b_{23}a_{31}\varphi^2-a_{12}\bar\varphi)(2ra_{12}b_{23}a_{31}-a_{12}^2-a_{23}^2)^2.
  \end{array}
  \end{equation}
  }
  For the meaning of $d$ in~\eqref{eq:test}, see equation~\eqref{eq:d}.
  This shows that ${\bf u}$ is on the real spine of $\B_1$ if and only
  if is is on the real spine of $\B_2$, and this happens if and only if
  $$
  (a_{12}a_{23}\bar\varphi^2-a_{31}\varphi)(a_{12}a_{31}\bar\varphi^2-a_{23}\varphi)(a_{23}a_{31}\bar\varphi^2-a_{12}\varphi)\in\R.
  $$ 
  Using the inequalities~\eqref{eq:key_inequalities}, we see that
  the last condition is equivalent to the requirement that
  $$
  a_{12}a_{23}a_{31}(\varphi^6-1)=0.
  $$ 
  This occurs if and only if the triangle is real.

  Part (3) follows from the fact that the closures of the components
  contain at most one point (see Lemma~9.1.5 in~\cite{goldman} again),
  and the fact that the vertices of the ridge are pairwise distinct.
 \end{pf}

We can now state the next assumption.
\begin{hyp}
  \label{hyp:2d-embedded}
  The 2-skeleton of the geometric realization of every side $a;b,c$ is
  embedded in the closure $\overline{\B}_{\bf a}$ in
  $\CHB 2$ of the bisector of Proposition~\ref{prop:ambient}.
\end{hyp}
This allows us to define the geometric realization of sides of the shell,
since the realization of the 2-skeleton of a side is a 2-ball, hence
it bounds a (piecewise smooth) 3-ball in the closure of the bisector.

As a final embeddedness hypothesis, we require:
\begin{hyp}
  \label{hyp:sphere}
  The 3-skeleton of the geometric realization is a manifold
  homeomorphic to $S^3$, embedded in $\CHB 2$.
\end{hyp}
In view of the solution of the Poincar\'e conjecture, checking
the ``manifold homeomeomorphic to $S^3$'' part is a completely
combinatorial check: we need to check that all links of the
corresponding cell complex are spheres, and that its fundamental group
is trivial. The embeddeding part of the assumption can be verified by
a large (but finite!) amount of computation, as explained
in~\cite{dpp2}. 

\begin{hyp}
  \label{hyp:poincare}
  The geometric realization of the invariant shell satisfies the
  hypotheses of the Poincar\'e polyhedron theorem for the cosets of
  the cyclic subgroup generated by $R_1R_2R_3$ (or $R_1J$ in the
  symmetric case).
\end{hyp}
Recall that the Poincar\'e polyhedron theorem for cosets of $H$ in
$\Gamma$ produces a polytope $\Pi$ that is a fundamental domain only
modulo the action of $H$, i.e. it produces a polytope that tiles $\CH
2$, but $\Pi$ is $H$-invariant and such that if the interior of two
images $\gamma_1 \Pi$ and $\gamma_2 \Pi$ intersect (for
$\gamma_j\in\Gamma$), then the cosets $\gamma_1H$ and $\gamma_2H$
coincide.

\subsection{The Poincar\'e polyhedron theorem}\label{sec:poincare}

For the general formulation of the Poincar\'e polyhedron theorem for
coset decompositions, see section~3.2 in~\cite{dpp2}. The only
additional difficulty, compared with the domains that appear
in~\cite{dpp2}, is that it can happen that some power of $P=R_1J$ or
$Q=R_1R_2R_3$ is a complex reflection, in which case it can stabilize
some ridges of the polyhedron. For simplicity, in the following
discussion, we use only $P$, but the same applies to $Q$ in the case
of non-symmetric triangle groups.

When implementing the Poincar\'e polyhedron we need to keep track of 
cycles of ridges, and we now recall this process. A given ridge $e_0$ is in 
exactly two sides $s_0$ and $s_1$ of the polyhedron. Suppose that the 
side pairing map corresponding to the side $s_1$ is $\gamma_1$ and that 
$\gamma_1(e_0)$ is the ridge $e_1$. Now, $e_1$ is in precisely two sides,
namely $\gamma_1(s_1)$ and a second side $s_2$. Suppose the side pairing
map associated to $s_2$ is $\gamma_2$ and the image of $e_1$ under
$\gamma_2$ is $e_2=\gamma_2(e_1)=\gamma_2\gamma_1(e_0)$. Repeating 
this process gives a sequence of sides $e_j=\gamma_j\cdots\gamma_1(e_0)$
and we call $\gamma_j\cdots\gamma_1$ the partial cycle associated
to $e_0$. We stop this process whenever $e_j$
is in the $P$-orbit of the original ridge, i.e. there exists a
$k\in\N$ such that $P^k(e_j)=e_0$, in which case the cycle
transformation is given by $P^k\gamma_j\dots\gamma_1$.

In case some ridges have non-trivial stabilizers under the action of
$\langle P\rangle$, there is some ambiguity in choosing $k$ as above,
and the rotation angles of $P^k\gamma_j\dots\gamma_1$ and
$P^l\gamma_j\dots\gamma_1$ will of course in general be different.

When this happens, we consider all possible choices of $k$, and verify
that the corresponding images of the polytope $D$ under powers of
$A=P^k\gamma_j\dots\gamma_1$ do not overlap. More precisely, if the
interiors of $D$ and $A^j(D)$ overlap, then these should be equal (but
$A^j$ need not be the identity, it may correspond to a symmetry of
$D$).

The presence of non-trivial stabilizers of ridges in the action of
$\langle P\rangle$ also has some consequences when writing explicit
presentations for our lattices in terms of generators and relations,
based on the tiling of $\CH 2$ by images of $D$.

Specifically, if a ridge $e$ has stabilizer in $\langle P\rangle$
generated by $P^k$, we need to include a presentation for the group
generated by $P^k$ and the corresponding cycle transformation $A$. For
the groups that occur in this paper, this occurs only for complex
ridges that are stabilized by a complex reflection $P^k$, and we need
to include a commutation relation 
$$
[A,P^k]=id.
$$

In the next two sections, we list some important information that can
be gathered by applying the Poincar\'e polyhedron theorem, namely
Vertex stabilizers (and more generally facet stabilizers) and
singularities of the quotient.

Two different sets of presentations for our lattices in terms of
generators and relations will be given in
sections~\ref{sec:geometric_presentations}
and~\ref{sec:natural_presentations}.

\subsubsection{Vertex stabilizers, cusps} \label{sec:vertex-stabs}

In the ``Vertex stabililizers'' tables in the appendix, for each group
where the algorithm produces a fundamental domain, we give a list of
representatives for vertices of the fundamental domain, under the
equivalence relation generated by side pairings. These are obtained by
tracking cycles of vertices, in the sense of the Poincar\'e polyhedron
theorem.

The corresponding groups are finite for vertices in $\CH 2$, and cusps
for ideal vertices. For many finite groups, the order of the group can
be obtained by Proposition~\ref{prop:braiding}. In all generality, the
full vertex stabilizers are computed from our fundamental polytopes,
by tracking orbits of vertices under side-pairing maps and the cyclic
group generated by $P$ (or $Q)$. This amounts to constructing a
directed graph whose vertices correspond to the vertices in the orbit,
with edges labelled by isometries that map the origin to the endpoint;
the stabilizer is then generated by the isometries corresponding to
generators of the fundamental group of that graph.

In particular, sometimes the stabilizer of a vertex is larger than just the group
generated by the complex reflections attached to complex faces of the
domain through that point. This happens for instance when some power of $P$ (or
$Q$) is a complex reflection with mirror through that vertex. In the
tables, we indicate this phenomenon by an asterisk.

In order to describe the finite groups, we use the Shephard-Todd
notation~\cite{shephardtodd} (or product of cyclic groups, when the
stabilizer is generated by two complex reflections with orthogonal
mirrors), i.e. $G_{k}$ denotes the $k$-th group in the Shephard-Todd
list, and the $G(m,p,n)$ are so-called \emph{imprimitive} groups, see
section~2 of~\cite{shephardtodd}.

\subsubsection{Singularities of the quotient} \label{sec:singularities}

In the appendix, we list the singular points of the quotient, for each
of the lattice $\Gamma$ where our algorithm produces a fundamental
domain $D$. The basic observation is that, by the definition of a
fundamental domain, for any element $\gamma\in\Gamma$ with a fixed
point in $\CHB 2$, there is a conjugate $\gamma'\in\Gamma$ that fixes
a point on the boundary $\partial D$ of $D$. Hence, in order to
determine conjugacy classes of fixed points of $\Gamma$, it is enough
to study stabilizers of facets of $D$.

The second basic tool used to list singular points of the quotient is
a theorem of Chevalley, according to which the quotient of $\C^2$ by a
finite subgroup of ${\rm GL}(2,\C)$ is smooth if and only if the group is
generated by complex reflections (see chapter~4 in~\cite{springer}).

Now for each facet $f$ of $D$, we determine the stabilizer $G_f$ of
$f$ (this is done computing cycles in the Poincar\'e polyhedron
theorem), and determine the reflection subgroup $R_f$, generated by
the set of complex reflections in $G_f$. The quotient has a singular
point on $f$ if and only if $R_f\subsetneq G_f$.

Note that all singularities turn out to be cyclic quotient
singularities, even though some facet stabilizers are not -- in those
cases, only the reflection subgroup of the facet stabilizer is
slightly complicated.

\section{Results} \label{sec:results}

\subsection{Good cases} 

It may seem unlikely that the above assumptions would ever be
satisfied, but in fact they turn out to be satisfied for most known
geometric constructions of lattices in $\pu(2,1)$.
The three classes of lattices we have in mind are:
\begin{itemize}
  \item Mostow/Deligne-Mostow groups;
  \item Sporadic triangle groups;
  \item Non-symmetric triangle groups from James Thompson's thesis.
\end{itemize}

We denote these three families of groups by $\Gamma(p,t)$ (where $t$
is a rational number), $\S(p,\tau)$ (where $\tau$ is a complex
number), and $\T(p,{\bf T})$ (where ${\bf T}$ is a triple of complex
numbers) respectively.

The values of $p$ such that these groups are lattices are listed, for
each value of the parameter, for sporadic groups in
Tables~\ref{tab:sporadicLattices}, and for Thompson groups
in Tables~\ref{tab:Tvalues},~\ref{tab:Trigid},
respectively.

In the next few sections, we describe the results of the algorithm for
the lattices in the above three families, in terms of general
structure and combinatorics of the shell
(sections~\ref{sec:resultsshell} and~\ref{sec:resultsembedded}) and
the verification of the hypotheses of the Poincar\'e polyhedron
theorem (section~\ref{sec:resultspoincare}). Technically, we defer
part of the proof of discreteness to calculations that are practically
impossible to perform by hand, but we have described in quite a bit of
detail the necessary computer verifications (see
section~\ref{sec:description}, and also~\cite{dpp2}).  A computer
program that implements the corresponding methods is available to
verify our claims, see~\cite{spocheck}.

\subsection{Combinatorial invariant shell} \label{sec:resultsshell}

It turns out that all these groups satisfy
Assumption~\ref{hyp:shell}, which says there exists a finite invariant shell,
at least on the combintorial level, and Assumption~\ref{hyp:irreducible},
which says all triangles are non degenerate. In fact, these assumptions are
satisfied for a much wider class of groups, discreteness is by no
means necessary at that stage.

We will describe the rough structure of the invariant shell by giving
a list of (ordered) triangles, each of which generates a side, by the
process described in Section~\ref{sec:pyramid}; we write
$$
[k]\ a;\ b,\ c
$$ 
to denote a $k$-gon pyramid with base $a$ and two consecutive
alteral edges given by $b$ and $c$. For each group, we list only side
representatives for each $P$-orbit of sides, and we only list one side
for each pair of opposite sides -- that is sides that are paired in
the sense of the Poincar\'e polyhedron theorem. The results are listed
in the ``Combinatorics'' tables in the appendix.

\subsection{Detailed combinatorics and embeddedness} \label{sec:resultsembedded}

Even though this rough description of the shell (in particular, the
number of sides) depends only on the shape parameter of the group, the
detailed combinatorics depend on the order $p$ of the generators. This
is illustrated in the pictures of the appendix, sections~\ref{app-s1}
through~\ref{app-S2}.
The pictures show the 1-skeleton of sides in geographical coordinates,
we label ideal vertices by red dots, finite vertices by blue dots, and
we label edges in word notation (at least when the label is not too
long to keep the pictures readable).

We only list side representatives, i.e. we pick a representative for
each $P$-orbit, and one for each pair of faces with inverse
side-pairing transformation.

As explained in~\cite{dpp2}, the embeddedness of the skeleton of the
domain can be reduced to a finite number of computations in the
relevant number field. Here the relevant field is the smallest field
$\ell$ such that the group can be represented as a subgroup of
$\pu(2,1,\mathcal{O}_\ell)$, see Section \ref{sec:spectrum}. 

Perhaps surprisingly, even when the triangle group is a lattice, the
invariant shell we build is not always embedded in $\CHB 2$. By
performing the computations as in~\cite{dpp2} (one way to do this is to
run our computer program~\cite{spocheck}), we obtain the following.
\begin{prop}\label{prop:embedded}
  The invariant shells of sporadic lattice triangle groups are all
  embedded in $\CHB 2$.  The invariant shells of Thompson lattice
  triangle groups are all embedded in $\CHB 2$ except $\T(5,{\bf
    \overline{H}_2})$.  The invariant shells of Mostow lattices
  $\Gamma(p,t)$ are all embedded except for those for the groups with
  parameters $(p,t)=(5,1/2)$, $(7,3/14)$ and $(9,1/18)$.
\end{prop}
For the three Mostow groups that are excluded in
Proposition~\ref{prop:embedded}, Assumption~\ref{hyp:1d-embedded}
fails.  A typical non-embedded 1-skeleton face is shown in
Figure~\ref{fig:nonembedded}.
\begin{figure}[htbp]
      \includegraphics[width = 0.4\textwidth]{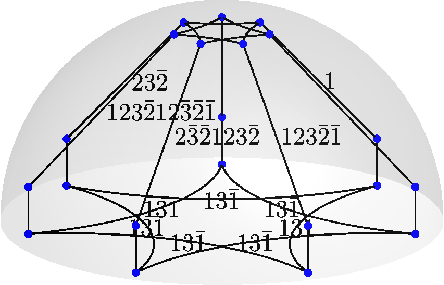}  
\caption{The 1-skeleton of pyramids is not always embedded, as here in
  the picture $\Gamma(7,3/14)$.}\label{fig:nonembedded}
\end{figure}
Note that the fact that these three groups do not fit well in the
framework of our paper is not a surprise, it can easily be explained
by Deligne-Mostow theory (see~\cite{delignemostow},~\cite{mostowihes}
or~\cite{thurstonshapes}).

In fact, the construction in our paper produces fundamental domains
for Deligne-Mostow groups corresponding to 5-tuples of weights that
satisfy condition $\Sigma$-INT with $\Sigma=S_3$ (for the relation
between our groups and Deligne-Mostow theory, see~\cite{parkersurvey}
for instance).
The Mostow groups where our polytopes are \emph{not} fundamental
polytopes are those corresponding to Deligne-Mostow groups
corresponding to 5-tuples that satisfy $\Sigma$-INT with $\Sigma=S_4$
rather than $S_3$. For such groups, one does not expect the quotient
to have the same structure, and a fundamental domain should be very
different from ours.

\subsection{Hypotheses of Poincar\'e} \label{sec:resultspoincare}

The hypotheses of the Poincar\'e polyhedron theorem can be checked as
explained in~\cite{dpp2}. This is done systematically in our computer
code~\cite{spocheck}, we summarize the result of the computations in
the statement of Proposition~\ref{prop:goodcases}. It turns out the
theorem applies in most, but not all, cases where the invariant shell
is embedded.
\begin{prop}\label{prop:goodcases}
  For all sporadic triangle groups and all Mostow groups apart from the
  $\Sigma$-INT examples with 4-fold symmetry, the hypotheses of the
  Poincar\'e polyhedron theorem are satisfied.

  The hypotheses also hold for all Thompson triangle groups except for
  $\T(12,{\bf E_2})$, $\T(7,{\bf \overline{H}_1})$,
  $\T(10,{\bf H_2})$ and $\T(5,{\bf \overline{H}_2})$.
\end{prop}

We now give some detail about how the hypotheses fail for the four
problematic Thompson groups listed in
Proposition~\ref{prop:goodcases}.
\begin{enumerate}
\item For $\T(12,{\bf E_2})$, the hypotheses of the Poincar\'e
  polyhedron fail, more specifically the local tiling near some ridges
  give extra overlap. In fact, the ridge which is the intersection of
  the pyramids $\bar212;12\bar1,3$ and $\bar313;12\bar1,3$ 
  gives a cycle transformation 
  $Q^{-1}\bar212=(12\bar13)^{-1}$, whose fourth power fixes $p_0$,
  without being a power of $Q$ (recall that $p_0$ denotes the isolated 
  fixed point of $Q$). What happens in this case is 
  that $(12\bar13)^4=(SQ)^2$ where $S$ is the extra symmetry given
  in Remark~\ref{rk:char} (3). In other words in this case
  $S=Q^2(12\bar13)^{-4}$ and so
  $\langle R_1,R_2,R_3\rangle=\langle R_1,R_2,S\rangle$. In fact, the
  algorithm does work with the spherical shell given by the algorithm, but 
  with the larger group $\langle S,Q\rangle$ as its stabilizer. One must 
  make several simple modifications, including adjoining 
  more relations when applying the Poincar\'e polyhedron theorem. Since 
  this group is arithmetic, we will not go into the details of the necessary 
  changes here.
\item For $\T(7,{\bf \bar H_1})$, the integrality condition
  fails. Indeed, the complex ridge given by the intersection of the
  two pyramids $3;1,2$ and $123\bar 2\bar 1;1,2$ has cycle
  transformation given by $12$. The isometry $R_1R_2$ is a regular
  elliptic element with angles $(3\pi/14,\pi)$, and its square is a
  complex reflection with angle $3\pi/7$, which is not of the form
  $2\pi/k$ for any $k\in\N$.
\item For $\T(10,{\bf H_2})$, the integrality condition fails. More
  specifically, the ridge on the mirror of $R_2^{-1}R_1R_2$ has cycle
  transformation given by $Q^2R_2^{-1}R_1R_2$, and this is a complex
  reflection with rotation angle $2\pi\cdot 3/10$.
\item For $\T(5,{\bf \bar H_2})$, the spherical shell is not
  embedded. In fact, in that case, the point $p_0$ (which is the
  isolated fixed point of $Q$) lies on the bottom ridge of
  $123\bar212\bar3\bar2\bar1;\,123\bar2\bar1,\,2$.
\end{enumerate}

For $\T(7,{\bf \bar H_1})$ and $\T(5,{\bf \bar H_2})$, note that the groups 
are each conjugate to a Mostow group where the algorithm runs fine,
see Proposition~\ref{prop:mostow_saves_us} (this commensurability
corresponds to a change of generators). Similarly, the group
$\T(10,{\bf H_2})$ has two alternative descriptions that allow to use
the algorithm, see Proposition~\ref{prop:isoS5_10}. 

Among Mostow groups, the problematic ones are $\Gamma(5,1/2)$,
$\Gamma(7,3/14)$, $\Gamma(9,1/18)$. 
Every one of these three groups is known to be conjugate to a Mostow
group without the extra 4-fold symmetry, by work of Sauter
(see~\cite{sauter} and also Corollary~10.18
in~\cite{delignemostowbook}). 

Specifically, we have that $\Gamma(5,1/2)$ is conjugate to the group
$\Gamma_\mu$ with $\mu=(3,3,3,3,8)/10$, which gives the same group as
$\mu=(2,3,3,3,9)/10$, which in turn gives the group $\Gamma(5,7/10)$.
Similarly, $\Gamma(7,3/14)$ corresponds to exponents
$(2,5,5,5,11)/14$, which gives the same group as $(5,5,5,5,8)/14$,
which gives the Mostow group $\Gamma(7,9/14)$. Finally,
$\Gamma(9,1/18)$ corresponds to $(7,7,7,7,8)/18$, which gives the same
group as $(2,7,7,7,13)/18$, which is $\Gamma(9,11/18)$.

\subsection{Geometric presentations} \label{sec:geometric_presentations}

In this section we list the presentations obtained from the Poincar\'e
polyhedron theorem (see section~\ref{sec:natural_presentations}) for
each of the groups where the hypotheses of the Poincar\'e poyhedron
theorem are satisfied, see section~\ref{sec:resultspoincare}.

We will call these presentations \emph{geometric presentations}, as
opposed to the \emph{natural presentations}, as described in
section~\ref{sec:natural_presentations}.
Rather than listing the actual computer output for presentations (which
is available via~\cite{spocheck}), we will list slightly modified
versions where we have used simple Tietze transformations to make the
presentation more readable.

As an example, the presentations for $\sigma_1$ groups obtained by the
computer are equivalent to the following. For $p=3$, we get
\begin{equation}\label{eq:pres_s1_3}
\begin{array}{c}
  \left\langle R_1, R_2, R_3, J \ | \ 	(R_1 J)^8, J^3, J R_1 J^{-1}R_2^{-1}, JR_2 J^{-1} R_3^{-1},\right.\\
         \br_6(R_1,R_2), \br_3(R_1, R_2 R_3 R_2 R_3^{-1} R_2^{-1}),\\
         \left.R_1^3 \right\rangle.
\end{array}
\end{equation}
For $p=4$, we get
\begin{equation}\label{eq:pres_s1_4}
\begin{array}{c}
   \left\langle R_1, R_2, R_3, J \ | \ (R_1 J)^8, J^3, J R_1 J^{-1} R_2^{-1}, J R_2 J^{-1} R_3^{-1},\right.\\
   \br_6(R_1,R_2), \br_3(R_1, R_2 R_3 R_2 R_3^{-1} R_2^{-1}),\\
   \left. R_1^4, (R_1 R_2)^{12}\right\rangle.  
\end{array}
\end{equation}
For $p=6$
\begin{equation}\label{eq:pres_s1_6}
\begin{array}{c}
   \left\langle 
   R_1, R_2, R_3, J \ | \  (R_1 J)^8, J^3, J R_1 J^{-1} R_2^{-1}, J R_2 J^{-1} R_3^{-1},\right.\\
   \br_6(R_1,R_2), \br_3(R_1, R_2 R_3 R_2 R_3^{-1} R_2^{-1}),\\
   \left. R_1^6, (R_1 R_2)^6, (R_1 R_2 R_3 R_2^{-1})^{12}\right\rangle.
\end{array}
\end{equation}
The three presentations in equations~\eqref{eq:pres_s1_3}
through~\eqref{eq:pres_s1_6} can be written in a uniform way, see the
$\sigma_1$ entry in Table~\ref{tab:geometric_presentations_1}, where by
convention the relations giving the order of $R_1 R_2$ and $R_1
R_2 R_3 R_2^{-1}$ can be omitted when the
exponents $3p/(p-3)$ or $4p/(p-4)$ are negative or infinite.

All geometric presentations are listed in the left part of
Tables~\ref{tab:geometric_presentations_1}
and~\ref{tab:geometric_presentations_2}.

\begin{table}[htbp]
{\tiny
\hspace{-1cm}
$$
\begin{array}{c}
  \begin{array}{|c|}
    \hline
    \S(\sigma_1,p), p=3,4,6\\ 
    \hline
    \begin{array}{c|c}
      \left\langle \vphantom{(R_1 R_2)^{\frac{3p}{p-3}}} R_1, R_2, R_3, J \ | \  (R_1 J)^8, J^3, J R_1 J^{-1} R_2^{-1}, J R_2 J^{-1} R_3^{-1},\right. & \left\langle \vphantom{(R_1 R_2)^{\frac{3p}{p-3}}} R_1, R_2, R_3, J \ | \ (R_1 J)^8, J^3, J R_1 J^{-1} R_2^{-1}, J R_2 J^{-1} R_3^{-1}, \right.\\
      \br_6(R_1,R_2),                                                 & \br_6(R_1,R_2), 
      \\
                       \br_3(R_1, R_2 R_3 R_2 R_3^{-1} R_2^{-1}),                                                                      & \br_4(R_1, R_3^{-1} R_2 R_3)\\ 
      \left. R_1^p, (R_1 R_2)^{\frac{3p}{p-3}}, (R_1 R_3^{-1} R_2 R_3)^{\frac{4p}{p-4}} \right\rangle           & \left.R_1^p, (R_1 R_2)^{\frac{3p}{p-3}}, (R_1 R_3^{-1} R_2 R_3)^{\frac{4p}{p-4}} \right\rangle
    \end{array} 
    \\
    \hline
  \end{array}
\\[1.5cm]
  \begin{array}{|c|}
    \hline
    \S(\bar\sigma_4,p), p=3,4,5,6,8,12\\ 
    \hline
    \begin{array}{c|c}
      \left\langle \vphantom{(R_1 R_2)^{\frac{4p}{p-4}}} R_1, R_2, R_3, J \ | \  (R_1 J)^7, J^3, J R_1 J^{-1} R_2^{-1}, J R_2 J^{-1} R_3^{-1},\right. & \left\langle \vphantom{(R_1 R_2)^{\frac{4p}{p-4}}} R_1, R_2, R_3, J \ | \  (R_1 J)^7, J^3, J R_1 J^{-1} R_2^{-1}, J R_2 J^{-1} R_3^{-1}, \right.\\
      \br_4(R_1,R_2),                                                                                           & \br_4(R_1,R_2), \\
                                                                                                                & \br_3(R_1,R_3^{-1} R_2 R_3), \\
      \left.R_1^p, (R_1 R_2)^{\frac{4p}{p-4}}, (R_1 R_3^{-1} R_2 R_3)^{\frac{6p}{p-6}} \right\rangle            & \left.R_1^p, (R_1 R_2)^{\frac{4p}{p-4}}, (R_1 R_3^{-1} R_2 R_3)^{\frac{6p}{p-6}} \right\rangle 
    \end{array}\\
    \hline
  \end{array}
\\[1.5cm]
  \begin{array}{|c|}
    \hline
    \S(\sigma_5,p), p=2,3,4\\ 
    \hline
    \begin{array}{c|c}
      \left\langle \vphantom{R_1 R_3^{-1} R_2 R_3)^{\frac{10p}{3p-10}}} R_1, R_2, R_3, J \ | \  (R_1 J)^{30}, J^3, J R_1 J^{-1} R_2^{-1}, J R_2 J^{-1} R_3^{-1},\right.   & \left\langle \vphantom{R_1 R_3^{-1} R_2 R_3)^{\frac{10p}{3p-10}}} R_1, R_2, R_3, J \ | \  (R_1 J)^{30}, J^3, J R_1 J^{-1} R_2^{-1}, J R_2 J^{-1} R_3^{-1},\right.\\
      \br_4(R_1,R_2),                                                                                                                                        & \br_4(R_1,R_2), \\ 
      \br_2((R_1J)^5, R_2 R_3^{-1} R_2^{-1} R_1 R_2 R_3 R_2^{-1}),                                                                                           & \br_5(R_1,R_3^{-1} R_2 R_3),
      \\
      R_1^p, (R_1 R_3^{-1} R_2 R_3)^{\frac{10p}{3p-10}},                                                                                                     & R_1^p, (R_1 R_3^{-1} R_2 R_3)^{\frac{10p}{3p-10}},\\ 
      \left .(R_2 R_3^{-1} R_2^{-1} R_1^{-1} R_2 R_3 R_2^{-1} R_1 R_2 R_3)^{\frac{3p}{p-3}} \right\rangle,                                                   & \left. (R_2 R_3^{-1} R_2^{-1} R_1^{-1} R_2 R_3 R_2^{-1} R_1 R_2 R_3)^{\frac{3p}{p-3}} \right\rangle,
    \end{array}\\
    \hline
  \end{array}
\\[1.5cm]
  \begin{array}{|c|}
    \hline
    \S(\sigma_{10},p),p=3,4,5,10\\ 
    \hline
    \begin{array}{c|c}
      \left\langle \vphantom{(R_1 R_2)^{\frac{10p}{3p-10}}} R_1, R_2, R_3, J \ | \  (R_1 J)^{5}, J^3, J R_1 J^{-1} R_2^{-1}, J R_2 J^{-1} R_3^{-1},\right . & \left\langle \vphantom{(R_1 R_2)^{\frac{10p}{3p-10}}} R_1, R_2, R_3, J \ | \  (R_1 J)^{5}, J^3, J R_1 J^{-1} R_2^{-1}, J R_2 J^{-1} R_3^{-1},\right. \\
      \br_5(R_1,R_2),                                                                                              & \br_5(R_1,R_2), \\ 
      \br_2(R_1, R_3^{-1} R_2^{-1} R_3 R_2 R_3),                                                                   & \br_3(R_1, R_3^{-1} R_2 R_3),\\
      \left. R_1^p, (R_1 R_2)^{\frac{10p}{3p-10}}, (R_1 R_3^{-1} R_2 R_3)^{\frac{6p}{p-6}} \right\rangle      & \left. R_1^p, (R_1 R_2)^{\frac{10p}{3p-10}}, (R_1 R_3^{-1} R_2 R_3)^{\frac{6p}{p-6}} \right\rangle
    \end{array}\\
    \hline
  \end{array}
\end{array}
$$
}
\caption{Geometric (left) and natural (right) presentations for all
  sporadic groups where the hypotheses of the Poincar\'e polyhedron
  theorem hold. Apart from inverses, relations involving negative or
  infinite exponents can be removed from the
  presentation.}\label{tab:geometric_presentations_1}
\end{table}

\begin{table}[htbp]
{\tiny
$$
\begin{array}{c}
  \begin{array}{|c|}
    \hline
    \T({\bf S_2},p), p=3,4,5\\ 
    \hline
    \begin{array}{c|c}
      \left\langle \vphantom{(R_1 R_3^{-1} R_2 R_3)^{\frac{10p}{3p-10}}} R_1, R_2, R_3 \ | \ (R_1 R_2 R_3)^5,\right.   & \left\langle  \vphantom{(R_1 R_3^{-1} R_2 R_3)^{\frac{10p}{3p-10}}} R_1, R_2, R_3 \ | \ (R_1 R_2 R_3)^5, \right.\\
      \br_3(R_2,R_3), \br_3(R_3,R_1), \br_4(R_1,R_2),                                                                  & \br_3(R_2,R_3), \br_3(R_3,R_1), \br_4(R_1,R_2),\\
                                                                                                                       &  \br_5(R_1,R_3^{-1} R_2 R_3)\\ 
      \left.R_1^p, R_2^p, R_3^p, (R_1 R_2)^{\frac{4p}{p-4}}, (R_1 R_3^{-1} R_2 R_3)^{\frac{10p}{3p-10}}\right\rangle   & \left.R_1^p, R_2^p, R_3^p, (R_1 R_2)^{\frac{4p}{p-4}}, (R_1 R_3^{-1} R_2 R_3)^{\frac{10p}{3p-10}} \right\rangle
    \end{array}\\
    \hline
  \end{array}
\\[1.5cm]
  \begin{array}{|c|}
    \hline
    \T({\bf E_2},p), p=3,4,6\\ 
    \hline
    \begin{array}{c|c}
      \left\langle  \vphantom{(R_1 R_3)^{\frac{4p}{p-4}}} R_1, R_2, R_3 \ | \ (R_1 R_2 R_3)^6,\right.                              & \left\langle \vphantom{(R_1 R_3)^{\frac{4p}{p-4}}} R_1, R_2, R_3 \ | \ (R_1 R_2 R_3)^6,\right. \\
      \br_3(R_2,R_3), \br_4(R_3,R_1), \br_4(R_1,R_2),                                                                              & \br_3(R_2,R_3), \br_4(R_3,R_1), \br_4(R_1,R_2),\\ 
      \br_2((R_1 R_2 R_3)^3, R_2^{-1} R_1 R_2),                                                                                    & \br_4(R_1 , R_3^{-1} R_2 R_3), 
      \\
      R_1^p, R_2^p, R_3^p, (R_1 R_2)^{\frac{4p}{p-4}}, (R_1 R_3)^{\frac{4p}{p-4}},                                                 & R_1^p, R_2^p, R_3^p, (R_1 R_2)^{\frac{4p}{p-4}}, (R_1 R_3)^{\frac{4p}{p-4}},\\
      \left.(R_1 R_3^{-1} R_2 R_3)^{\frac{4p}{p-4}}, (R_3 R_1 R_2 R_1^{-1})^{\frac{3p}{p-3}} \right\rangle                         & \left.(R_1 R_3^{-1} R_2 R_3)^{\frac{4p}{p-4}}, (R_3 R_1 R_2 R_1^{-1})^{\frac{3p}{p-3}} \right\rangle
    \end{array}\\
    \hline
  \end{array}
\\[1.5cm]
  \begin{array}{|c|}
    \hline
    \T({\bf H_1},p), p=2\\ 
    \hline
    \begin{array}{c|c}
      \left\langle \vphantom{(R_1 R_2)^{\frac{4p}{p-4}}} R_1, R_2, R_3 \ | \ (R_1 R_2 R_3)^{42},\right.                          &   \left\langle \vphantom{(R_1 R_2)^{\frac{4p}{p-4}}} R_1, R_2, R_3 \ | \ (R_1 R_2 R_3)^{42},\right.\\
      \br_3(R_2,R_3), \br_3(R_3,R_1), \br_4(R_1,R_2),                                                                            &   \br_3(R_2,R_3), \br_3(R_3,R_1), \br_4(R_1,R_2),\\
      \br_2( (R_1 R_2 R_3)^3, R_2^{-1}R_1R_2R_3R_2^{-1}R_1^{-1}R_2),   &  \br_7(R_1,R_3^{-1}R_2R_3), \\ 
      \left. R_1^p, R_2^p, R_3^p\right\rangle                                                                                    &   \left. R_1^p, R_2^p, R_3^p \right\rangle 
    \end{array}\\
    \hline
  \end{array}
\\[1.5cm]
  \begin{array}{|c|}
    \hline
    \T({\bf H_2},p), p=2,3,5\\ 
    \hline
    \begin{array}{c|c}
      \left\langle \vphantom{(R_1 R_2)^{\frac{10p}{3p-10}}} R_1, R_2, R_3 \ | \ (R_1 R_2 R_3)^{15},\right.                    & \left\langle \vphantom{(R_1 R_2)^{\frac{10p}{3p-10}}} R_1, R_2, R_3 \ | \ (R_1 R_2 R_3)^{15},\right.             \\
      \br_3(R_2,R_3), \br_3(R_3,R_1),                                                          & \br_3(R_2,R_3), \br_3(R_3,R_1), \br_5(R_1,R_2),      \\
      \br_5(R_1, R_3^{-1} R_2 R_3), \br_2( (R_1 R_2 R_3)^3, R_2^{-1}R_1R_2),                          & \br_5(R_1, R_3^{-1} R_2 R_3) \\ 
      R_1^p, R_2^p, R_3^p, (R_1 R_2)^{\frac{10p}{3p-10}},                                                                     & R_1^p, R_2^p, R_3^p, (R_1 R_2)^{\frac{10p}{3p-10}},\\ 
      \left.(R_1 R_3^{-1} R_2 R_3)^{\frac{10p}{3p-10}}, (R_2 R_1 R_2^{-1} R_3 R_1 R_3^{-1})^{\frac{5p}{2p-5}} \right\rangle   & \left.(R_1 R_3^{-1} R_2 R_3)^{\frac{10p}{3p-10}}, (R_2 R_1 R_2^{-1} R_3 R_1 R_3^{-1})^{\frac{5p}{2p-5}} \right\rangle 
    \end{array}\\
    \hline
  \end{array}
\\[1.5cm]
\end{array}
$$
}
\caption{Geometric (left) and natural (right) presentations for all
  Thompson groups where the hypotheses of the Poincar\'e polyhedron
  theorem hold.}\label{tab:geometric_presentations_2}
\end{table}

\subsection{Natural presentations} \label{sec:natural_presentations}

The list of braid relations that occur in geometric presentations (see
section~\ref{sec:geometric_presentations}) may seem strange to the
reader. Recall that one of our assumptions is that for every side
bounding our polytope, the complex reflections attached to every pair
of consecutive lateral edges satisfy a braid relation (see
Assumption~\ref{hyp:shell}). Not every such braid relation occurs in
the geometric presentation, however.

In this section, we give a systematic way to produce a presentation
for every lattice in our list (more precisely, lattices where the 
hypotheses of the Poincar\'e polyhedron theorem are astisfied). 
Motivated by the fact that our lattices
are essentially uniquely determined by the four braid lengths
$\br(R_1,R_2),\br(R_2,R_3),\br(R_3,R_1),\br_d(R_1,R_3^{-1}R_2R_3)$
(see~\cite{thompson},~\cite{kamiyaparkerthompson}), one naturally
expects that these braid relations should be enough to reconstruct all
others in the geometric presentations.

Accordingly, we construct \emph{natural presentations}, having the same
generators as the geometric presentations, and the following relations:
\begin{enumerate}
\item Basic relations between generators (in the equilateral case, $J$
  has order 3 and conjugates $R_j$ into $R_{j+1}$, and in the
  non-equilateral case, $Q$ is equal to $R_1R_2R_3$), and relations
  giving the order of $P=R_1J$ or $Q$;
\item Orders of complex reflections stabilizing each complex face.
  These correspond to the bases of pyramids, or top faces when they
  are truncated; see the ``Combinatorics'' tables in the appendix;
\item 
  The following braid relations corresponding to the parameters $(a,b,c;d)$:
  $$
  \br_a(R_1,R_2),\quad \br_b(R_2,R_3),\quad \br_c(R_3,R_1),\quad \br_d(R_1,R_3^{-1}R_2R_3).
  $$
\end{enumerate}
The corresponding presentations are listed in the right part of
Tables~\ref{tab:geometric_presentations_1}
and~\ref{tab:geometric_presentations_2} and~\ref{tab:geometric_presentations_2}.

Note that we could use Proposition~\ref{prop:braiding}
and~\ref{tab:geometric_presentations_2} to deduce the order of complex
reflections stabilizing the top faces in (2), but we do not know how
to deduce these orders directly from the other relations.

At this stage, it is not at all clear that the geometric and the
natural presentations for one given group are equivalent, i.e. the
corresponding finitely presented groups are isomorphic. We will prove
this in section~\ref{sec:geometric=natural}.

\subsection{Equivalence of the geometric and the natural presentations} \label{sec:geometric=natural}

The goal of this section is to prove that for every lattice where the
hypotheses of the Poincar\'e polyhedron theorem are satified, the
natural presentation is indeed a presentation for the lattice, i.e. it
is equivalent to the geometric presentation.

\begin{thm}\label{compare_presentations}
  For each lattice where the hypotheses of the Poincar\'e
  polyhedron theorem hold, the natural presentation is indeed a presentation
  of the corresponding lattice.
\end{thm}

We will prove this by a case-by-case analysis, compare the left and
right halves of Tables~\ref{tab:geometric_presentations_1}
and~\ref{tab:geometric_presentations_2} to see what we need to prove.

We start by proving that the natural braid relations (together with
basic relations between generators) imply all vertex braid relations
(we call \emph{vertex braid relations} the braid relations that appear
in Assumption~\ref{hyp:shell}).
\begin{thm}\label{braid-for-apex}
  For each lattice in our list the natural relations imply the
  appropriate braid relation at each vertex.
\end{thm}

Recall that for every triangle group lattice $\Gamma$, the geometric
and natural presentations have the same generating set, which we
denote by $X$ ($X=\{R_1,R_2,R_3,J\}$ if $\Gamma$ is a sporadic or
Mostow group, and $X=\{R_1,R_2,R_3\}$ if it is a Thompson group). We
write $G$ (resp. $N$) for the set of geometric (resp. natural)
relations. If $G\subset N$, then there is nothing to prove, since we
know $\langle X|G\rangle$ is a presentation by the Poincar\'e
polyhedron theorem, and we also know that the relations in $N$
hold. If $G\not\subset N$, we will enlarge the set of relations $N$ to
a new set $N'$ where $G\subset N'$ and $\langle X|N'\rangle$ is
isomorphic to $\langle X|N\rangle$.

It turns out that the calculations needed to prove
Theorem~\ref{compare_presentations} are essentially the same as those
needed to prove Theorem~\ref{braid-for-apex}.  In some cases, all the
vertices are $P$-images (or $Q$-images) of a vertex where we have one
of the four braid relations in the natural presentation. In this case
there is nothing to prove.  To find which extra braid relations are
required we refer to the description of the pyramids given in
Appendix~\ref{app-s5}.

We now summarize what we need to prove in each case and where we do so:
  \begin{enumerate}
  \item $\sigma_1$: for both Theorem~\ref{compare_presentations} and Theorem~\ref{braid-for-apex} 
  we need to show that the natural presentation implies $\br_3(R_1,R_2R_3R_2R_3^{-1}R_2^{-1})$
  and for Theorem~\ref{braid-for-apex} we need to show $\br_3(R_1,R_3^{-1}R_2^{-1}R_2R_3R_2)$ as well;
  see Proposition~\ref{prop-braid-s1}.
  \item $\bar\sigma_4$:  we have $G\subset N$ and there is nothing to prove for either theorem.
  \item $\sigma_5$: we show the natural presentation implies
  $\br_2\bigl((R_1J)^5,R_2R_3^{-1}R_2^{-1}R_1R_2R_3R_2^{-1}\bigr)$ for Theorem~\ref{compare_presentations} and 
  $\br_6(R_2,R_3^{-1}R_2^{-1}R_1^{-1}R_2R_3R_2^{-1}R_1R_2R_3)$ for Theorem~\ref{braid-for-apex};
  see Corollary~\ref{cor-extra-braiding} (1).
  \item $\sigma_{10}$:  we need to show that the natural presentation implies 
  $\br_2(R_1,R_2R_3R_2R_3^{-1}R_2^{-1})$ for 
  Theorem~\ref{compare_presentations}; see Proposition~\ref{prop-braid-s10}.
  There is nothing to prove for Theorem~\ref{braid-for-apex}.
  \item ${\bf S}_2$: we have $G\subset N$ and there is nothing to prove for either theorem.
  \item ${\bf E}_2$: we need to show that the natural presentation implies 
  $\br_2\bigl((R_1R_2R_3)^2,R_2^{-1}R_1R_2\bigr)$
   for Theorem~\ref{compare_presentations} and $\br_6(R_1R_2R_1^{-1}, R_3)$;
  for Theorem~\ref{braid-for-apex}; see Corollary~\ref{cor-extra-braiding} (2).
  \item ${\bf H}_1$: we show the natural presentation implies 
  $\br_2\bigl((R_1R_2R_3)^3, R_2^{-1}R_1R_2R_3R_2^{-1}R_1^{-1}R_2\bigr)$
  for Theorem~\ref{compare_presentations} and   $\br_{14}(R_2^{-1}R_1R_2,R_3^{-1}R_1R_2R_1^{-1}R_3)$
  for Theorem~\ref{braid-for-apex}; see Corollary~\ref{cor-extra-braiding} (3).
  \item ${\bf H}_ 2$: we need to show that the natural presentation implies 
  $\br_2\bigl((R_1R_2R_3)^3, R_2^{-1}R_1R_2\bigr)$ 
   for Theorem~\ref{compare_presentations} and   $\br_{10}(R_1R_2R_3R_2^{-1}R_1^{-1},R_2)$
   for Theorem~\ref{braid-for-apex}; see Corollary~\ref{cor-extra-braiding} (4).
    \end{enumerate}

Observe that in many cases we need to prove a braid relation of length 2 and a braid relation of length $n$.
It turns out that these are both consequences of a braid relation of length 4, as shown in the following
general lemma.

\begin{lem}\label{lem-pre-braid}
  Let $G$ be a group and suppose $A,B\in G$ satisfy $\br_4(A,B)$.
  \begin{enumerate}
  \item If $C\in G$ is any element that commutes with $B$ then
    $\br_2(B,\,CBA^{-1}B^{-1}A^{-1})$.
  \item If $B^n=Id$ then  $\br_n(A,BAB^{-1})$.
  \end{enumerate}
\end{lem}

\begin{proof}
We first prove (1). The condition $\br_4(A,B)$ implies $B$ commutes with $(BA)^2$. We can write
$$
CBA^{-1}B^{-1}A^{-1}=CB(BA)^{-2}B.
$$
Since $B$ commutes with $C$, $B$ and $(BA)^2$ we see that it commutes with $CBA^{-1}B^{-1}A^{-1}$
as required.

  For part (2), observe that $(BA)^2$ lies in the center of $\langle
  A,B\rangle$. Note that for every integer $m$, we have
  \begin{eqnarray*}
    (ABAB^{-1})^{m}(BAB^{-1}A)^{-m}
    & = & (ABAB^{-1})^{m}A^{-1}(ABAB^{-1})^{-m}A \\
    & = & \bigl(B^{-1}(BA)^2B^{-1})^{m}A^{-1}(B^{-1}(BA)^2B^{-1})^{-m}A \\
    & = & B^{-2m}A^{-1}B^{2m}A.
  \end{eqnarray*}
  We used the fact that $(BA)^2$ commutes with $A$ and $B$ on the last
  line.  If $n$ is even then setting $m=n/2$ immediately gives the
  result. If $n$ is odd then setting $m=(n-1)/2$ we have
  \begin{eqnarray*}
    BAB^{-1}(ABAB^{-1})^{(n-1)/2}(BAB^{-1}A)^{-(n-1)/2}A^{-1}
    & = &  BAB^{-1}(B^{-(n-1)}A^{-1}B^{n-1}A)A^{-1} \\
    & = & BAB^{-n}A^{-1}B^{n-1}.
  \end{eqnarray*}
  Since $B$ has order $n$ the last line is the identity, which
  completes the proof.
\end{proof}

\medskip

Lemma~\ref{lem-pre-braid} has the following geometric
interpretation. In all the cases where we use it, the map $A$ will be
a complex reflection of angle $2\pi/p$, in which case $\langle
A,B\rangle$ is a central extension of a $(2,p,n)$ triangle group, see
Lemma~\ref{prop:braiding}.  The apex of a pyramid with $A$ and
$BAB^{-1}$ as lateral edges is also fixed by $B$ and there is an
$n$-gon in the projective line of complex lines through the apex (see
section~\ref{sec:pyramid}).
The top ridge will be an $n$-gon whose center is fixed by $B$ and
whose vertices are the intersection points of this complex line with
the mirrors of $B^jAB^{-j}$. An example is the top hexagon in the
lower right hand pyramid of the pictures in Appendix~\ref{app-s5}.

We now describe how to apply this lemma to obtain extra braid
relations for some of the families of groups. We show that some power
of $P$, $Q$ or another symmetry braids with length 4 with one of the
complex reflections.

\begin{lem}\label{lem-length4-braiding}
  \begin{enumerate}
  \item The natural presentation for $\sigma_5$ implies $\br_4(R_2,P^{-5})$.
  \item Let $S$ satisfy $SR_1S^{-1}=R_1$, $SR_2S^{-1}=R_3$ and
    $SR_3S^{-1}=R_3^{-1}R_2R_3$ (as the symmetry $S$ of ${\bf E}_2$
    groups, see Remark~\ref{rk:char}). Then $\br_4(R_1,R_2)$ implies
    $\br_4(R_3,R_3^{-1}R_2^{-1}R_1^{-1}S^{-1})$.
  \item The natural presentation for $\overline{\bf H}_1$ implies
    $\br_4(R_2^{-1}R_1R_2,Q^{-3})$.
  \item Let $Q_{1/2}$ satisfy $Q_{1/2}R_1Q_{1/2}^{-1}=R_1R_2R_1^{-1}$,
    $Q_{1/2}R_2Q_{1/2}^{-1}=R_3$ and
    $Q_{1/2}R_3Q_{1/2}^{-1}=R_3^{-1}R_1R_3$ (as does the symmetry of
    ${\bf H}_2$ groups described in Remark~\ref{rk:char}). Then the
    braid relations in the natural presentations for ${\bf H}_2$ imply
    $\br_4(R_1R_2R_3R_2^{-1}R_1^{-1},\,Q_{1/2}^{-3})$. 
    \item The natural presentation for Mostow groups implies $\br_4(R_1,P^{-2})$.
  \end{enumerate}
\end{lem}

\begin{proof}
  We give the proof of (1).  First note that $P^5=J^{-1}23123=12312J^{-1}$. Now using 
  $JR_2=R_3J$ and $J^{-1}R_2^{-1}=R_1^{-1}J^{-1}$, then $\br_4(R_2,R_3)$, 
  and $\br_5(R_1,R_2R_3R_2^{-1})$ we have:
  \begin{eqnarray*}
    (P^{-5}2P^{-5}2)(P^5\bar{2}P^5\bar{2}) 
    & = & (J\bar{2}\bar{1}\bar{3}\bar{2}\bar{1})2(\bar{3}\bar{2}\bar{1}\bar{3}\bar{2}J)2(J^{-1}23123)\bar{2}(12312J^{-1})\bar{2} \\
    & = & (J\bar{2}\bar{1}\bar{3}\bar{2}\bar{1})2\bar{3}\bar{2}\bar{1}(\bar{3}\bar{2}323)123\bar{2}(12312\bar{1}J^{-1}) \\
    & = & (J\bar{2}\bar{1}\bar{3}\bar{2}\bar{1})(2\bar{3}\bar{2}\bar{1}23\bar{2}123\bar{2})(12312\bar{1}J^{-1}) \\
    & = & (J\bar{2}\bar{1}\bar{3}\bar{2}\bar{1})123\bar{2}12\bar{3}\bar{2}\bar{1}(12312\bar{1}J^{-1}) \\
    & = & J\bar{2}\bar{1}\bar{2}1212\bar{1}J^{-1}.
  \end{eqnarray*}
  The last line is the identity using $\br_4(R_1,R_2)$. Therefore $\br_4(R_2,P^{-5})$.

Now consider (2). Recall (see Remark~\ref{rk:char}) that $S\in
\pu(2,1)$ satisfies
\begin{equation}\label{eq:S_E2}
  SR_1S^{-1}=R_1,\quad SR_2S^{-1}=R_3,\quad SR_3S^{-1}=R_3^{-1}R_2R_3.
\end{equation}
First, we simplify and then we use $S^{-1}R_3S=R_2$ and $S^{-1}R_1R_2R_1^{-1}S=R_1R_2R_3R_2^{-1}R_1^{-1}$.
\begin{eqnarray*}
(\bar{3}\bar{2}\bar{1}\bar{S}3\bar{3}\bar{2}\bar{1}\bar{S}3)(S123\bar{3}S123\bar{3}) 
& = & \bar{3}\bar{2}\bar{1}\bar{S}\bar{2}\bar{1}(\bar{S}3S)12S12 \\
& = &  \bar{3}\bar{2}\bar{1}\bar{S}(\bar{2}\bar{1}212)S12 \\
& = &  \bar{3}\bar{2}\bar{1}(\bar{S}12\bar{1}S)12 \\
& = &  \bar{3}\bar{2}\bar{1}123\bar{2}\bar{1}12 .
\end{eqnarray*}
All the terms in the last line cancel.

For (3) we argue similarly using the braid relations.
\begin{eqnarray*}
\lefteqn{ (Q^{-3}\bar{2}12Q^{-3}\bar{2}12)(Q^3\bar{2}\bar{1}2Q^3\bar{2}\bar{1}2) } \\
& \stackrel{\br_4(1,2)}{=} & (\bar{3}\bar{2}\bar{1})^2\bar{3}(\bar{2}\bar{1}\bar{2}12)
(\bar{3}\bar{2}\bar{1})^2\bar{3}(\bar{2}\bar{1}\bar{2}1212)3(123)^2(\bar{2}\bar{1}212)3(123)^2\bar{2}\bar{1}2 \\
& \stackrel{\br_3(1,3)}{=} & (\bar{3}\bar{2}\bar{1})^2\bar{3}1\bar{2}\bar{1}\bar{3}\bar{2}\bar{1}\bar{3}\bar{2}
(\bar{1}\bar{3}131)2312312\bar{1}3(123)^2\bar{2}\bar{1}2 \\
& \stackrel{\br_3(2,3)}{=} & (\bar{3}\bar{2}\bar{1})^2\bar{3}1\bar{2}\bar{1}\bar{3}\bar{2}\bar{1}
(\bar{3}\bar{2}323)12312\bar{1}3(123)^2\bar{2}\bar{1}2 \\
& \stackrel{\br_4(1,2)}{=} & (\bar{3}\bar{2}\bar{1})^2\bar{3}1\bar{2}\bar{1}\bar{3}
(\bar{2}\bar{1}212)312\bar{1}3(123)^2\bar{2}\bar{1}2 \\
& \stackrel{\br_3(1,3)}{=} & \bar{2}(2\bar{3}\bar{2})\bar{1}\bar{3}\bar{2}(\bar{1}\bar{3}1)\bar{2}(\bar{1}\bar{3}1)2(\bar{1}31)2(\bar{1}31)231(23\bar{2})\bar{1}2 \\
& \stackrel{\br_3(2,3)}{=} & \bar{2}(\bar{3}\bar{2}3)\bar{1}(\bar{3}\bar{2}3)\bar{1}(\bar{3}\bar{2}3)\bar{1}(\bar{3}23)1(\bar{3}23)1(\bar{3}23)1(\bar{3}23)\bar{1}2.
\end{eqnarray*}
The last line is the identity using $\br_7(R_1,R_3^{-1}R_2R_3)$.

Now consider (4). Recall that 
\begin{equation}\label{eq:S_H2}
  Q_{1/2}R_1Q_{1/2}^{-1}=R_1R_2R_1^{-1},\quad 
  Q_{1/2}R_2Q_{1/2}^{-1}=R_3,\quad 
  Q_{1/2}R_1Q_{1/2}^{-1}=R_3^{-1}R_1R_3.
\end{equation}
First use $Q_{1/2}^2=R_1R_2R_3$ to write $Q_{1/2}^{-3}R_1R_2R_3=Q_{1/2}^{-1}$.
Then use $Q_{1/2}^{-1}R_3Q_{1/2}=R_2$ and finally 
$Q_{1/2}^{-1}R_1R_2R_1R_2^{-1}R_1^{-1}Q_{1/2}=R_1R_3^{-1}R_2R_3R_1^{-1}$:
\begin{eqnarray*}
(Q_{1/2}^{-3}123\bar{2}\bar{1}Q_{1/2}^{-3}123\bar{2}\bar{1})(Q_{1/2}^312\bar{3}\bar{2}\bar{1}Q_{1/2}^312\bar{3}\bar{2}\bar{1})  
& = & \bar{Q}_{1/2}\bar{2}\bar{1}\bar{Q}_{1/2}3Q_{1/2}12Q_{1/2}12\bar{3}\bar{2}\bar{1} \\
& = & \bar{Q}_{1/2}\bar{2}\bar{1}212Q_{1/2}12\bar{3}\bar{2}\bar{1} \\
& = & \bar{Q}_{1/2}121\bar{2}\bar{1}Q_{1/2}12\bar{3}\bar{2}\bar{1} \\
& = & 1\bar{3}23\bar{1}12\bar{3}\bar{2}\bar{1}. 
\end{eqnarray*}
The last line is the identity using $\br_3(R_2,R_3)$.

Part (5)
\begin{eqnarray*}
(P^{-1}1P^{-2}1)(P^2\bar{1}P^2\bar{1}) & = & P^{-2}1\bar{3}23\bar{1}P^2\bar{1} \\
& = & \bar{3}\bar{1}313\bar{1}.
\end{eqnarray*}
The last line is the identity using $\br_3(R_1,R_3)$.

\end{proof}

\medskip

Since we know the order of $P$, $Q$ or the other symmetry in each
case, we combine Lemmas~\ref{lem-pre-braid}
and~\ref{lem-length4-braiding} to obtain the following corollary.

\begin{cor}\label{cor-extra-braiding}
  \begin{enumerate}
  \item In $\sigma_5$: The natural presentation implies
  \begin{itemize}
  \item $\br_2(P^5,\,R_2R_3^{-1}R_2^{-1}R_1R_2R_3R_2^{-1})$;
  \item $\br_6(R_2,R_3^{-1}R_2^{-1}R_1^{-1}R_2R_3R_2^{-1}R_1R_2R_3)$.
  \end{itemize}
   \item In ${\bf E}_2$: Let $R_1$, $R_2$ and $R_3$ satisy the braid
     relations in the natural presentation for ${\bf E}_2$ groups, and
     let $S$ have the conjugation relations of
     equation~\eqref{eq:S_E2}. Then we have
     \begin{itemize}
     \item $\br_2(SR_1R_2R_3,\,R_2^{-1}R_1R_2)$, in particular, $\br_2((R_1R_2R_3)^3,\,R_2^{-1}R_1R_2)$;   
     \item $\br_6(R_1R_2R_1^{-1}, R_3)$. 
     \end{itemize}
  \item In $\overline{\bf H}_1$: The natural presentation implies
  \begin{itemize}
  \item  $\br_2((R_1R_2R_3)^3,\,R_2^{-1}R_1R_2R_3R_2^{-1}R_1^{-1}R_2)$;
  \item     $\br_{14}(R_2^{-1}R_1R_2, R_3^{-1}R_1R_2R_1^{-1}R_3)$.
  \end{itemize}
  \item In ${\bf H}_2$: Let $R_1$, $R_2$ and $R_3$ satisfy the braid
    relations in the natural presentation for ${\bf H}_2$ groups and the
    conjugation relations in equation~\eqref {eq:S_H2}. Then we have
 \begin{itemize}
  \item $\br_2(Q_{1/2}^3,\,R_2^{-1}R_1R_2)$, in particular
    $\br_2\bigl((R_1R_2R_3)^3,R_2^{-1}R_1R_2)$;
  \item $\br_{10}(R_1R_2R_3R_2^{-1}R_1^{-1}, R_2)$.
  \end{itemize}
  \end{enumerate}
\end{cor}

\begin{proof}
In each case, we use Lemma~\ref{lem-pre-braid} with $C=R_1R_2R_3$. 
\begin{enumerate}
\item We set $A=R_2$ and $B=P^{-5}$. We have 
$P^{-5}R_2P^5=R_3^{-1}R_2^{-1}R_1^{-1}R_2R_3R_2^{-1}R_1R_2R_3$ and
so
$$
2\bar3\bar2123\bar2=(123)(\bar3\bar2\bar12\bar3\bar2123)\bar{2}=C(BA^{-1}B^{-1})A^{-1}.
$$
This proves the first part. For the second, observe $P$ has order 30 and so $P^5$
has order 6.
\item We write $\Gamma$ for the group generated by $R_1,R_2$ and $R_3$
  and relations given by the natural presentation. There is a unique
  $\sigma\in Aut(\Gamma)$ that satisfies
  $$ 
  \sigma(R_1)=R_1, \sigma_1(R_2)=R_3, \sigma_1(R_3)=R_3^{-1}R_2R_3,
  $$ 
  and $\sigma^3=Id$.  Now consider the corresponding morphism
  $\varphi:\Z/3\Z\rightarrow {\textrm Aut}(\Gamma)$, where
  $\varphi(k)=\sigma^k$, and write $\Gamma'=\Gamma\rtimes_\varphi
  \Z/3\Z$ for the corresponsing semi-direct product.

  Finally, we set $R_j'=(R_j,0)$ and $S'=(0,1)\in\Gamma'$, and
  consider $A=R_1'R_2'R_1'^{-1}$,
  $B=R_3'^{-1}R_2'^{-1}R_1'^{-1}S'^{-1}$. Then we have
  $$
  BAB^{-1}=(S'R_1'R_2'R_3')^{-1}R_1'R_2'R_1'^{-1}(S'R_1'R_2'R_3')=R_3'
  $$ 
  and so
  $$
  R_2'^{-1}R_1'R_2' = (R_1'R_2'R_3')(R_3'^{-1})R_1'R_2'^{-1}R_1'^{-1}=C(BAB^{-1})A^{-1}.
  $$
  Since $S'$ and $R_1'R_2'R_3'$ commute and $R_1'R_2'R_3'$ has order
  6, we see that $S'R_1'R_2'R_3'$ has order 6. Hence part~(2) of
  Lemma~\ref{lem-pre-braid} implies
  $\br_6(R_1'R_2'R_1'^{-1},R_3')$, which gives
  $\br_6(R_1R_2R_1^{-1},R_3)$ after projecting onto the
  first factor of the semi-direct product. 

  Also, part~(1) of Lemma~\ref{lem-pre-braid} gives
  $\br_2(S'R_1'R_2'R_3',\,R_2'^{-1}R_1'R_2')$, i.e. $S'R_1'R_2'R_3'$
  commutes with $R_2'^{-1}R_1'R_2'$. Now $S'$ has order 3 and
  $R_1R_2R_3$ has order 6, so $(S'R_1'R_2'R_3')^3=(R_1'R_2'R_3')^3$,
  so $(R_1'R_2'R_3')^3$ also commutes with $R_2'^{-1}R_1'R_2'$,
  hence the first part.

\item 
  We set $A=R_2^{-1}R_1R_3$ and $B=Q^{-3}$. We have
  $Q^{-3}R_2^{-1}R_1R_3Q^3=R_3^{-1}R_1R_2R_1^{-1}R_3$. This means that
  $$
  \bar{2}123\bar{2}\bar{1}2=(123)(\bar{3}1\bar{2}\bar{1}3)(\bar{2}\bar{1}2)=C(BA^{-1}B^{-1})A^{-1}.
  $$
  This proves the first part.
  For the second part, observe that $Q$ has order 42 and so $Q^3$ has order 14.
\item The main difficulty is similar to the one in case~(2), and it
  comes from the fact that the symmetry $Q_{1/2}$ is not in the group
  generated by $R_1$, $R_2$ and $R_3$. Once again, we construct the
  symmetry inside a semi-direct product constructed
  group-theoretically from the group given by the natural
  presentation.
  
  Start by noting that equation~\eqref{eq:S_H2} defines an
  automorphism $\alpha$ of the group $F$ defined by the ${\bf H_2}$
  natural presentation. To see that the map $R_1\mapsto
  R_1R_2R_1^{-1}$, $R_2\mapsto R_3$, $R_3\mapsto R_3^{-1}R_1R_3$
  induces a homomorphism $\alpha:F\rightarrow F$, one needs to verify
  that $\br_3(3,\bar313)$, $\br_3(\bar313,12\bar1)$,
  $\br_5(12\bar1,3)$, $\br_5(12\bar1,\bar3\bar13\cdot 3\cdot \bar313)$. The only
  non-trivial verification is the fact that $\br_5(1,\bar323)$ implies
  $\br_5(3,12\bar1)$, and this follows from an easy computation using
  $\br_3(1,3)$ and $\br_3(2,3)$.

  One easily verifies that $\alpha:F\rightarrow F$ is bijective, so it
  is an automorphism of $F$. Explicit computation shows that
  $\alpha^2$ is given by the conjugation in $F$ by $R_1R_2R_3$, so
  $\alpha$ has order 30, and $\alpha^{15}$ has order $2$. We denote by
  $\varphi:Z/2\Z\rightarrow {\textrm Aut}(F)$ the corresponding
  morphism, and $F'=F\rtimes_\varphi Z/2\Z$.

  In $F'$, we have a symmetry playing the same role as $Q_{1/2}$, so
  from this point on we ignore the fact that $Q_{1/2}$ is not in $F$.

  We set $A=R_3$ and $B=Q_{1/2}^{-3}$. 
  Moreover, $Q_{1/2}^{-3}R_3Q_{1/2}^3=R_3^{-1}R_1R_2R_1^{-1}R_3$ and so
  $$
    \bar{2}12=(123)(\bar{3}1\bar{2}\bar{1}3)(\bar{3})=C(BA^{-1}B^{-1})A^{-1}.
  $$ 
  This proves the first part.  Finally, $Q_{1/2}$ has order 30 and so
  $Q_{1/2}^3$ has order 10, which proves the second part.
\end{enumerate}
\end{proof}

\medskip

We need to treat $\sigma_1$ and $\sigma_{10}$ separately, since they do not fit into the general framework described
above.

\begin{prop} \label{prop-braid-s1}
For $\sigma_1$: If $JR_iJ^{-1}=R_{i+1}$, $\br_6(R_2,R_3)$ and $P^8$ then the braid relations
$\br_4(R_1,R_3^{-1}R_2R_3)$,
$ \br_3(R_1,R_3^{-1}R_2^{-1}R_3R_2R_3)$ and
$ \br_3(R_1,R_2R_3R_2R_3^{-1}R_2^{-1})$ are equivalent.
\end{prop}

\begin{proof}
Since $P=R_1J$ has order 8 we see that $R_1R_2R_3=P^3=P^{-5}=R_3^{-1}R_2^{-1}R_1^{-1}R_3^{-1}R_2^{-1}J$. Thus
\begin{eqnarray*}
1(23\bar{2})1(23\bar{2})\bar{1}(2\bar{3}\bar{2})\bar{1}(2\bar{3}\bar{2})
& = & (\bar{3}\bar{2}\bar{1}\bar{3}\bar{2}J)\bar{2}123\bar{2}\bar{1}2(\bar{J}23123)
(2\bar{3}\bar{2}) \\
& = & \bar{3}\bar{2}\bar{1}\bar{3}\bar{2}\bar{3}231\bar{3}\bar{2}3231(232\bar{3}\bar{2}) \\
& = & (\bar{3}\bar{2})\bigl(\bar{1}(\bar{3}\bar{2}\bar{3}23)1(\bar{3}\bar{2}323)1(\bar{3}\bar{2}\bar{3}23)\bigr)(23).
\end{eqnarray*}
Hence $\br_4(R_1,R_2R_3R_2^{-1})$ and $ \br_3(R_1,R_3^{-1}R_2^{-1}R_3R_2R_3)$ are equivalent.
Similarly 
\begin{eqnarray*}
1(232\bar{3}\bar{2})1(23\bar{2}\bar{3}\bar{2})\bar{1}(23\bar{2}\bar{3}\bar{2}) 
& = & (\bar{3}\bar{2}\bar{1}\bar{3}\bar{2}J)2\bar{3}\bar{2}123\bar{2}
(\bar{J}23123)23\bar{2}\bar{3}\bar{2} \\
& = & \bar{3}\bar{2}\bar{1}\bar{3}\bar{2}3\bar{1}\bar{3}231\bar{3}
231(2323\bar{2}\bar{3}\bar{2}) \\
& = & (\bar{3}\bar{2})\bigl(\bar{1}(\bar{3}\bar{2}3)\bar{1}(\bar{3}23)1(\bar{3}23)1(\bar{3}\bar{2}3)\bigr)(23).
\end{eqnarray*}
Hence $ \br_3(R_1,R_2R_3R_2R_3^{-1}R_2^{-1})$ and $\br_4(R_1,R_3^{-1}R_2R_3)$ are equivalent.
Finally, since 
$$
P^2R_1P^{-2}=R_1(R_2R_3R_2^{-1})R_1^{-1},\quad P^2(R_3^{-1}R_2R_3)P^{-2}=R_1
$$
we see that $\br_4(R_1,R_2R_3R_2^{-1})$ and $\br_4(R_1,R_3^{-1}R_2R_3)$ are equivalent.
\end{proof}

\medskip

\begin{prop} \label{prop-braid-s10}
For $\sigma_{10}$: If $JR_iJ^{-1}=R_{i+1}$, $\br_5(R_2,R_3)$ and $P^5$ then
\begin{itemize}
  \item $\br_2(R_1,R_3^{-1}R_2^{-1}R_3R_2R_3)$;
  \item $\br_3(R_1,R_3^{-1}R_2R_3)$.
\end{itemize}
\end{prop}

\begin{proof}
Since $P=R_1J$ has order 5, we have 
$$
R_1^{-1}R_3^{-1}R_2^{-1}=J(R_1J)^{-3}J^{-1}=J(R_1J)^2J^{-1}=R_2R_3J^{-1}
$$ 
and so
\begin{eqnarray*}
\bar{1}(\bar{3}\bar{2}\bar{3}23)1(\bar{3}\bar{2}323) 
& = & (23\bar{J})\bar{3}(J\bar{3}\bar{2})\bar{3}\bar{2}323 \\
& = & 23\bar{2}\bar{3}\bar{2}\bar{3}\bar{2}323.
\end{eqnarray*}
The last line is the identity since $\br_5(R_2,R_3)$.
Similarly 
\begin{eqnarray*}
\bar{1}(\bar{3}\bar{2}3)\bar{1}(\bar{3}23)1(\bar{3}23) 
& = & (23\bar{J})3\bar{1}\bar{3}(J\bar{3}\bar{2})\bar{3}23 \\
& = & 232\bar{3}\bar{2}\bar{3}\bar{2}\bar{3}23.
\end{eqnarray*}
Again, the last line is the identity since $\br_5(R_2,R_3)$.
\end{proof}

\medskip

\section{Commensurability invariants}

\subsection{Adjoint trace fields} \label{sec:tracefields}

The basic commensurability invariant we will use it the adjoint trace
field, i.e. the field generated by the traces of $Ad(\gamma)$, for all
group elements $\gamma$. Because our groups preserve a Hermitian form, this is simply the field generated by $|\tr \gamma|^2$,
where $\gamma$ runs over all group elements. It is well-known that this field is finitely generated,
and is a commensurability invariant (see section~2.5
of~\cite{mostowpacific}, Proposition~(12.2.1)
of~\cite{delignemostow} or \cite{paupert}).

In order to compute these fields explicitly, we first use an upper
bound given by Pratoussevitch's trace formula, i.e. Theorems~4 and~10
of~\cite{pratoussevitch} (see also section~17.2
of~\cite{mostowpacific}). A convenient formulation is given in
Corollary~5.9 of~\cite{parkertraces}, which gives the following:
\begin{prop}\label{prop:generalupperbound}
  For every triangle group generated by three reflections of order
  $p$, the field of traces $\Q(\tr{\rm Ad}\ \Gamma)$ is totally real,
  contained in $\Q(\vert\rho\vert^2, \vert\sigma\vert^2,
  \vert\tau\vert^2, \rho\sigma\tau, \bar\rho\bar\sigma\bar\tau, a)$,
  where $a=e^{2\pi i/p}$.
\end{prop}
We have phrased this in terms of non-symmetric triangle groups, but
the symmetric case is of course a special case obtained by taking
$\rho=\sigma=\tau$.

In fact, for all groups we consider, the upper bound in
Proposition~\ref{prop:generalupperbound} can be simplified. Indeed,
$\vert\rho\vert^2$, $\vert\sigma\vert^2$, $\vert\tau\vert^2$ are all
rational except for the Thompson groups of type ${\bf H_2}$, where
$\vert\sigma\vert^2=(3+\sqrt{5})/2$, but the latter is then contained
in $\Q(\rho\sigma\tau)$. Moreover, all number fields we consider are
Galois, so the complex conjugate $\bar\rho\bar\sigma\bar\tau$ is also
contained in $\Q(\rho\sigma\tau)$. We then get the following result.
\begin{prop}\label{prop:specialupperbound}
  If $\Gamma$ is any of the sporadic, Thompson or Mostow lattices,
  the field $\Q(\tr{\rm Ad}\ \Gamma)$ is contained in
  $\Q(\rho\sigma\tau,a)$.
\end{prop}
This gives an upper bound for the adjoint trace field, which is given
by the real subfield of $\Q(\rho\sigma\tau,a)$. In fact we check that
this upper bound is sharp. In order to do this, one needs to compute a
few explicit traces, see the formulas in section~\ref{sec:noneq}. We
also use
$$
\tr(R_3R_2R_1)=3-|\rho|^2-|\sigma|^2-|\tau|^2-u^3\bar\rho\bar\sigma\bar\tau,
$$ 
since $|\tr(R_3R_2R_1)|^2$ often gives a generator of the adjoint
trace field. In particular, since $\tr(R_1)=u^2+2\overline{u}$ where
$a=u^3$, we get:
\begin{prop}\label{prop:lowerbound}
  Let $\mu=3-|\rho|^2-|\sigma|^2-|\tau|^2$.  The field
  $\Q(|a+2|^2,|\mu-\bar a \rho\sigma\tau|^2)$ is contained in
  $\Q(\tr{\rm Ad}\ \Gamma)$.
\end{prop}
The values of $|\tr(R_3R_2R_1)|^2$ for all sporadic and Thompson
triangle lattices are gives in Tables~\ref{tab:tracesSporadic}
and~\ref{tab:tracesThompson}. Note that $|\mu-\bar a\rho\sigma\tau|^2$
generates the adjoint trace field for most groups, i.e. all but
$\T(5,{\bf \overline{H}_2})$. For that group, $|a+2|^2$ gives a generator.

\begin{table}[htbp]
\begin{tabular}{ccc}
  $\tau$        & $p$     & $|\tr(R_3R_2R_1)|^2$   \\\hline
  $\sigma_1 $   &  3      & $3(11+2\sqrt{6})$\\
                &  4      & $3(21+4\sqrt{2})$\\
                &  6      & $3(31+2\sqrt{6})$\\\hline
  $\overline{\sigma}_4$& 3 & $(19+3\sqrt{21})/2$\\
                &  4      & $17+3\sqrt{7}$\\
                &  5      & $\lambda$\\
                &  6      & $(49+3\sqrt{21})/2$\\
                &  8      & $(34+15\sqrt{2}+3\sqrt{14})/2$\\
                &  12     & $(34+15\sqrt{3}+3\sqrt{7})/2$\\\hline
  $\sigma_5$    &  2      & $(7+3\sqrt{5})/2$\\
                &  3      & $17+3\sqrt{5}$\\
                &  4      & $(24+9\sqrt{3}+3\sqrt{15})/2$\\\hline
  $\sigma_{10}$ &  3      & $4(3+\sqrt{5})$  \\
                &  4      & $(45+17\sqrt{5})/2$   \\
                &  5      & $(57+23\sqrt{5})/2$   \\
                &  10     & $39+16\sqrt{5}$
\end{tabular}
\caption{Values of traces giving a generator for the adjoint trace
  field. In the table, $\lambda$ is a generator for
  $\Q(\sqrt{14}\sqrt{5+\sqrt{5}})$, as can be seen from the fact that
  $\sqrt{14}\sqrt{5+\sqrt{5}}=(50\lambda^3-2040\lambda^2+18414\lambda-18538)/2403$.}\label{tab:tracesSporadic}
\end{table}
\begin{table}[htbp]
\begin{tabular}{cccc}
  ${\bf T}$     & $p$              & $|\tr(R_3R_2R_1)|^2$               & $|\tr(R_1)|^2$\\\hline
  ${\bf S_2}$   &  3               & $3+\sqrt{5}$                       & \\
                &  4               & $(6+\sqrt{3}+\sqrt{15})/2$         & \\
                &  5               & $-2\alpha_{15}+\alpha_{15}^2+\alpha_{15}^3$ & \\\hline
  ${\bf E_2}$   &  4               & $8+4\sqrt{3}$                      & \\
                &  6               & $16$                               & \\
                &  12              & $8+4\sqrt{3}$                      & \\\hline
                &                  &                                    &  \\[-11pt]
  ${\bf \overline{H}_1}$  &  2    & $2+2\cos(2\pi/7)$                   & \\
                &  7               & $(1+2\cos(2\pi/7))^2$              & \\\hline
  ${\bf H_2}$   &  2               & $(5+\sqrt{5})/2$                   & \\
                &  3               & $-1-5\alpha_{15}+2(\alpha_{15}^2+\alpha_{15}^3)$       & \\
                &  5               & $6+2\sqrt{5}$                      & \\
                &  10              & $5+2\sqrt{5}$                      & \\\hline
                &                  &                                    &  \\[-11pt]
  ${\bf \overline{H}_2}$  &  5    & $1$                                & $4+\sqrt{5}$\\
\end{tabular}
\caption{Values of traces giving a generator for the adjoint trace
  field, for Thompson lattices.  In the table, $\alpha_n$ stands for
  $2\cos(2\pi/n)$. We list $|\tr(R_1)|^2$ only if the first column
  does not already generate the adjoint trace
  field.}\label{tab:tracesThompson}
\end{table}

From these values, it is a bit cumbersome (but not really difficult)
to check that for each of the lattices we consider, the lower bound
given by Proposition~\ref{prop:lowerbound} has the same degree as the
upper bound given by the real subfield of $\Q(\rho\sigma\tau,a)$, see
Proposition~\ref{prop:specialupperbound}.

The adjoint trace fields for sporadic, Thompson and Mostow lattices
are listed in the appendix.  Those in the Appendix
(section~\ref{sec:mostowinvariants}) can be obtained in section~17.3
of Mostow's original paper~\cite{mostowpacific}, or more efficiently
by converting the groups into hypergeometric monodromy groups using
equation~\eqref{eq:converttomu} and applying Lemma~(12.5)
of~\cite{delignemostow}.

\subsection{Signature spectrum and non-arithmeticity index} \label{sec:spectrum}

Let $\Gamma$ be a lattice of $\pu(2,1)$, and assume $k=\Q(\tr{\rm
  Ad}\ \Gamma)$ is a totally real number field. It follows from the
discussion in section~12 in~\cite{delignemostow} that, up to complex
conjugation, $\Gamma$ is contained in a unique $k$-group whose real
points give a group isomorphic to $\pu(2,1)$. Moreover, $k$ is the
smallest possible number field with that property.

For the lattices considered in this paper, the $k$-structure is
obvious, since they are all contained in the integer points
$\su(H,\mathcal{O}_\mathbb{L})$ of groups the form $G=\su(H)$, where
$H$ is a Hermitian matrix with entries in $\mathbb{L}$, where
$\mathbb{L}$ is a CM-field with maximal totally real subfield given by
$k$. In particular, the automorphisms of $\mathbb{L}$ commute with
complex conjugation, they all preserve $k$, and they come in complex
conjugate pairs $\left\{
\varphi_{1},\overline{\varphi}_{1},...,\varphi_{r},\overline{\varphi}_{r}\right\}$,
where $\varphi_i$ and $\overline{\varphi}_i$ induce the same
automorphism of $k$, but for $i\neq j$, $\varphi_i$ and $\varphi_j$
have different restrictions to $k$. Note also that the restriction of
the automorphisms $\varphi_j$ to $k$ give all the automorphisms of
$k$.

For every automorphism $\varphi$ of $k$ which is the restriction of
some $\varphi_j$ as above, the Galois conjugate group $G^\varphi$ is
given by $\su(H^{\varphi_j})$, where $H^{\varphi_j}$ is obtained
from $H$ by applying $\varphi_j$ to every entry of $H$. Note that this
is only well-defined up to complex conjugation.

The following arithmeticity criterion is well known (we often refer to
this statement as the Mostow/Vinberg arithmeticity criterion), see
section~4 of~\cite{mostowpacific} or section~12
of~\cite{delignemostow}.
\begin{prop}
  $\Gamma$ is arithmetic if and only if for every non-trivial
  automorphism $\varphi$ of $k$, $G^\varphi$ preserves a definite
  Hermitian form.
\end{prop}
This suggests a way to measure how far a given lattice is from being
arithmetic, see the following definition.
\begin{dfn}
  Let $\Gamma$ be as above.
  \begin{enumerate}
  \item The signature spectrum of $\Gamma$ is the set of signatures
    $(p_i,q_i)$ of the Hermitian form preserved by $G^{\varphi_i}$,
    where $\varphi_i$ ranges over all automorphisms of $k=\Q(\tr{\rm
      Ad}\ \Gamma)$.
  \item The non-arithmeticity index of $\Gamma$ is the number of
    non-trivial automorphisms $\varphi$ of $k$ such that $G^\varphi$
    preserves an indefinite Hermitian form.
  \end{enumerate}
\end{dfn}
Note that the signature spectrum is not completely well-defined, since
$\su(H)=\su(\lambda H)$ for any real number $\lambda\neq 0$ (in
particular one could take $\lambda<0$), but this is really the only
ambiguity. Observe also that the signature spectrum clearly determines
the non-arithmeticity index.

Now the key observation is that a lattice in $\su(2,1)$ acts
irreducibly on $\C^3$, so it preserves a unique Hermitian form (up to
scaling). This is of course also true for the Galois conjugates of a
given lattice. Since a subgroup of finite index in a lattice is also a
lattice, we get that the non-arithmeticity index is a commensurability
invariant. For future reference, we summarize this discussion in the
statement Proposition~\ref{prop:index}.
\begin{prop} \label{prop:index}
  Let $\Gamma$ be as above, and let $\Gamma'\subset \Gamma$ be a
  subgroup of finite index. Then $\Gamma$ and $\Gamma'$ have the same
  signature spectrum, and the same non-arithmeticity index.
\end{prop}

\subsection{Commensurators}

In order to refine the partition into commensurability classes, it is
also useful to consider properties of the commensurator proved by
Margulis.  Recall that the commensurator of $\Gamma$ in $G$ is the
group $C_G(\Gamma)$ of elements $g\in G$ such that $\Gamma \cap
g\Gamma g^{-1}$ has finite index in both $\Gamma$ and $g\Gamma
g^{-1}$. The following result follows from Theorem~IX.1.13
in~\cite{margulisBook}.
\begin{thm}\label{thm:commensurator}
  Let $\Gamma$ be a non-arithmetic lattice in $G=\pu(2,1)$. Then
  $\Gamma$ has finite index in $C_G(\Gamma)$, in particular
  $C_G(\Gamma)$ is a lattice.
\end{thm}
In section~\ref{sec:numberOfClasses}, we will use the following
reformulation of Theorem~\ref{thm:commensurator} (obtained from the
latter by taking $\Gamma$ to be the common commensurator of $\Gamma_1$
and $\Gamma_2$).
\begin{prop}\label{prop:commonsupgroup}
  Suppose $\Gamma_1$ and $\Gamma_2$ are commensurable non-arithmetic
  lattices in $\pu(2,1)$. Then there exists a lattice $\Gamma$ and a
  $g\in\pu(2,1)$ such that $\Gamma_1$ and $g\Gamma_2g^{-1}$ are both
  finite index subgroups of $\Gamma$.
\end{prop}

\section{Commensurability relations} \label{sec:commensurability}

The goal of this section is to prove
Theorem~\ref{thm:numberofclasses}. The detailed statement, giving
explicit representatives for each commensurability class, is given in
the form of a table, see Table~\ref{tab:commclasses}. The end result
is that among 2-dimensional non-arithmetic Deligne-Mostow, sporadic
and Thompson lattices, there are precisely 22 commensurability
classes.

\subsection{Some isomorphisms between triangle groups} \label{sec:isomorphisms}

\subsubsection{Non-rigid non-equilateral triangle groups} \label{sec:iso-nonrigid}

\begin{prop}
\begin{enumerate}
\item  For every $\p$, the group $\T(\p,{\bf S_1})$ is
  conjugate to the sporadic triangle group $\S(\p,\bar\sigma_4)$.
\item  For every $\p$, the group $\T(\p,{\bf E_1})$ is
  conjugate to the sporadic triangle group $\S(\p,\sigma_1)$.
\end{enumerate}
\end{prop}

\begin{pf}
We start with the proof of part~(1).

We write $R_1$, $R_2$, $R_3$ for standard generators of a sporadic
triangle group for $\bar\sigma_4$. Recall that this is characterized
up to conjugation by $\tr R_1J=\bar\sigma_4$, and that this implies
that $R_1J$ has order 7,
\begin{eqnarray}\label{eq:braid4}
  &(R_iR_j)^2=(R_jR_i)^2
\end{eqnarray}
i.e. $\br(R_i,R_j)=4$ (for $i\neq j$), and
\begin{eqnarray}\label{eq:braid3}
  &R_1(R_2R_3R_2^{-1})R_1=(R_2R_3R_2^{-1})R_1(R_2R_3R_2^{-1})\\
  &R_1(R_3^{-1}R_3R_2)R_1=(R_3^{-1}R_3R_2)R_1(R_3^{-1}R_3R_2)
\end{eqnarray}
i.e. $\br(R_1,R_2R_3R_2^{-1})=\br(R_1,R_3^{-1}R_2R_3)=3$.

Now consider the group elements $M_1=R_3^{-1}R_2R_3$,
$M_2=R_2R_3R_2^{-1}$ and $M_3=R_1$. These three matrices generate the
sporadic group, since
$$
R_1=M_3,\quad R_2=M_2^{-1}M_1M_2,\quad R_3=M_1M_2M_1^{-1}.
$$

We claim that the three isometries $M_1,M_2,M_3$ (can be
simultaneously conjugated to) generate an ${\bf S_1}$ group.
Let ${\bf m}_1=R_3^{-1}{\bf n}_2$, ${\bf m}_2=R_2{\bf n}_3$ and
${\bf m}_3={\bf n}_1$ be polar vectors to the mirrors of 
$R_1$, $R_2$ and $R_3$. Then the parameters associated with
$\langle M_1,M_2,M_3\rangle$ are
\begin{eqnarray*}
\rho' & = & (u^2-\bar{u})\frac{\langle{\bf m}_2,{\bf m}_1\rangle}
{\Vert{\bf m}_2\Vert\,\Vert{\bf m}_1\Vert}=-\bar{u}^2\bar\sigma_4,\\
\sigma' &= &(u^2-\bar{u})\frac{\langle{\bf m}_3,{\bf m}_2\rangle}
{\Vert{\bf m}_3\Vert\,\Vert{\bf m}_2\Vert}
=\tau'=(u^2-\bar{u})\frac{\langle{\bf m}_1,{\bf m}_3\rangle}
{\Vert{\bf m}_1\Vert\,\Vert{\bf m}_3\Vert}
=u(\bar\sigma_4-\sigma_4^2).
\end{eqnarray*}
Since $\bar\sigma_4-\sigma_4^2=1$ we see that 
$|\rho'|=\sqrt{2}$, $|\sigma'|=|\tau'|=1$ and 
$\rho'\sigma'\tau'=-\bar\sigma_4=(1+i\sqrt{7})/2$. 
These are the same parameters as for ${\bf S}_1$. Therefore the two 
groups are conjugate.

The proof of part~(2) is similar. In that case the sporadic group is
defined by $\tau=-1+i\sqrt{2}$, $(R_1R_2R_3)$ has order 8 and 
$$
\br(R_j,R_k)=6,\ \br(R_1,R_2R_3R_2^{-1})=4,\ \br(R_1,R_2R_3R_2R_3^{-1}R_2^{-1})=3.
$$

Explicit generators of type ${\bf E_1}$ are given by 
$$
M_1=R_2R_3R_2R_3^{-1}R_2^{-1},\quad M_2=R_1R_3^{-1}R_2^{-1}R_3R_2R_3R_1^{-1},\quad M_3=R_1,
$$
and these generate the same group because
$$
R_1=M_3,\ R_2=(M_3^{-1}M_2^{-1}M_3)M_1(M_3^{-1}M_2M_3),\ R_3=M_1(M_3^{-1}M_2M_3)M_1^{-1}.
$$
Setting ${\bf m}_1=R_2R_3{\bf n}_2$, ${\bf m}_2=R_1R_3^{-1}R_2^{-1}{\bf n}_3$
and ${\bf m}_3={\bf n}_1$ we have
$$
\rho'=4\bar\tau^2-11\tau+\tau^4=i\sqrt{2},\quad 
\sigma'=-\bar{u}(2\tau-\bar{\tau}^2)=\bar{u},\quad
\tau'=-u(2\tau-\bar{\tau}^2)=u.
$$
\end{pf}

The Thompson groups with $\p=2$ were shown to be commensurable to
explicit Mostow groups in~\cite{thompson}. In a similar vein, we have
the following.
\begin{prop}\label{prop:mostow_saves_us}
\begin{enumerate}
\item The group $\T(7,{\bf \bar H_1})$ is conjugate to the Mostow
  group $\Gamma(7,9/14)$.
\item The group $\T(5,{\bf \bar H_2})$ is conjugate to the Mostow
  group $\Gamma(5,7/10)$.
\end{enumerate}
\end{prop}

\begin{pf}
\begin{enumerate}
\item
  In the group $\Gamma(7,9/14)$, one
  verifies that $M_1=R_1$,
  $M_2=(R_2R_1^{-1}R_2)^{-1}R_1(R_2R_1^{-1}R_2)$ and $M_3=R_3$ are
  conjugate to standard generators for $\T(7,{\bf \bar H_1})$. 
  Writing ${\bf m}_1={\bf n}_1$, ${\bf m}_2=(R_2R_1^{-1}R_2)^{-1}{\bf n}_1$,
  ${\bf m}_3={\bf n}_3$ and $\tau^3=1$, we find
  $$
  \rho'=e^{6\pi i/7}(-1-i\sqrt{7})/2,\quad
  \sigma'=e^{6\pi i/7}\bar\tau,\quad \tau'=\tau.
  $$
  Since $\tau^3=1$ this means
  $$
  |\rho'|=\sqrt{2},\quad |\sigma'|=|\tau'|=1,\quad \rho'\sigma'\tau'=
  e^{-2\pi i/7}(-1-i\sqrt{7})/2
  $$
  as required.
  
  One can check that
  $$
    R_2=M_3(M_2M_3M_2^{-1}M_1)^{-3}M_1^{-1}M_2M_1,
  $$ 
  which shows that $M_1$, $M_2$, $M_3$ generate the same group as
  $R_1$, $R_2$, $R_3$.
\item
  In the group $\Gamma(5,7/10)$, one verifies that $M_1=R_1$,
  $M_2=R_2^{-1}R_3R_2$ and $M_3=R_2$ are conjugate to standard
  generators for $\T(5,{\bf \bar H_2})$. Indeed, writing ${\bf m}_1={\bf n}_1$,
  ${\bf m}_2=R_2^{-1}{\bf n}_3$, ${\bf m}_3={\bf n}_2$ and arguing as above:
  $$
  \rho'=-\bar{u}\bar{\tau}-u^2\tau^2,\quad
  \sigma'=-u^3\bar{\tau},\quad
  \tau'=-u\bar{\tau}.
  $$
  In $\Gamma(5,7/10)$ we have $u=e^{2\pi i/15}$ and $\tau=-e^{-i\pi/3}$.
  Hence $|\rho'|=2\cos(\pi/5)$, $|\sigma'|=|\tau'|=1$ and 
  $\rho'\sigma'\tau'=-e^{2\pi i/5}-e^{4\pi i/5}$.
  
  Moreover, $R_1$, $R_2$ and $R_3$
  generate the corresponding Mostow group (the clearly
  generate the subgroup generated by $R_1$, $R_2$ and $R_3$, and that
  subgroup is equal to the group generated by $R_1$ and $J$, because
  3 does not divide the order of $R_1J$, which is 4,
  see~\cite{sauter} for instance).
\end{enumerate}
\end{pf}

\subsubsection{Rigid non-equilateral triangle groups} \label{sec:iso-rigid}

In this section by explain some relations of rigid triangle groups
(Thompson groups with parameters ${\bf S_2}$, ${\bf S_3}$, ${\bf S_4}$
or ${\bf E_3}$, see Table~\ref{tab:Trigid}) with other triangle
groups.
\begin{prop}
  For every $p=5,6,7,8,9,10,12,18$, $\T(p,{\bf S_3})$ is the Livn\'e
  group with parameter $p$.
\end{prop}
\begin{pf}
  This follows from changes of parameters as
  in~\cite{kamiyaparkerthompson}. More specifically, in $\T(p,{\bf
    S_3})$, the complex reflections $R_1$, $R_1R_2R_1^{-1}$, $R_3$
  are generators that pairwise have braid length 3. This allows us to
  identify as Mostow groups, see section~\ref{sec:mostowinvariants} for more
  details.
\end{pf}

\begin{prop}
  The lattices $\T(p,{\bf S_4})$, $p=4,5,6,8,12$ are isomorphic to
  arithmetic Mostow groups.
\end{prop}
\begin{pf}
In the Mostow group generated by $R_1$ and $J$, the elements
$J(R_1R_2)^{-1}$ and $R_1R_3J$ are complex reflections, with known angle
(see~\cite{mostowpacific}, or~\cite{parkersurvey}).

Note also that $J(R_1R_2)^{-1}$ commutes with $R_2$, since
$$
  J R_2^{-1}R_1^{-1}\cdot R_2\cdot R_1R_2 J^{-1}\cdot R_2^{-1}=JR_1J^{-1}R_2^{-1}=Id,
$$ 
where we have used the braid relation $\br(R_1,R_2)=3$. Similarly, one
checks that $R_1R_3J$ commutes with $R_3$.

Also, we have
$$
(R_1JR_2^{-1}R_1^{-1})^2=R_1(JR_2^{-1})^2R_1^{-1}=(JR_2^{-1})^2.
$$ 
This implies that the braid length $\br(R_1,J(R_1R_2)^{-1})$ is either
2 or 4, but one easily checks that these two complex reflections do
not commute, so $\br(R_1,J(R_1R_2)^{-1})=4$.
 
Above, we have used the fact that $R_1$ commutes with $(JR_2^{-1})^2$,
which is true since
$$
\bigl[R_1,(JR_2^{-1})^2\bigr]
=R_1JR_2^{-1}J^{-1}J^{-1}R_2^{-1}R_1^{-1}R_2JJR_2J^{-1}
=R_1R_3^{-1}R_1^{-1}R_3^{-1}R_1R_3=Id.
$$

In the five Mostow groups listed in Table~\ref{tab:mostow234}, the
corresponding elements (either $J(R_1R_2)^{-1}$ or $(R_1R_3J)^{-1}$
depending on the order of generators) have the same order as $R_j$,
and the elements in the second column are $(2,3,4)$ triangle group
generators.
\begin{table}
  \begin{tabular}{c|c}
    Mostow group        &    (2,3,4)-generators\\
\hline
     $\Gamma(4,1/4)$    &   $R_1$, $J(R_1R_2)^{-1}$, $R_2$\\    
     $\Gamma(5,1/10)$   &   $R_1$, $J(R_1R_2)^{-1}$, $R_2$\\    
     $\Gamma(6,0)$      &   $R_1$, $J(R_1R_2)^{-1}$, $R_2$\\    
     $\Gamma(8,1/8)$    &   $R_1$, $(R_1R_3J)^{-1}$, $R_3$\\    
     $\Gamma(12,1/4)$   &   $R_1$, $(R_1R_3J)^{-1}$, $R_3$\\    
  \end{tabular}
  \caption{We write these 5 Mostow groups as (2,3,4)-triangle groups,
    by considering the triple of reflections in the second
    column.}\label{tab:mostow234}
\end{table}
  \end{pf}

\begin{prop} \label{prop:isoS5}
  For every $p>2$, $\T(p,{\bf S_5})$ is isomorphic to the group
  $\S(p,\sigma_{10})$. 
\end{prop}
\begin{pf}
  This follows from two successive changes of generators as
  in~\cite{kamiyaparkerthompson}. One checks that the $2,3,5;5$
  triangle groups are the same as $3,5,5;2$ triangle groups, which are
  the same as $5,5,5;3$ triangle groups. The latter correspond to
  sporadic $\sigma_{10}$ groups.
\end{pf}

For the special case $p=10$ in Proposition~\ref{prop:isoS5}, we have
an extra isomorphism.
\begin{prop}\label{prop:isoS5_10}
  The group $\T(10,{\bf S_5})$ is isomorphic to $\T(10,{\bf H_2})$
  (and also to $\S(10,\sigma_{10})$).
\end{prop}
\begin{pf}
  In the group $\T(10,{\bf H_2})$, one considers
  $M=((R_1R_2R_3)^2R_2^{-1}R_1R_2)^{-3}$, which is a complex
  reflection with angle $\pi/5$. 
  
  One checks by explicit computation that the matrices $R_2,M,R_3$
  generate a $(2,3,5)$-triangle group, i.e. the group $\T(10,{\bf
    S_5})$.
\end{pf}

\subsection{Determination of the number of commensurability classes} \label{sec:numberOfClasses}

In this section, we summarize the current lower bound on the number of
commensurability classes of non-arithmetic lattices in $\pu(2,1)$.
A lot of this can be done by using only
rough commensurability invariants, i.e. cocompactness, adjoint
trace fields, and non-arithmeticity index (see
section~\ref{sec:spectrum}).

The table for Mostow and Deligne-Mostow groups show that there are at
most 13 commensurability classes of Deligne-Mostow lattices in
$\pu(2,1)$. As mentioned above, the results
in~\cite{delignemostowbook},~\cite{kappesmoller} and
~\cite{mcmullengaussbonnet} imply that there are in fact precisely 9
commensurability classes there.

\subsubsection{Cocompact groups}

Among the non-arithmetic Thompson lattices, the groups $\T(5,{\bf S_2})$ and
$\T(3,{\bf H_2})$ cannot be commensurable to any Deligne-Mostow
lattice nor to any sporadic group, but in principle they could be
commensurable with each other. We will now exclude that possibility:
\begin{prop}\label{prop:notcomm_compact}
  The groups $\Gamma_1=\T(5,{\bf S_2})$ and $\Gamma_2=\T(3,{\bf H_2})$
  are not commensurable.
\end{prop}

We give an argument that relies on the following volume
estimate for lattices containing complex reflections of large order.
\begin{prop}\label{prop:volest_refl}
Let $\Gamma$ be a discrete subgroup of ${\rm PU}(2,1)$ containing a
complex reflection $A$ of order $n\ge 7$. Let $m_A$ denote the mirror
of $A$, and let $\Gamma_A$ denote the stabilizer of $m_A$ in $\Gamma$. Then
$$
{\rm Vol}(\Gamma\backslash{\bf H}^2_\C)\ge 
\frac{\pi(1-2\sin\frac{\pi}{n})}{2 n\sin\frac{\pi}{n}}{\rm Vol}(m_A/\Gamma_A).
$$
Moreover, if there is no $g\in\Gamma$ such that $g(m_A)$ is orthogonal to $m_A$, then 
$$
{\rm Vol}(\Gamma\backslash{\bf H}^2_\C)\ge 
\frac{\pi(1-2\sin\frac{\pi}{n})}{n\sin\frac{\pi}{n}}{\rm Vol}(m_A/\Gamma_A).
$$
\end{prop}

\begin{proof}
Normalize in the Siegel domain so that 
$m_A=\bigl\{(\zeta,v,u)\in{\bf H}^2_\C : \zeta = 0\bigr\}$. 
Then, a point $(\zeta_1,v_1,u_1)\in{\bf H}^2_\C$ a distance $\delta$ 
from $L_A$ satisfies
$$
\cosh^2\left(\frac{\delta}{2}\right)
=\frac{|\zeta_1|^2+u_1}{u_1}.
$$
In other words $|\zeta_1|^2=u_1\bigl(\cosh(\delta)-1\bigr)/2$.
Let $N(\delta)$ be the $\delta$-neighborhood of $m_A$. 
Then
\begin{eqnarray*}
{\rm Vol}\bigl(N(\delta)/\Gamma_A\bigr) & = &
\frac{1}{n}\int_u\int_v\int_x\int_y \frac{4}{u^3} \, du\, dv\, dx\, dy \\
& = & \frac{1}{n}\int_u\int_v \frac{4\pi|\zeta_1|^2}{u^3} \, du\,  dv\\
& = & \frac{2\pi(\cosh\delta-1)}{n}\int_u\int_v\frac{1}{u^2}\,du\,dv \\
& = & \frac{2\pi(\cosh\delta-1)}{n}{\rm Vol}(m_A/\Gamma_A) \\
\end{eqnarray*}
Using Theorem~5.2 of~\cite{parkerjorgensen}, 
we see that if there is no $g\in\Gamma$ so that $g(m_A)$
is orthogonal to $m_A$ and if $\cosh(\delta)\ge \frac{1}{2\sin(\pi/n)}$ then
$N(\delta)$ does not intersect its images under elements of
$\Gamma-\Gamma_A$.
For such a $\delta$ we have
\begin{eqnarray*}
{\rm Vol}(\Gamma\backslash{\bf H}^2_\C)
& \ge & {\rm Vol}\bigl(N(\delta)/\Gamma_A\bigr) \\
& \ge & \frac{\pi(1-2\sin\frac{\pi}{n})}{n\sin\frac{\pi}{n}}{\rm Vol}(m_A/\Gamma_A).
\end{eqnarray*}
\end{proof}

\begin{pf} (of Proposition~\ref{prop:notcomm_compact})
  We argue by contradiction, let us assume they are
  commensurable. Then by Proposition~\ref{prop:commonsupgroup} we may
  assume that both of them are contained in a common lattice $\Gamma$.

  Recall that $\Gamma_1$ has Euler characteristic 133/300, and
  $\Gamma_2$ has Euler characteristic 26/75. Let us denote by $d_j$
  the index of $\Gamma_j$ in $\Gamma$.  Since $133/300d_1=26/75d_2$,
  and 133 and 26 are relatively prime, we must have $d_1=133d_1'$ and
  $d_2=26d_2'$ for some integers $d_1',d_2'$. In other words, the
  Euler characteristic of $\Gamma$ is of the form $1/300d$ for some
  integer $d$.

  We denote by $R_j^{(k)}$ the $j$-th standard generator of
  $\Gamma_k$.  Consider the $\C$-Fuchsian subgroup $F_1$ of
  $\Gamma_1$, generated by $R_1^{(1)}$ and $R_2^{(1)}$. Since these
  two reflections braid with length 5, $F_1$ is central extension of a
  $(2,5,5)$-triangle group, with center generated by a complex
  reflection $(R_1^{(1)}R_1^{(2)})^2$ of order 10 (see
  Proposition~\ref{prop:braiding}).
  
  The commensurator $\Gamma$ contains a possibly larger Fuchsian
  subgroup $F\supset F_1$, which is a central extension of either a
  (2,5,5)-triangle group or a (2,4,5)-triangle group, with center of
  order $n$, where $n$ is a multiple of 10.

  Suppose first that there is no $g\in\Gamma$ such that $g(m_1)$ is
  orthogonal to $m_1$. Then by Proposition~\ref{prop:volest_refl},
  $$
  Vol(\Gamma\setminus \CH 2)
  \geq V\pi\frac{1-2\sin\frac{\pi}{n}}{n\cdot\sin\frac{\pi}{n}}
  \geq V\pi\frac{1-2\sin\frac{\pi}{10}}{10\cdot\sin\frac{\pi}{10}},
  $$
  where $V=\pi/10$ is the co-area of the $(2,4,5)$-triangle group.

  It follows that 
  $$
  \chi(\Gamma\setminus \CH 2) = \frac{3}{8\pi^2} Vol(\Gamma\setminus \CH 2) \geq \frac{3}{400}\frac{\sqrt{5}-1}{2}> \frac{1}{216},
  $$
  which is impossible, since $\Gamma$ has Euler characteristic $1/300d_1<1/216$.

  Hence we assume there exists a $g\in\Gamma$ such that $g(m_1)$ is
  orthogonal to $m_1$. In that case, $gRg^{-1}$ gives an element of
  order 10 acting on $m_F$ as a rotation of order 10, where
  $R=(R_1^{(1)}R_1^{(2)})^2$ generates the center (i.e. pointwise
  stabilizer) of $m_F$.

  Now $F$ is a central extension of either a (2,4,5) or a
  (2,5,5)-triangle group, but these triangle groups contain no element
  of order 10, which is a contradiction.
\end{pf}

\subsubsection{Non-cocompact groups}

Among the non-cocompact lattices we constructed, there are three pairs
of non-arithmetic lattices with the same rough commensurability
invariants; the following result shows that these pairs are actually
in different commensurability classes.
\begin{prop}\label{prop:non-comm-cusped}
\begin{enumerate}
\item[(1)] The groups ${\mathcal S}(3,\sigma_1)$ and ${\mathcal S}(6,\sigma_1)$ are not commensurable.
\item[(2)] The groups ${\mathcal S}(4,\sigma_5)$ and ${\mathcal T}(4,{\bf S}_2)$ are not commensurable.
\item[(3)] The groups $\Gamma(6,1/6)$ and ${\mathcal T}(4,{\bf E}_2)$ are not commensurable.
\end{enumerate}
\end{prop}

Our proof of Proposition~\ref{prop:non-comm-cusped} relies on studying
the cusps of the groups ${\mathcal S}(3,\sigma_1)$, ${\mathcal
  S}(4,\sigma_5)$ and $\Gamma(6,1/6)$. An alternative proof of
part~(3) follows from work of Kappes and M\"oller \cite{kappesmoller},
because of Proposition~\ref{prop:t4E2mostow}.
\begin{prop}\label{prop:t4E2mostow}
  The group $\T(4,{\bf E_2})$ is (conjugate to) a subgroup of index 3
  in the Deligne-Mostow group $\Gamma_{\mu,\Sigma}$ for
  $\mu=(3,3,5,6,7)/12$, $\Sigma=\Z_2$
\end{prop}
\begin{pf}
A presentation for Deligne-Mostow groups without three-fold symmetry
is given by Pasquinelli in \cite{Pasquinelli2}. The group 
$\Gamma_{\mu,\Sigma}$ is $(3,4,4)$ in her notation.
The most convenient presentation is the alternative one given at the end 
of Section 4.1 of \cite{Pasquinelli2}, which for this group is
$$
\left\langle A,\,B,\,R\ :\ 
\begin{array}{c}
A^4,\ B^3,\ R^4,\ (BRA)^{24},\ (ARBR)^6, \\
\br_4(B,R),\ \br_2((BRA)^{-2},R),\ \br_2(A,B)
\end{array}
\right\rangle.
$$
One then easily verifies that writing
$R_1=A$, $R_2=R$, $R_3=BRB^{-1}$ gives a group isomorphic
to $\T(4,{\bf E_2})$. This is done by showing each of these
two presentations implies the other.
Furthermore, one can extend this isomorphism by writing
$S=B$ where $S$ is the extra symmetry from Remark~\ref{rk:char} (3).
Using the fact that $B$ has order 3, the relation 
$\br_4(B,R)$ is equivalent to $\br_3(R,BRB^{-1})$.
Then using this relation, the relation $\br_2((BRA)^{-2},R)$ is
equivalent to $\br_4(A,R)$. Together with $\br_2(A,B)$ this
immediately gives $\br_4(A,BRB^{-1})$ and $\br_4(A,R(BRB^{-1})R^{-})$.
Moreover, $B$ commutes with $ARBR$ and so $(ARBR)^6$ is 
equivalent to $(AR(BRB^{-1}))^6$. Finally, using $\br_4(A,R)$, 
$\br_4(B,R)$ and $\br_2(A,B)$
we have $(BRA)^{24}=(BRAR)^{24}(\bar{A}BR\bar{B}AR)^{-12}$. Thus
$(BRA)^{24}$ is equivalent to $(\bar{A}BR\bar{B}AR)^{12}$, where we also
use $(BRAR)^6$.
\end{pf}

\medskip

We remark that exactly the same argument shows that
$\T(3,{\bf E_2})$ is (conjugate to) a subgroup of index 3
  in the Deligne-Mostow group $\Gamma_{\mu,\Sigma}$ for
  $\mu=(1,1,3,3,4)/6$, and $\Sigma=\Z_2$ acting by swapping
  just the factors $1/6$.

For each of the groups ${\mathcal S}(3,\sigma_1)$, ${\mathcal
  S}(4,\sigma_5)$ and $\Gamma(6,1/6)$, the mirrors $m_1$ and $m_2$ of
$R_1$ and $R_2$ intersect in a point $p_{12}$ on $\partial{\bf
  H}^2_{\mathbb C}$. This means that the group $\langle R_1,
R_2\rangle$ is a parabolic group. We give an upper bound on the
largest cusp neighborhood that is precisely invariant under this
parabolic group. We discuss volume bounds for cusp groups in a more
general context, since we believe that these results could have wider
applications.

We begin by considering the general structure of parabolic groups
$\Gamma_\infty$ generated by two complex reflections $A$ and $B$ both
with order $p$.  The center $Z(\Gamma_\infty)$ of
$\Gamma_\infty=\langle A,B\rangle$ is generated by a vertical
translation, which we denote by $T$. We need to consider the following
three cases (see Proposition \ref{prop:braiding}).
\begin{enumerate}
\item[(1)] $p=3$: In this case, $A$ and $B$ braid with length 6, $\Gamma_\infty$ is a central extension of 
the rotation subgroup of a $(3,3,3)$ triangle group and $T=(AB)^3$.
\item[(2)] $p=4$: In this case, $A$ and $B$ braid with length 4, $\Gamma_\infty$ is a central extension of 
the rotation subgroup of a $(2,4,4)$ triangle group and $T=(AB)^2$.
\item[(3)] $p=6$: In this case, $A$ and $B$ braid with length 3, $\Gamma_\infty$ is a central extension of 
the rotation subgroup of a $(2,3,6)$ triangle group and $T=(AB)^3$.
\end{enumerate}
We will actually use some more detailed information about $\Gamma_\infty$.
This information can be deduced from the presentations of Heisenberg lattices 
in Section 7.1 Dekimpe \cite{Dekimpe}.
Instead of giving details of this, we give a geometrical proof instead.
To that end, let $\Lambda_\infty$ denote the subgroup of
$\Gamma_\infty$ consisting of Heisenberg translations. Note that
$\Lambda_\infty$ is a central extension of the translation subgroup of
the corresponding triangle group.

Since $Z(\Gamma_\infty)$ is a group of Heisenberg translations, it is
contained in $\Lambda_\infty$, and in fact
$Z(\Gamma_\infty)=Z(\Lambda_\infty)$. Moreover, the commutator
subgroup of $\Lambda_\infty$ is a finite index subgroup of
$Z(\Gamma_\infty)$. Our next goal is to determine that index in each
of the three cases.
\begin{lem}
  Let $\Gamma_\infty=\langle A,B\rangle$ be as above. Then, the
  Heisenberg lattice $\Lambda_\infty$ has index $p$ in $\Gamma_\infty$
  and is generated by $A^{-1}B$, $AB^{-1}$ and $T$.  The commutator
  subgroup of $\Lambda_\infty$ is all of $Z(\Gamma_\infty)$ when $p=3$
  or $6$ and has index 2 in $Z(\Gamma_\infty)$ when $p=4$.
\end{lem}

\begin{pf}
This lemma could be deduced from Dekimpe \cite{Dekimpe} as indicated
above.

Let $\Gamma_*$ denote the rotation subgroup of one of the above
Euclidean triangle groups, and let $\Lambda_*$ be its translation
subgroup. In each case $\langle A^{-1}B, AB^{-1}\rangle$ projects to
$\Lambda_*$ and the index of $\Lambda_*$ in $\Gamma_*$ is $p$.  This
may easily be checked using Euclidean geometry, for example by
normalizing the projection of $A$ to be $z\longmapsto e^{2\pi i/p}z$
and the projection of $B$ to be $z\longmapsto e^{2\pi i/p}z+1$. This gives
an obvious isomorphism between $\Lambda_*$ and the discrete ring
${\mathbb Z}[e^{2\pi i/p}]$ (recall that $p=3$, $4$ or $6$). 
Thus $\Lambda_\infty$ has index $p$ in $\Gamma_\infty$ and is 
generated by $A^{-1}B$, $AB^{-1}$ and $T$.

Each commutator $[C,D]$ in $\Lambda_\infty$ is a vertical translation whose
length is proportional to the area of the parallelogram spanned by 
the projections of $C$ and $D$ in ${\mathbb C}$; see page 446 of \cite{parkervolumes}
for example. 
With the above
normalization, it is clear that the parallelogram spanned by the projections
of  $A^{-1}B$ and $AB^{-1}$ has the smallest area among any positive
area parallelograms spanned by elements of $\Lambda_*$.  Hence,
the commutator subgroup of $\Lambda_\infty$ is generated by
$[A^{-1}B,AB^{-1}]$. It remains to write this
commutator as a power of the generator $T$ of $Z(\Gamma_\infty)$. We
split this into three cases:
\begin{enumerate}
\item[(1)] $p=3$: Since $A$ and $B$ have order 3, we see that
\begin{eqnarray*}
[A^{-1}B,AB^{-1}] &= & (A^{-1}B)(AB^{-1})(B^{-1}A)(BA^{-1}) \\
& = & A(ABABAB)A^{-1}=T.
\end{eqnarray*}
\item[(2)] $p=4$: In this case $A$ and $B$ have order 4 and $(AB)^2=(BA)^2$.
Therefore
\begin{eqnarray*}
[A^{-1}B,AB^{-1}] & = & (A^{-1}B)(AB^{-1})(B^{-1}A)(BA^{-1}) \\
& = & A^2(ABAB)BABA^{-1}=A(AB)^4A^{-1}=T^2.
\end{eqnarray*}
\item[(3)] $p=6$. In this case, $A$ and $B$ have order 6 and satisfy the classical braid relation. Therefore:
\begin{eqnarray*}
[A^{-1}B,AB^{-1}] & = & (A^{-1}BA)B^{-1}B^{-1}(ABA^{-1}) \\
& = & BAB^{-4}AB=(BAB)(BAB)=T.
\end{eqnarray*}
\end{enumerate}
\end{pf}

\medskip

We now give a formula for the volumes of certain cusp neighborhoods
associated to the group $\Gamma_\infty$. This follows the methods
in~\cite{parkervolumes}.

\begin{prop}\label{prop:from-jrp-vol}
  Let $\Gamma$ be a discrete subgroup of ${\rm SU}(2,1)$ and let
  $\Gamma_\infty$ be a parabolic subgroup of $\Gamma$ fixing a point
  of $\partial{\bf H}^2_{\mathbb C}$.  Let $\Lambda_\infty$ be the
  maximal lattice of Heisenberg translations in $\Gamma_\infty$ and
  let $m$ be the index of $\Lambda_\infty$ in $\Gamma_\infty$.  Let
  $T$ be a generator of $Z(\Gamma_\infty)$ and let $q$ be the positive
  integer so that the shortest non-trivial commutator in
  $\Lambda_\infty$ is $T^q$.  Let $C$ be any element of
  $\Gamma-\Gamma_\infty$.  If $B_\infty$ is any horoball that is precisely
  invariant under $\Gamma_\infty$ in $\Gamma$, then 
  $$
  Vol(\Gamma_\infty\backslash B_\infty) \le \frac{\bigl(3-{\rm tr}(TCTC^{-1})\bigr)q}{2m}.
  $$
\end{prop}

\begin{proof}
We construct a horoball $B'_\infty$ that intersects its image under $C$, and so is
not precisely invariant, and so that
$$
Vol(\Gamma_\infty\backslash B'_\infty) = \frac{\bigl(3-{\rm tr}(TCTC^{-1})\bigr)q}{2m}.
$$ 
Following the normalization in~\cite{parkervolumes}, we use the second
Hermitian form (denoted by $J_0$ in \cite{parkervolumes})
and we suppose that $\Gamma_\infty$
fixes $q_\infty$, which corresponds to $[1,0,0]^t$.  Without loss of
generality, suppose that $C(q_\infty)=q_o$ is the origin in Heisenberg
cordinates, which corresponds to $[0,0,1]^t$.  Let $B'_\infty$ be a
horoball based at $q_\infty$ and consider its image
$B'_o=C(B'_\infty)$ under $C$ based at the point $q_o$. Suppose that
the height of $B'_{\infty}$ is chosen so that $B'_\infty$ and $B'_o$
are disjoint, but their boundaries intersect in a single point. If $C$
has the form given in equation (1.3) of~\cite{parkervolumes}, this
condition is precisely that the height of $B'_\infty$ is
$h=2/|c|$. Again, following the normalization in~\cite{parkervolumes},
suppose that
  $$
  T=\left(\begin{matrix} 1 & 0 & it/2 \\ 0 & 1 & 0 \\ 0 & 0 & 1 \end{matrix}\right).
  $$ 
  Therefore, using the formula for $Vol(B'_\infty/\Gamma_\infty)$
  given on page 446 of~\cite{parkervolumes}, we have:
  $$
  Vol(\Gamma_\infty\backslash B'_\infty)
  =\frac{1}{2h^2}\cdot\frac{t^2q}{m}=\frac{|c|^2t^2q}{8m}.
  $$ 
  In order to express this in an invariant way, we want to write
  $|c|^2t^2$ in terms of traces.  Since we suppose that $C$ sends
  $q_\infty$ to $q_o$, this means that
  $$
  CTC^{-1}=\left(\begin{matrix} 1 & 0 & 0 \\ 0 & 1 & 0 \\ |c|^2it/2 & 0 & 1 \end{matrix}\right).
  $$
  Hence $|c|^2t^2/4=3-{\rm tr}(TCTC^{-1})$, which gives the result.
\end{proof}

\medskip

We want to apply this result in the case where $\Gamma$ is one of
${\mathcal S}(3,\sigma_1)$, ${\mathcal S}(4,\sigma_5)$ or $\Gamma(6,1/6)$;
the parabolic subgroup is $\Gamma_\infty=\langle R_1, R_2\rangle$ and the map $C$ is $J$.
We have already found the integers $m$ (which is $p$ in each case) and $q$ needed 
to apply the theorem. It remains to find ${\rm tr}(TCTC^{-1})$.

\begin{lem}
  Suppose that $R_1$, $R_2$, $u$ and $\tau$ are as given in Section
  \ref{sec:equilateral}.  If $R_1R_2$ is parabolic then
  $|\tau|^2=2-u^3-\bar{u}^3$.
\end{lem}

\begin{proof}
  The trace of $R_1R_2$ is $u(2-|\tau|^2)+\bar{u}^2$ and the
  intersection of the mirrors of $R_1$ and $R_2$, denoted $p_{12}$,
  corresponds to a $\bar{u}^2$-eigenvector ${\bf p}_{12}$.  In order
  for $R_1R_2$ to be parabolic, it must have a repeated eigenvalue
  $\bar{u}^2$ whose eigenspace is spanned by ${\bf p}_{12}$. In
  particular, the trace of $R_1R_2$ is $u^4+2\bar{u}^2$. The result
  follows by solving for $|\tau|^2$ in
  $$
  u(2-|\tau|^2)+\bar{u}^2={\rm tr}(R_1R_2)=u^4+2\bar{u}^2.
  $$
\end{proof}

\medskip

We will be interested in three cases:
\begin{enumerate}
\item[(1)] $p=3$: This means $u^3+\bar{u}^3=-1$ and $|\tau|^2=3$. 
\item[(2)] $p=4$: This means $u^3+\bar{u}^3=0$ and $|\tau|^2=2$. 
\item[(3)] $p=6$: This means $u^3+\bar{u}^3=1$ and $|\tau|^2=1$. 
\end{enumerate}

\begin{prop}
The center of $\Gamma_\infty=\langle R_1,R_2\rangle$ is generated by a scalar multiple of
$$
T_{12}=\left(\begin{matrix} 
1 & 0 & (2+u^6)\bigl((u^5-u^2)\bar{\tau}+u^2\tau^2\bigr) \\
0 & 1 & -(2+\bar{u}^6)\bigl((\bar{u}^5-\bar{u}^2)\tau+\bar{u}^2\bar{\tau}^2\bigr) \\
0 & 0 & 1 \end{matrix}\right).
$$
\end{prop}

\begin{proof}
We perform the calculation in each of the three cases.

\begin{enumerate}
\item[(1)] When $p=3$ the center is generated by $(R_1R_2)^3$. Just using
$|\tau|^2=3$, we obtain:
$$
(R_1R_2)^3
=\left(\begin{matrix} 
u^3 & 0 & \bar{u}^6(1-u^3)\bigl((u^5-u^2)\bar{\tau}+u^2\tau^2\bigr) \\
0 & u^3 & -u^3(1-\bar{u}^3)\bigl((\bar{u}^5-\bar{u}^2)\tau+\bar{u}^2\bar{\tau}^2\bigr) \\
0 & 0 & \bar{u}^6 \end{matrix}\right).
$$
Using $u^3+\bar{u}^3=-1$, we see that $u^3=\bar{u}^6$ and $1-u^3=2+u^6$, and so 
$(R_1R_2)^3=u^3T_{12}$.
\item[(2)] When $p=4$ the center is generated by $(R_1R_2)^2$. Just  using $|\tau|^2=2$,
we obtain:
$$
(R_1R_2)^2
=\left(\begin{matrix} 
-u^2 & 0 & \bar{u}^4\bigl((u^5-u^2)\bar{\tau}+u^2\tau^2\bigr) \\
0 & -u^2 & u^2\bigl((\bar{u}^5-\bar{u}^2)\tau+\bar{u}^2\bar{\tau}^2) \\
0 & 0 & \bar{u}^4 \end{matrix}\right).
$$
Using $u^3+\bar{u}^3=0$, we see that $-u^2=\bar{u}^4$ and $1=2+u^6$, and so 
$(R_1R_2)^2=-u^2T_{12}$. 
\item[(3)] When $p=6$ the center is generated by $(R_1R_2)^3$. Just using $|\tau|^2=1$,
we obtain:
$$
(R_1R_2)^3
=\left(\begin{matrix} 
-u^3 & 0 & \bar{u}^6(1+u^3)\bigl((u^5-u^2)\bar{\tau}+u^2\tau^2\bigr) \\
0 & -u^3 & u^3(1+\bar{u}^3)\bigl((\bar{u}^5-\bar{u}^2)\tau+\bar{u}^2\bar{\tau}^2) \\
0 & 0 & \bar{u}^6 \end{matrix}\right).
$$
Using $u^3+\bar{u}^3=1$, we see that $-u^3=\bar{u}^6$ and $1+u^3=2+u^6$, and so 
$(R_1R_2)^3=-u^3T_{12}$.
\end{enumerate}
\end{proof}

\begin{cor}\label{cor:vol-upper-bd}
Let $\Gamma_\infty=\langle R_1,R_2\rangle$. If $B_\infty$ is any horoball that 
is precisely invariant under $\Gamma_\infty$ in $\Gamma$ then
$$
Vol(\Gamma_\infty\backslash B_\infty)
\le \frac{q}{2p}\Bigl|(2+u^6)\bigl((u^5-u^2)\bar{\tau}+u^2\tau^2\bigr)\Bigr|^2
$$
where $q=1$ when $p=3$ or $6$ and $q=2$ when $p=4$.
\end{cor}

\begin{proof}
We apply Proposition \ref{prop:from-jrp-vol} with $C=J$. This means that 
$$
JT_{12}J^{-1}=T_{23}=
\left(\begin{matrix} 
1 & 0 & 0 \\
(2+u^6)\bigl((u^5-u^2)\bar{\tau}+u^2\tau^2\bigr) & 1 & 0 \\
-(2+\bar{u}^6)\bigl((\bar{u}^5-\bar{u}^2)\tau+\bar{u}^2\bar{\tau}^2\bigr) & 0 & 1 \end{matrix}\right).
$$
Hence 
$$
{\rm tr}(T_{12}T_{23})=3-\Bigl|(2+u^6)\bigl((u^5-u^2)\bar{\tau}+u^2\tau^2\bigr)\Bigr|^2.
$$
The result follows since we know $m=p$ and $q=1$ for $p=3,6$ and $q=2$ for $p=4$.
\end{proof}

\medskip

On the other hand, we have the following lower bound for $Vol(\Gamma_\infty\backslash B_\infty)$:

\begin{prop}\label{prop:vol-lower-bd}
Let $\Gamma_1$ be a lattice in ${\rm SU}(2,1)$. Suppose that
$\Gamma_1$ is an index $d$ subgroup of a lattice $\Gamma$ in ${\rm
  SU}(2,1)$.  Suppose that $\Gamma_1\backslash{\bf H}^2_{\mathbb C}$
has only one cusp and let $\Gamma_1^{\infty}$ be the corresponding
parabolic subgroup. If $B_1$ is the largest horoball that is
precisely invariant under $\Gamma_1^\infty$ in $\Gamma_1$ then
$$
Vol(\Gamma_1^\infty\backslash B_1)\ge d/4.
$$
\end{prop}

\begin{proof}
Let $\Gamma_\infty$ be a parabolic subgroup of $\Gamma$, which we may
and will assume contains $\Gamma_1^\infty$.  Since
$\Gamma_1\backslash{\bf H}^2_{\mathbb C}$ has only one cusp, the index
of $\Gamma_1^\infty$ in $\Gamma_\infty$ is the same as the index of
$\Gamma_1$ in $\Gamma$, namely $d$.  From Theorem 4.1
of~\cite{parkervolumes} we know that there is a horoball $B$ so that
$Vol(\Gamma_\infty\backslash B)\ge 1/4$.  Clearly a horoball precisely
invariant under $\Gamma_\infty$ in $\Gamma$ is also precisely
invariant under $\Gamma_1^\infty$ in $\Gamma_1$, and hence $B\subset
B_1$. Since $\Gamma_1^\infty$ has index $d$ in $\Gamma_\infty$, the
corresponding covering $\Gamma_1^\infty\backslash B\rightarrow
\Gamma_\infty\backslash B$ has degree $d$. This implies that
$$
Vol(\Gamma_1^\infty\backslash B_1)\geq Vol(\Gamma_1^\infty\backslash B)
= d \, Vol(\Gamma_\infty\backslash B)\geq d/4.
$$
\end{proof}

\medskip

Combining Corollary~\ref{cor:vol-upper-bd} and Proposition~\ref{prop:vol-lower-bd}
we get the following bound on the index $d$ of a lattice containing $\Gamma$: 
$$
d\le \frac{2q}{p}\Bigl|(2+u^6)\bigl((u^5-u^2)\bar{\tau}+u^2\tau^2\bigr)\Bigr|^2.
$$
For the groups we are interested in, this bound is:
\begin{enumerate}
\item[(1)] ${\mathcal S}(3,\sigma_1)$: 
$$
d\le \frac{2}{3}\Bigl|(2+u^6)\bigl((u^5-u^2)\bar{\tau}+u^2\tau^2\bigr)\Bigr|^2
=6+2\sqrt{6}<11.
$$
\item[(2)] ${\mathcal S}(4,\sigma_5)$:
$$
d\le \Bigl|(2+u^6)\bigl((u^5-u^2)\bar{\tau}+u^2\tau^2\bigr)\Bigr|^2
=\frac{7+\sqrt{5}+3\sqrt{3}+\sqrt{15}}{2}<10.
$$
\item[(3)] $\Gamma(6,1/6)$:
$$
d\le \frac{1}{3}\Bigl|(2+u^6)\bigl((u^5-u^2)\bar{\tau}+u^2\tau^2\bigr)\Bigr|^2
=2+\sqrt{3}<4.
$$
\end{enumerate}

\begin{proof} (Proposition \ref{prop:non-comm-cusped})
\begin{enumerate}
\item[(1)] Suppose that the groups ${\mathcal S}(3,\sigma_1)$ and ${\mathcal S}(6,\sigma_1)$ are commensurable. Then by Proposition~\ref{prop:commonsupgroup} we may assume that they are contained in a common lattice $\Gamma$, say with indices $d_1$ and $d_2$ respectively. The corresponding orbifold Euler characteristics are:
$$
\chi\bigl({\mathcal S}(3,\sigma_1)\backslash{\bf H}^2_{\mathbb C}\bigr)=2/9,\quad 
\chi\bigl({\mathcal S}(6,\sigma_1)\backslash{\bf H}^2_{\mathbb C}\bigr)=43/72.
$$ 
Therefore the orbifold Euler characteristic of the commensurator $\Gamma$ is
$$
\frac{2}{9d_1}=\frac{43}{72d_2}.
$$
That is, $43d_1=16d_2$ and so $d_1\ge 16$. This contradicts the above bound of 11 on the index of any lattice containing ${\mathcal S}(3,\sigma_1)$.
\item[(2)] Suppose the groups ${\mathcal S}(4,\sigma_5)$ and ${\mathcal T}(4,{\bf S}_2)$ 
are commensurable. Their orbifold Euler characteristics are:
$$
\chi\bigl({\mathcal S}(4,\sigma_5)\backslash{\bf H}^2_{\mathbb C}\bigr)=17/36,\quad 
\chi\bigl({\mathcal T}(4,{\bf S}_2)\backslash{\bf H}^2_{\mathbb C}\bigr)=1/3.
$$ 
Arguing as before, the index $d_1$ of ${\mathcal S}(4,\sigma_5)$ in its commensurator must
be at least 17, contradicting the above bound of 10.
\item[(3)] Suppose the groups $\Gamma(6,1/6)$ and ${\mathcal T}(4,{\bf E}_2)$ are commensurable. Their orbifold Euler characteristics are:
$$
\chi\bigl(\Gamma(6,1/6)\backslash{\bf H}^2_{\mathbb C}\bigr)=11/144,\quad 
\chi\bigl({\mathcal T}(4,{\bf E}_2)\backslash{\bf H}^2_{\mathbb C}\bigr)=17/32.
$$ 
As before, $d_1$ must be at least 11, contradicting the above bound of 4.
\end{enumerate}
\end{proof}

\begin{appendices}

\section{Combinatorial data, commensurability invariants, presentations} \label{sec:data}

For each of the lattice considered in this paper (where our algorithm
produces a fundamental domain), we list
\begin{enumerate}
  \item The type of the triangle group, i.e. six braid lengths
    together with the order of $Q=R_1R_2R_3$. We will use the notation
    $a,b,c; d,e,f; g$ for
    $$
    \br(2,3),\,\br(3,1),\,\br(1,2);\; 
    \br(1,\bar323),\,\br(1,23\bar2),\,\br(3,12\bar1);\; o(Q),
    $$
    where $o(Q)$ means the order of $Q$.
  \item The orbifold Euler characteristic, and basic commensurability
    invariants (adjoint trace field, cocompactness, arithmeticity and
    non-arithmeticity index);
  \item The values of the order $p$ of reflections such that the group
    is a lattice. Values in parentheses indicate that our algorithm
    fails to give a fundamental domain for that group (see
    section~\ref{sec:resultspoincare} for details of how the algorithm fails
    in each case);
  \item The rough structure of the invariant shell, in the form of a
    list of side representatives. Recall that $[k]\,a;\,b,\,c$ stands
    for a pyramid with a $k$-gon as its base, which occurs when
    $\br(b,c)=k$;
  \item A presentation in terms of generators and relations (for the
    sake of brevity and clarity, we write the braid relation
    $(ab)^{n/2}=(ba)^{n/2}$ as $\br_n(a,b)$). 
    We give slightly more relations than in the relations present in
    the natural presentations (see
    sections~\ref{sec:natural_presentations}), so that the reader
    can reconstruct the fundamental domain from the
    presentation. More specifically, we list every braid relation that
    corresponds to the apex of some side representative in the
    domain. If there is truncation of that apex for some values of
    $p$, then we also list a power relation next to the corresponding
    braid relation.
  \item A list of conjugacy classes of vertex stabilizers;
  \item A list of singular points in the quotient, with the type of (cyclic) singularity.
\end{enumerate}
Dashed tables list groups that we know to be lattices, but  
where our algorithm does not produce a fundamental domain. If that is
the case, and we know an alternative description for the group that
makes the algorithm work, we give it in the commensurability invariant
table.


\subsection{Sporadic $\sigma_1$} \label{app-s1}

\begin{center}
Triangle group type: 6,6,6; 4,4,4; 8
\end{center}

\begin{center}
Lattice for $p=3,4,6$.
\end{center}

\begin{center}
Commensurability invariants:\\
\begin{tabular}{|c|c|c|c|c|c|}\hline
  $p$ & $\chi^{orb}$  & $\Q(\tr{\rm Ad}\ \Gamma)$ & CM field                      & C?    & A? \\ \hline
  3   &     2/9       & $\Q(\sqrt{6})$            & $\Q(i\sqrt{2},i\sqrt{3})$     &  NC   &  NA(1)       \\
  4   &     7/16      & $\Q(\sqrt{2})$            & $\Q(i,\sqrt{2})$              &  NC   &  NA(1)      \\
  6   &     43/72     & $\Q(\sqrt{6})$            & $\Q(i\sqrt{2},i\sqrt{3})$     &  NC   &  NA(1)      \\\hline
\end{tabular}
\end{center}

\begin{center}
Presentations: \label{natural-presentation-s1}
\begin{equation*}
 \begin{array}{c}
   \left \langle\,R_1, R_2, R_3, J\vphantom{(R_1R_2R_3R_2^{-1})^{\frac{4p}{p-4}}}\,\right.\left\vert\,  
     R_1^p,\, 
     J^3,\, 
     (R_1J)^8,\,
     R_3=JR_2J^{-1}=J^{-1}R_1J,\, \qquad\qquad\qquad\qquad\qquad \right.\\
  \br_6(R_1,R_2),\, (R_1R_2)^{\frac{3p}{p-3}},\,
         \br_4(R_1,R_2R_3R_2^{-1}),\, 
                   (R_1R_2R_3R_2^{-1})^{\frac{4p}{p-4}}\\
         \left.\br_3(R_1,R_2R_3R_2R_3^{-1}R_2^{-1}),\br_3(R_1,R_3^{-1}R_2^{-1}R_3R_2R_3)\vphantom{(R_1R_2R_3R_2^{-1})^{\frac{4p}{p-4}}} \right\rangle
 \end{array}
\end{equation*}
\end{center}

\begin{center}
Combinatorics:\\
\begin{tabular}{|c|c|c|c|c|}
\hline
  Triangle & \#($P$-orb) & Top trunc.                                      & Top ideal \\
\hline
   $[6]\ 1;\ 2,\ 3$ &  8  & $p=4,6$                                          & $p=3$\\                   
   $[4]\ 2;\ 1,\ 23\bar2$ &  8  & $p=6$                                      & $p=4$\\
   $[3]\ 23\bar2;\ 1,\ 232\bar3\bar2$ & 8 &                                  & $p=6$\\
   $[3]\ 232\bar3\bar2;\ 1,\ \bar3\bar2323$ & 8 &                            & $p=6$\\
\hline
\end{tabular}
\end{center}
\begin{figure}[htbp]
  \begin{center}
    \begin{tabular}{c}
      $p=3$\\
    \end{tabular}
    \hfill
    \begin{minipage}[c]{0.2\textwidth}
      \includegraphics[width = \textwidth]{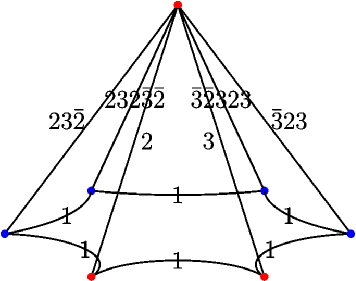}
    \end{minipage}\hfill
    \begin{minipage}[c]{0.2\textwidth}
      \includegraphics[width = 0.9\textwidth]{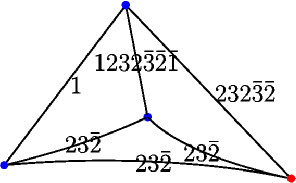}
    \end{minipage}\hfill
    \begin{minipage}[c]{0.2\textwidth}
      \includegraphics[width = \textwidth]{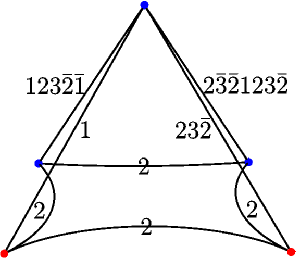}
    \end{minipage}\hfill
    \begin{minipage}[c]{0.2\textwidth}
      \includegraphics[width = \textwidth]{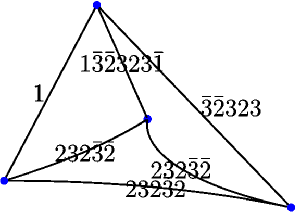}
    \end{minipage}
    \hfill\
  \end{center}
  \hfill
  \begin{center}
    \begin{tabular}{c}
      $p=4$\\
    \end{tabular}
    \hfill
    \begin{minipage}[c]{0.2\textwidth}
      \includegraphics[width=\textwidth]{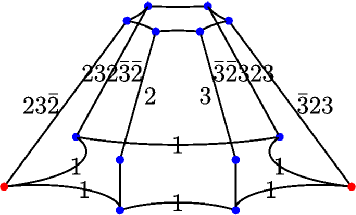}
    \end{minipage}\hfill
    \begin{minipage}[c]{0.2\textwidth}
      \includegraphics[width=0.9\textwidth]{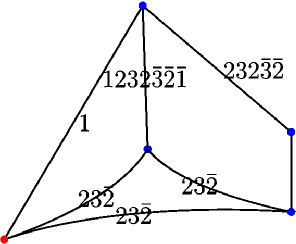}
    \end{minipage}\hfill
    \begin{minipage}[c]{0.2\textwidth}
      \includegraphics[width=\textwidth]{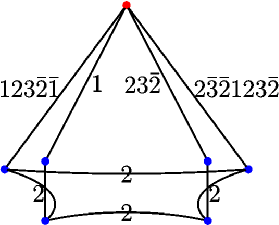}
    \end{minipage}\hfill
    \begin{minipage}[c]{0.2\textwidth}
      \includegraphics[width=\textwidth]{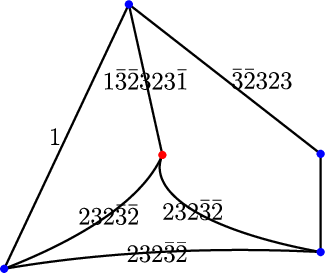}
    \end{minipage}\hfill\
  \end{center}
  
  \begin{center}
    \begin{tabular}{c}
      $p=6$\\
    \end{tabular}
    \hfill
    \begin{minipage}[c]{0.22\textwidth}
      \includegraphics[width=\textwidth]{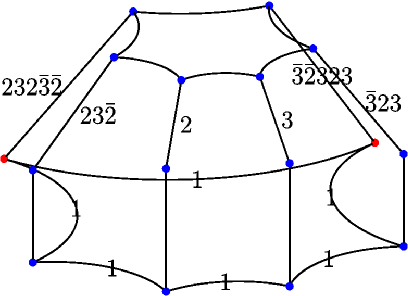}
    \end{minipage}\hfill
    \begin{minipage}[c]{0.22\textwidth}
      \centering
      \includegraphics[width=0.6\textwidth]{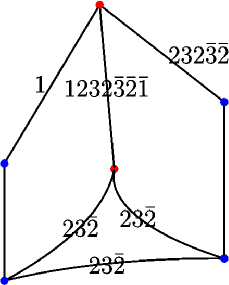}
    \end{minipage}\hfill
    \begin{minipage}[c]{0.22\textwidth}
      \centering
      \includegraphics[width=\textwidth]{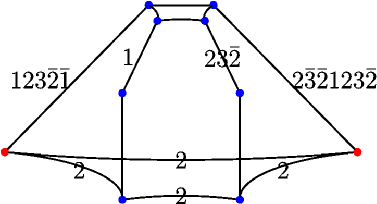}
    \end{minipage}\hfill
    \begin{minipage}[c]{0.22\textwidth}
      \includegraphics[width=0.8\textwidth]{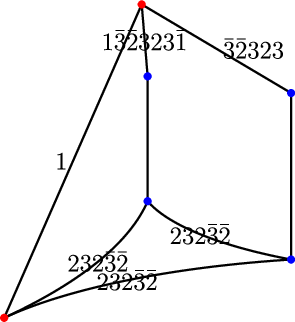}
    \end{minipage}\hfill\
  \end{center}
\end{figure}

\begin{center}
Vertex stabilizers:\\
\begin{tabular}{|c|ccc|ccc|}
\hline
  $p$   &   Vertex                &  Order   &   Nature                 &    Vertex                       &   Order    & Nature\\ 
  \hline
   3    &   $p_{12}$              & $\infty$ &    Cusp              &&&\\
        &   $p_{1,23\bar2}$       &   72     &    $G_5$              &&&\\
        &   $p_{1,232\bar3\bar2}$ &   24     &    $G_4$              
        &   $p_{1,\bar3\bar2323}$ &   24     &    $G_4$        \\      
 \hline
   4    &   $p_{1,(12)^3}$        &   16     &    $\Z_4\times\Z_4$
        &   $p_{1,(13)^3}$        &   16     &    $\Z_4\times\Z_4$\\
        &   $p_{1,23\bar2}$       & $\infty$ &    Cusp              &&&\\
        &   $p_{1,232\bar3\bar2}$ &   96     &    $G_8$
        &   $p_{1,\bar3\bar2323}$ &   96     &    $G_8$\\
\hline
  6    &   $p_{1,(12)^3}$        &   12     &    $\Z_6\times\Z_2$
       &   $p_{1,(13)^3}$        &   12     &    $\Z_6\times\Z_2$\\
       &   $p_{1,(123\bar2)^2}$  &   36     &    $\Z_6\times\Z_6$
       &   $p_{1,(1\bar323)^2}$  &   36     &    $\Z_6\times\Z_6$\\
       &   $p_{1,232\bar3\bar2}$ & $\infty$ &    Cusp
       &   $p_{1,\bar3\bar2323}$ & $\infty$ &    Cusp\\
 \hline
\end{tabular}
\end{center}

\begin{center}
  Singular points\\
  \begin{tabular}{|c|c|c|}
    \hline
    $p$ & Element &  Type\\
    \hline
    $3,4,6$   &  $J$    & $\frac{1}{3}(1,2)$\\
          &  $P$    & $\frac{1}{8}(1,3)$\\
    \hline
    $4,6$ & $R_1R_2$ & $\frac{1}{3}(1,1)$\\
    \hline
    $6$   & $R_1R_2R_3R_2^{-1}$ & $\frac{1}{2}(1,1)=A_1$\\
    \hline
  \end{tabular}
\end{center}


\subsection{Sporadic $\overline{\sigma}_4$} \label{app-s4c}

\begin{center}
Triangle group type: 4,4,4; 3,3,3; 7
\end{center}

\begin{center}
Lattice for $p=3,4,5,6,8,12$.
\end{center}

\begin{center}
Commensurability invariants:\\
\begin{tabular}{|c|c|c|c|c|c|}\hline
  $p$  & $\chi^{orb}$ & $\Q(\tr{\rm Ad}\ \Gamma)$          &  CM field                     &                  C?    & A? \\ \hline
  3    & 2/63         & $\Q(\sqrt{21})$                    & $\Q(i\sqrt{3},i\sqrt{7})$     & C    &  A           \\
  4    & 25/224       & $\Q(\sqrt{7})$                     & $\Q(i,\sqrt{7})$              & NC   &  NA(1)      \\
  5    & 47/280       & $\Q(\sqrt{\frac{5+\sqrt{5}}{14}})$ & $\Q(\zeta_5,i\sqrt{7})$       & C    &  NA(2)      \\
  6    & 25/126       & $\Q(\sqrt{21})$                    & $\Q(i\sqrt{3},i\sqrt{7})$     & NC   &  NA(1)      \\
  8    & 99/448       & $\Q(\sqrt{2},\sqrt{7})$            & $\Q(i,\sqrt{2},\sqrt{7})$     & C    &  NA(2)      \\
  12   & 221/1008     & $\Q(\sqrt{3},\sqrt{7})$            & $\Q(i,\sqrt{3},\sqrt{7})$     & C    &  NA(2)      \\\hline
\end{tabular}
\end{center}

\begin{center}
Presentations:
\begin{eqnarray*}
&\left\langle\,R_1, R_2, R_3, J\vphantom{(R_1R_2R_3R_2^{-1})^{\frac{6p}{p-6}}}\,\right.\left\vert\,  
     R_1^p,\, 
     J^3,\, 
     (R_1J)^7,\,
     R_3=JR_2J^{-1}=J^{-1}R_1J,\, 
     \right. \qquad\qquad\qquad\qquad\qquad & \\ 
&\left.\br_4(R_1,R_2),\,     (R_1R_2)^{\frac{4p}{p-4}},\, 
     \br_3(R_1,R_2R_3R_2^{-1}),\, (R_1R_2R_3R_2^{-1})^{\frac{6p}{p-6}} 
\,\right\rangle&
\end{eqnarray*}
\end{center}

\begin{center}
Combinatorics:\\
\begin{tabular}{|c|c|c|c|c|}
\hline
  Triangle & \#($P$-orb) & Top trunc.                                      & Top ideal \\
\hline
               $[4]\ 1;\ 2,\ 3$ &  7  & $p=5,6,8,12$                                     & $p=4$\\                   
               $[3]\ 2;\ 1,\ 23\bar2$ &  7 & $p=8,12$                                    & $p=6$\\
\hline
\end{tabular}
\end{center}

\begin{figure}[htbp]
\begin{center}
  \hfill
    \begin{tabular}{c}
      $p=3$\\
    \end{tabular}
    \begin{minipage}[c]{0.18\textwidth}
      \includegraphics[width = 0.8\textwidth]{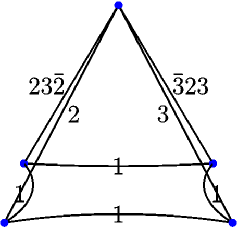}
    \end{minipage}
    \begin{minipage}[c]{0.18\textwidth}
      \includegraphics[width = \textwidth]{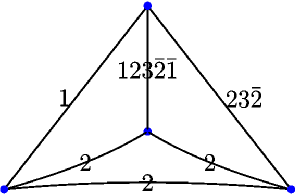}
    \end{minipage}
  \hfill\,
    \begin{tabular}{c}
      $p=4$\\
    \end{tabular}
    \begin{minipage}[c]{0.18\textwidth}
      \includegraphics[width = 0.8\textwidth]{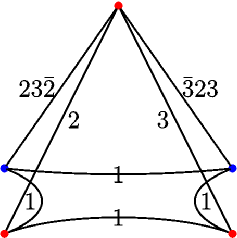}
    \end{minipage}
    \begin{minipage}[c]{0.18\textwidth}
      \includegraphics[width = \textwidth]{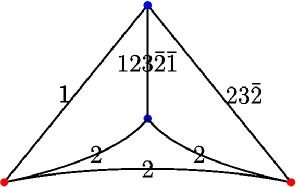}
    \end{minipage}
  \hfill\,
\end{center}

\begin{center}
  \hfill
    \begin{tabular}{c}
      $p=5$\\
    \end{tabular}
    \begin{minipage}[c]{0.18\textwidth}
      \includegraphics[width = \textwidth]{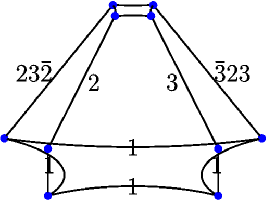}
    \end{minipage}
    \begin{minipage}[c]{0.18\textwidth}
      \centering
      \includegraphics[width = 0.8\textwidth]{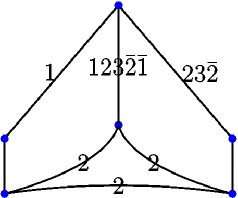}
    \end{minipage}
  \hfill\,
    \begin{tabular}{c}
      $p=6$\\
    \end{tabular}
    \begin{minipage}[c]{0.18\textwidth}
      \includegraphics[width = \textwidth]{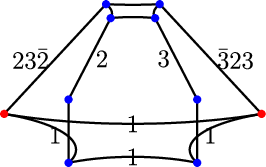}
    \end{minipage}
    \begin{minipage}[c]{0.18\textwidth}
      \centering
      \includegraphics[width = 0.8\textwidth]{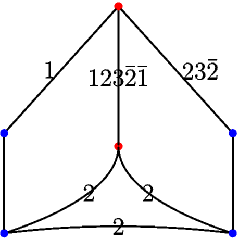}
    \end{minipage}
  \hfill\,
\end{center}

\begin{center}
  \hfill
  \begin{tabular}{c}
    $p=8$\\
  \end{tabular}
  \begin{minipage}[c]{0.18\textwidth}
    \includegraphics[width = \textwidth]{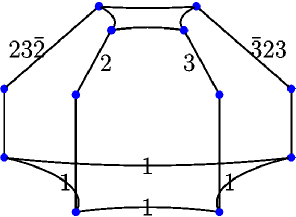}
  \end{minipage}
  \begin{minipage}[c]{0.18\textwidth}
    \centering
    \includegraphics[width = 0.6\textwidth]{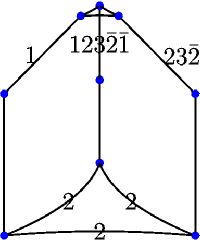}
  \end{minipage}
  \hfill\,
    \begin{tabular}{c}
      $p=12$\\
    \end{tabular}
    \begin{minipage}[c]{0.18\textwidth}
      \includegraphics[width = \textwidth]{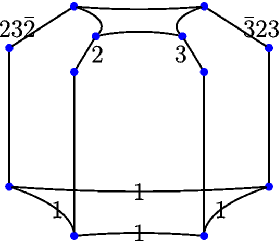}
    \end{minipage}
    \begin{minipage}[c]{0.18\textwidth}
      \centering
      \includegraphics[width = 0.6\textwidth]{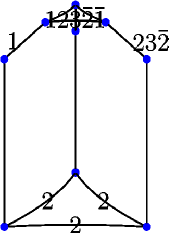}
    \end{minipage}
  \hfill\,
\end{center}
\end{figure}

\begin{center}
Vertex stabilizers:\\
\begin{tabular}{|c|ccc|ccc|}
\hline
  $p$   &   Vertex                &  Order   &   Nature                 &    Vertex                       &   Order    & Nature\\ 
  \hline
  3    &   $p_{12}$              &   72     &    $G_5$  &&&\\                 
       &   $p_{1,23\bar2}$       &   24     &    $G_4$ &&&\\
  \hline
  4    &   $p_{12}$              & $\infty$ &    Cusp  &&&\\                 
       &   $p_{1,23\bar2}$       &   96     &    $G_8$ &&&\\
 \hline
  5    &   $p_{1,(12)^2}$        &   50     &    $\Z_5\times\Z_{10}$
       &   $p_{1,(13)^2}$        &   50     &    $\Z_5\times\Z_{10}$\\
       &   $p_{1,23\bar2}$       &   600    &    $G_{16}$ &&&\\
 \hline
  6    &   $p_{1,(12)^2}$        &   36     &    $\Z_6\times\Z_{6}$
       &   $p_{1,(13)^2}$        &   36     &    $\Z_6\times\Z_{6}$\\
       &   $p_{1,23\bar2}$       & $\infty$ &    Cusp &&&\\
\hline
  8    &   $p_{1,(12)^2}$        &   32     &    $\Z_8\times\Z_4$
       &   $p_{1,(13)^2}$        &   32     &    $\Z_8\times\Z_4$\\
       &   $p_{1,(123\bar2)^3}$  &   64     &    $\Z_8\times\Z_8$&&&\\
\hline
  12   &   $p_{1,(12)^2}$        &   36     &    $\Z_{12}\times\Z_3$
       &   $p_{1,(13)^2}$        &   36     &    $\Z_{12}\times\Z_3$\\
       &   $p_{1,(123\bar2)^3}$  &   48     &    $\Z_{12}\times\Z_4$&&&\\
 \hline
\end{tabular}
\end{center}
\begin{center}
  Singular points:\\
  \begin{tabular}{|c|c|c|}
    \hline
    $p$ & Element &  Type\\
    \hline
    3,4,5,6,8,12  &  $J$                   & $\frac{1}{3}(1,2)$\\
    3,4,5,6,8,12  &  $P$                   & $\frac{1}{7}(1,3)$\\
    3,4,5,6,8,12  & $R_2R_3R_2^{-1}P^2$    & $A_1$\\
    5,6,8,12      & $R_2R_3$               & $A_1$\\
    8,12          & $R_1R_2R_3R_2^{-1}$    & $\frac{1}{3}(1,1)$\\
    8,12          & $R_1R_2R_3R_2^{-1}R_1$ & $A_1$\\
    \hline
  \end{tabular}
\end{center}


\subsection{Sporadic $\sigma_5$} \label{app-s5}

\begin{center}
Triangle group type: 4,4,4; 5,5,5; 30\\
$P^5$ is a complex reflection
\end{center}

\begin{center}
Lattice for $p=2,3,4$.
\end{center}

\begin{center}
Commensurability invariants:\\
\begin{tabular}{|c|c|c|c|c|c|}\hline
  $p$ & $\chi^{orb}$  & $\Q(\tr{\rm Ad}\ \Gamma)$ &  CM field                                                &  C?   & A? \\ \hline
 2    &     1/45      & $\Q(\sqrt{5})$            & $\Q(\sigma_5^3)=\Q(i\sqrt{3},\sqrt{5})$ &  C    &  A             \\
 3    &     49/180    & $\Q(\sqrt{5})$            & $\Q(i\sqrt{3},\sqrt{5})$                &  NC   &  NA(1)       \\
 4    &     17/36     & $\Q(\sqrt{3},\sqrt{5})$   & $\Q(i,\sqrt{3},\sqrt{5})$               &  NC   &  NA(3)       \\\hline
\end{tabular}
\end{center}

\begin{center}
Presentations:
{\small
\begin{eqnarray*}
&\left\langle\,R_1, R_2, R_3, J\vphantom{(R_1R_2)^{\frac{10p}{3p-10}}}\,\right.\left\vert 
     R_1^p,\, J^3,\, (R_1J)^{30},\, R_3=JR_2J^{-1}=J^{-1}R_1J,\, \br_4(R_1,R_2),\, \br_5(R_1,R_2R_3R_2^{-1}),\, \right. \qquad\qquad\qquad &\\ 
&\left. (R_1R_2R_3R_2^{-1})^{\frac{10p}{3p-10}},\, \br_6(R_2,R_3^{-1}R_2^{-1}R_1^{-1}R_2R_3R_2^{-1}R_1R_2R_3),\,
     (R_2\cdot R_3^{-1}R_2^{-1}R_1^{-1}R_2R_3R_2^{-1}R_1R_2R_3)^{\frac{3p}{p-3}} 
\,\right\rangle&
\end{eqnarray*}
}
\end{center}

\begin{center}
Combinatorics:\\
\begin{tabular}{|c|c|c|c|}
\hline
  Triangle                                                   & \#($P$-orb) & Top trunc.   & Top ideal \\
\hline
  $[4]\ 1;\ 2,\ 3$                                           &    30       &              & $p=4$\\              
  $[5]\ 2;\ 1,\ 23\bar2$                                     &    30       & $p=4$        & \\
  $[6]\ 2\bar3\bar2123\bar2;\ 2,\ \bar3\bar2\bar123\bar2123$ &     5       & $p=4$        & $p=3$\\
\hline
\end{tabular}
\end{center}

\begin{figure}[htbp]
\begin{center}
  \hfill
    \begin{tabular}{c}
      $p=2$\\
    \end{tabular}
    \hfill
    \begin{minipage}[c]{0.28\textwidth}
      \includegraphics[width=\textwidth]{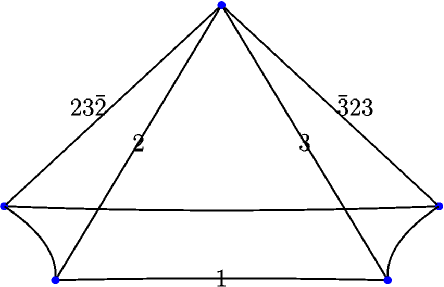}
    \end{minipage}
    \begin{minipage}[c]{0.28\textwidth}
      \includegraphics[width=\textwidth]{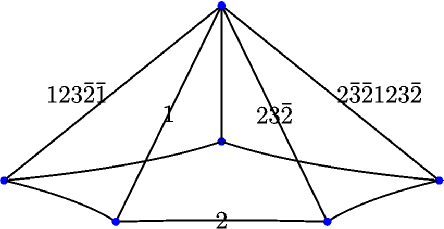}
    \end{minipage}
    \begin{minipage}[c]{0.28\textwidth}
      \includegraphics[width=\textwidth]{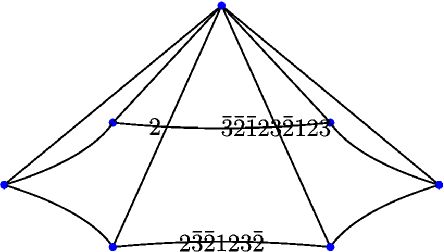}
    \end{minipage}
  \hfill\,
\end{center}

\begin{center}
  \hfill
    \begin{tabular}{c}
      $p=3$\\
    \end{tabular}
    \hfill
    \begin{minipage}[c]{0.28\textwidth}
      \includegraphics[width=\textwidth]{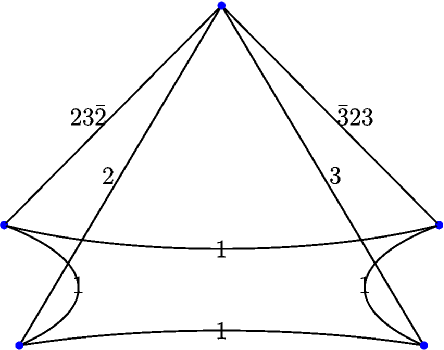}
    \end{minipage}
    \begin{minipage}[c]{0.28\textwidth}
      \includegraphics[width=\textwidth]{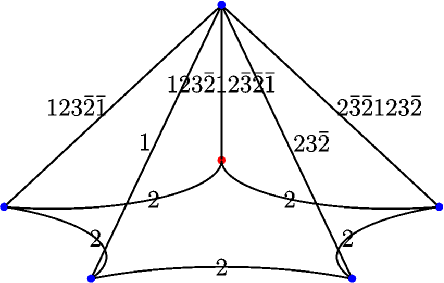}
    \end{minipage}
    \begin{minipage}[c]{0.28\textwidth}
      \includegraphics[width=\textwidth]{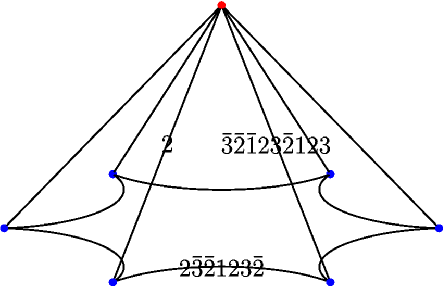}
    \end{minipage}
  \hfill\,
\end{center}

\begin{center}
  \hfill
    \begin{tabular}{c}
      $p=4$\\
    \end{tabular}
    \hfill
    \begin{minipage}[c]{0.28\textwidth}
      \includegraphics[width=\textwidth]{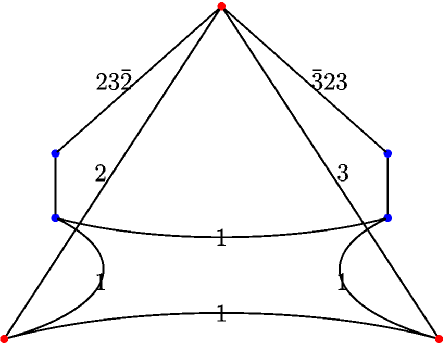}
    \end{minipage}
    \begin{minipage}[c]{0.28\textwidth}
      \includegraphics[width=\textwidth]{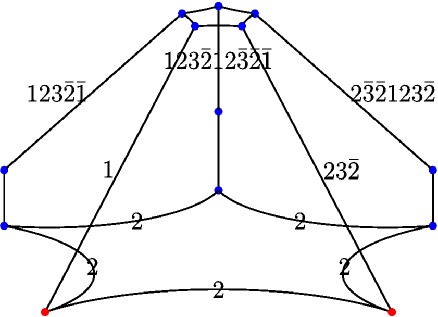}
    \end{minipage}
    \begin{minipage}[c]{0.28\textwidth}
      \includegraphics[width=\textwidth]{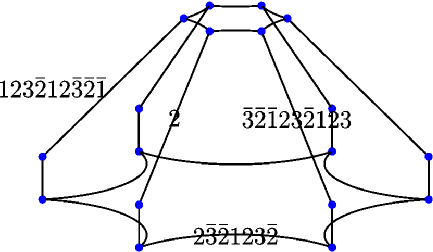}
    \end{minipage}
  \hfill\
\end{center}
\end{figure}

\begin{center}
Vertex stabilizers:\\
\begin{tabular}{|c|ccc|}
\hline
  $p$   &   Vertex                &  Order   &   Nature             \\ 
\hline
   2    &   $p_{12}$                                 &   8      &    $G(4,4,2)$ \\
        &   $p_{1,23\bar2}$                          &   10     &    $G(5,5,2)$ \\
        &   $p_{2,123\bar212\bar3\bar2\bar1}$                                &   72     &    $G(6,1,2)^{(*)}$ \\
\hline
   3    &   $p_{12}$                                 &   72     &    $G_{5}$ \\
        &   $p_{1,23\bar2}$                          &   360    &    $G_{20}$ \\
        &   $p_{2,123\bar212\bar3\bar2\bar1}$
                              & $\infty$ &    Cusp$^{(*)}$  \\
\hline
  4    &   $p_{12}$                                 & $\infty$ &    Cusp              \\
       &   $p_{1,(123\bar2)^5}$                      &    16    &    $\Z_4\times\Z_4$  \\
       &   $p_{2,(2\cdot 123\bar212\bar3\bar2\bar1)^3}$     &    48    &    ${\Z_{12}\times\Z_{4}}^{(*)}$  \\
\hline
\end{tabular}
\end{center}

\begin{center}
  Singular points:\\
  \begin{tabular}{|c|c|c|}
    \hline
     $p$ & Element &  Type\\
    \hline
      2,3,4  &  $J$                & $\frac{1}{3}(1,2)$\\
      2,3,4  &  $P$                & $\frac{1}{5}(1,2)$\\
    \hline
      4      & $123\bar2123\bar21$ & $A_1$\\
      4      & $P^5\bar2$          & $A_1$\\
      4      & $123\bar2$          & $\frac{1}{5}(1,3)$\\
    \hline
  \end{tabular}
\end{center}


\subsection{Sporadic $\sigma_{10}$} \label{app-s10}

\begin{center}
Triangle group type: 5,5,5; 3,3,3; 5
\end{center}

\begin{center}
Lattice for $p=3,4,5,10$.
\end{center}

\begin{center}
Commensurability invariants:\\
\begin{tabular}{|c|c|c|c|c|c|}\hline
  $p$ & $\chi^{orb}$  & $\Q(\tr{\rm Ad}\ \Gamma)$ & CM field                   & C? & A? \\ 
\hline 
 3    &     1/45      & $\Q(\sqrt{5})$            & $\Q(i\sqrt{3},\sqrt{5})$   &  C    &  A             \\
 4    &     3/32      & $\Q(\sqrt{5})$            & $\Q(i,\sqrt{5})$           &  C    &  A             \\
 5    &      1/8      & $\Q(\sqrt{5})$            & $\Q(\zeta_5)$              &  C    &  A             \\
 10   &     13/100    & $\Q(\sqrt{5})$            & $\Q(\zeta_5)$              &  C    &  A             \\
\hline
\end{tabular}
\end{center}

\begin{center}
Presentations:
\begin{eqnarray*}
&\left\langle\,R_1, R_2, R_3, J\vphantom{(R_1R_2)^{\frac{10p}{3p-10}}}\,\right.\left\vert\,  
     R_1^p,\, J^3,\,(R_1J)^5,\, R_3=JR_2J^{-1}=J^{-1}R_1J,\,\right. \qquad\qquad\qquad &\\ 
&\left.\br_5(R_1,R_2),\, (R_1R_2)^{\frac{10p}{3p-10}},\, \br_3(R_1,R_2R_3R_2^{-1}),\, (R_1R_2R_3R_2^{-1})^{\frac{6p}{p-6}} 
\,\right\rangle&
\end{eqnarray*}
\end{center}

\begin{center}
Combinatorics:\\
\begin{tabular}{|c|c|c|c|c|}
\hline
  Triangle & \#($P$-orb) & Top trunc.                                      & Top ideal \\
  $[5]\ 1;\ 2,\ 3$ &  5  & $p=4,5,10$                                       &\\              
  $[3]\ 2;\ 1,\ 23\bar2$ &  5 & $p=10$                                      &\\
\hline
\end{tabular}
\end{center}

\begin{figure}[htbp]
\begin{center}
  \hfill
    \begin{tabular}{c}
      $p=3$\\
    \end{tabular}
    \begin{minipage}[c]{0.18\textwidth}
      \includegraphics[width = \textwidth]{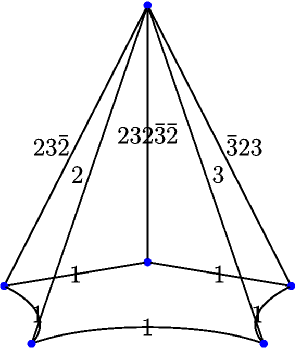}
    \end{minipage}
    \begin{minipage}[c]{0.18\textwidth}
      \includegraphics[width = \textwidth]{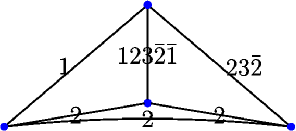}
    \end{minipage}
  \hfill\,
    \begin{tabular}{c}
      $p=4$\\
    \end{tabular}
    \begin{minipage}[c]{0.18\textwidth}
      \includegraphics[width = \textwidth]{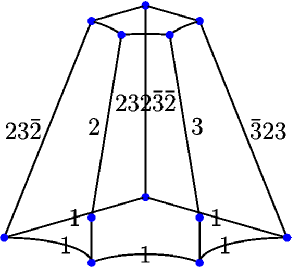}
    \end{minipage}
    \begin{minipage}[c]{0.18\textwidth}
      \includegraphics[width = \textwidth]{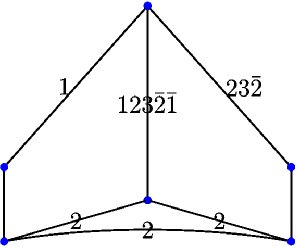}
    \end{minipage}
  \hfill\,
\end{center}

\begin{center}
  \hfill
    \begin{tabular}{c}
      $p=5$\\
    \end{tabular}
    \begin{minipage}[c]{0.182\textwidth}
      \includegraphics[width = \textwidth]{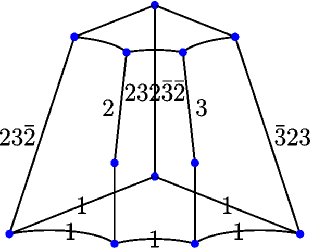}
    \end{minipage}
    \begin{minipage}[c]{0.182\textwidth}
      \centering
      \includegraphics[width = 0.7\textwidth]{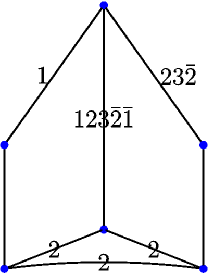}
    \end{minipage}
  \hfill\,
    \begin{tabular}{c}
      $p=10$\\
    \end{tabular}
    \begin{minipage}[c]{0.182\textwidth}
      \includegraphics[width = \textwidth]{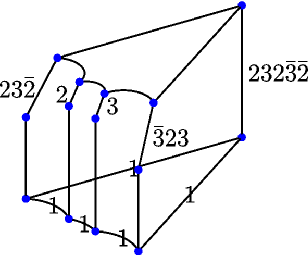}
    \end{minipage}
    \begin{minipage}[c]{0.182\textwidth}
      \centering
      \includegraphics[width = 0.4\textwidth]{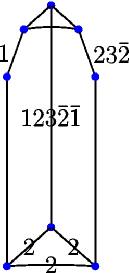}
    \end{minipage}
  \hfill\,
\end{center}
\end{figure}

\begin{center}
Vertex stabilizers:\\
\begin{tabular}{|c|ccc|ccc|}
\hline
$p$   &   Vertex                &  Order   &   Nature  \\ 
\hline
3    &   $p_{12}$                      &   360    &    $G_{20}$\\
     &   $p_{1,23\bar2}$               &   24     &    $G_4$   \\
     &   $p_{1,\bar3\bar2323}$         &   9      &    $\Z_3\times\Z_3$\\
\hline
4    &   $p_{1,(13)^5}$                &   16     &    $\Z_4\times\Z_4$\\
     &   $p_{1,23\bar2}$               &   96     &    $G_8$   \\
     &   $p_{1,\bar3\bar2323}$         &   16     &    $\Z_4\times\Z_4$\\
\hline
5    &   $p_{1,(13)^5}$                &   10     &    $\Z_5\times\Z_2$\\
     &   $p_{1,23\bar2}$               &   600    &    $G_{16}$   \\
     &   $p_{1,\bar3\bar2323}$         &   25     &    $\Z_5\times\Z_5$\\
\hline
10    &   $p_{1,(13)^5=id}$                &   10     &    $\Z_{10}$\\
      &   $p_{1,(123\bar2)^3}$          &   50     &    $\Z_{10}\times\Z_{5}$   \\
      &   $p_{1,\bar3\bar2323}$         &   100    &    $\Z_{10}\times\Z_{10}$\\
\hline
\end{tabular}
\end{center}

\begin{center}
Singular points:\\
  \begin{tabular}{|c|c|c|}
    \hline
    $p$ & Element &  Type\\
    \hline
    3,4,5  &  $Q$                & $\frac{1}{5}(1,2)$\\
           & $R_2Q^2$            & $A_1$\\
    \hline
    4,5    & $R_2Q^4$            & $\frac{1}{5}(1,2)$\\
    4,5    & $23\bar2(Q\bar2)^2$ & $A_1$\\
    \hline
     5     & $R_3Q^4$            & $A_1$\\
    \hline
  \end{tabular}
\end{center}


\subsection{Thompson ${\bf S_2}$} \label{app-S2}

\begin{center}
Triangle group type: 3,3,4; 5,5,5; 5 
\end{center}

\begin{center}
Lattice for $p=3,4,5$.
\end{center}

\begin{center}
Commensurability invariants:\\
\begin{tabular}{|c|c|c|c|c|c|}\hline
  $p$ & $\chi^{orb}$  & $\Q(\tr {\rm Ad}\ \Gamma)$ & CM field                  &  C?   & A?       \\ 
\hline
  3   & 2/15          & $\Q(\sqrt{5})$             & $\Q(i\sqrt{2},\sqrt{5})$  & C   &  A       \\
  4   & 1/3           & $\Q(\sqrt{3},\sqrt{5})$    & $\Q(i,\sqrt{3},\sqrt{5})$ & NC  &  NA(3)   \\
  5   & 133/300       & $\Q(\cos(2\pi/15))$        & $\Q(\zeta_{15})$          & C   &  NA(1)   \\
\hline
\end{tabular}
\end{center}

\begin{center}
Presentations:
\begin{eqnarray*}
&\left\langle\,R_1, R_2, R_3\vphantom{(R_1R_2R_3R_2^{-1})^{\frac{6p}{p-6}}}\,\right.\left\vert\,  
     R_1^p,\, R_2^p,\, R_3^p,\, 
     (R_1R_2R_3)^5,\, \br_3(R_1,R_3),\,\br_3(R_2,R_3)\, 
     \right. \qquad\qquad\qquad\qquad\qquad & \\ 
&\left.
     \br_4(R_1,R_2),\, (R_1R_2)^{\frac{4p}{p-4}},\,
     \br_5(R_1,R_2R_3R_2^{-1}),\, (R_1R_2R_3R_2^{-1})^{\frac{10p}{3p-10}} 
\,\right\rangle&
\end{eqnarray*}
\end{center}

\begin{center}
Combinatorics:\\
\begin{tabular}{|c|c|c|c|c|}
\hline
  Triangle & \#($P$-orb) & Top trunc.                                      & Top ideal \\
\hline
$[3]\ 1;\ 2,\ 3$ &  5  &                                                 &\\                   
$[3]\ 23\bar2;\ 1,\ 3$ &  5  &                                           &\\                   
$[4]\ 3;\ 1,\ 2$ &  5  & $p=5$                                           & $p=4$\\
$[5]\ 2;\ 1,\ 23\bar2$ &  5  & $p=4,5$                                   & \\
\hline
\end{tabular}
\end{center}

\begin{figure}[htbp]
\begin{center}
\hfill
\begin{tabular}{c}
$p=3$\\
\end{tabular}
\hfill
\begin{minipage}[c]{0.18\textwidth}
    \includegraphics[width = \textwidth]{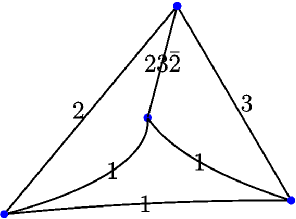}
  \end{minipage}\hfill
\begin{minipage}[c]{0.18\textwidth}
  \includegraphics[width = \textwidth]{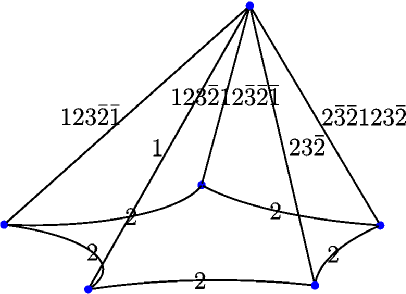}
\end{minipage}\hfill
\begin{minipage}[c]{0.18\textwidth}
  \includegraphics[width = \textwidth]{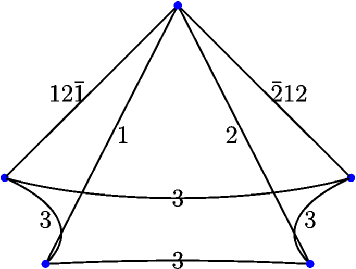}
\end{minipage}\hfill
\begin{minipage}[c]{0.18\textwidth}
  \includegraphics[width = \textwidth]{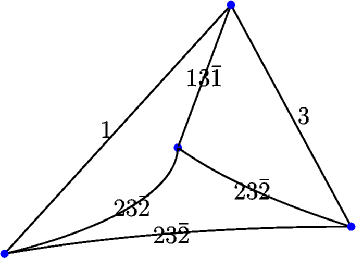}
\end{minipage}
\hfill\
\end{center}

\begin{center}
\hfill
\begin{tabular}{c}
$p=4$\\
\end{tabular}
\hfill
\begin{minipage}[c]{0.18\textwidth}
    \includegraphics[width = \textwidth]{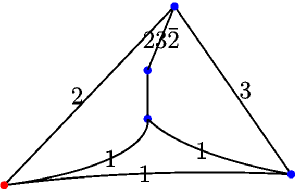}
  \end{minipage}\hfill
\begin{minipage}[c]{0.18\textwidth}
  \includegraphics[width = \textwidth]{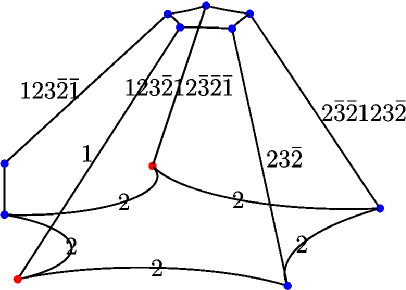}
\end{minipage}\hfill
\begin{minipage}[c]{0.18\textwidth}
  \includegraphics[width = \textwidth]{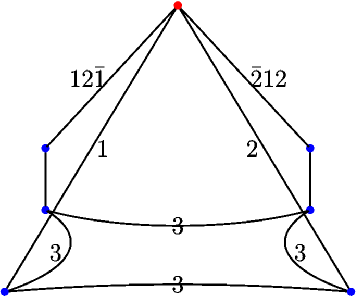}
\end{minipage}\hfill
\begin{minipage}[c]{0.18\textwidth}
  \includegraphics[width = \textwidth]{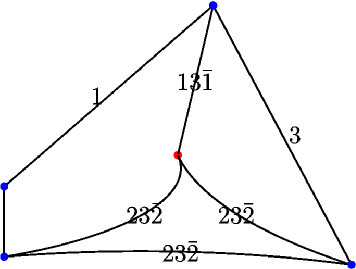}
\end{minipage}
\hfill\
\end{center}

\begin{center}
\hfill
\begin{tabular}{c}
$p=5$\\
\end{tabular}
\hfill
\begin{minipage}[c]{0.18\textwidth}
    \includegraphics[width = \textwidth]{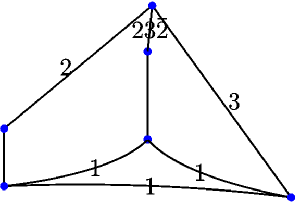}
  \end{minipage}\hfill
\begin{minipage}[c]{0.18\textwidth}
  \includegraphics[width = \textwidth]{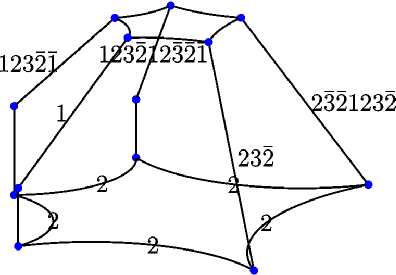}
\end{minipage}\hfill
\begin{minipage}[c]{0.18\textwidth}
  \includegraphics[width = \textwidth]{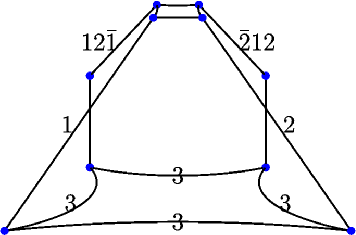}
\end{minipage}\hfill
\begin{minipage}[c]{0.18\textwidth}
  \includegraphics[width = \textwidth]{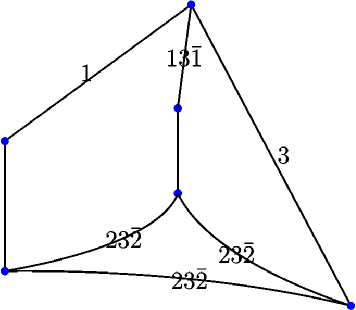}
\end{minipage}
\hfill\
\end{center}
\end{figure}

\begin{center}
Vertex stabilizers:\\
\begin{tabular}{|c|ccc|ccc|}
\hline
 $p$   &   Vertex             &  Order   &   Nature                 &    Vertex                       &   Order    & Nature\\ 
  \hline
  3    &   $p_{13}$              &   24     &    $G_4$
       &   $p_{23}$              &   24     &    $G_4$\\
       &   $p_{12}$              &   72     &    $G_5$           &&&\\     
       &   $p_{1,23\bar2}$       &   360    &    $G_{20}$        &&&\\
\hline
  4    &   $p_{13}$              &   96     &    $G_8$
       &   $p_{23}$              &   96     &    $G_8$\\ 
       &   $p_{12}$              & $\infty$ &    Cusp           &&&\\
       &   $p_{1,(123\bar2)^5}$  &   16     &    $\Z_4\times\Z_4$  &&&\\
\hline
  5    &   $p_{13}$              &  600     &    $G_{16}$
       &   $p_{23}$              &  600     &    $G_{16}$\\
       &   $p_{1,(12)^2}$        &   50     &    $\Z_5\times\Z_{10}$
       &   $p_{2,(12)^2}$        &   50     &    $\Z_5\times\Z_{10}$\\
       &   $p_{1,(123\bar2)^5}$  &   10     &    $\Z_5\times\Z_2$  &&&\\
 \hline
\end{tabular}
\end{center}

\begin{center}
Singular points:\\
  \begin{tabular}{|c|c|c|}
    \hline
   $p$ & Element             &  Type\\
    \hline
    3,4,5  &  $Q$                & $\frac{1}{5}(1,2)$\\
           & $R_2Q^2$            & $A_1$\\
    \hline
    4,5    & $R_2Q^4$            & $\frac{1}{5}(1,2)$\\
    4,5    & $23\bar2(Q\bar2)^2$ & $A_1$\\
    \hline
    5      & $R_3Q^4$            & $A_1$\\
    \hline
  \end{tabular}
\end{center}


\subsection{Thompson ${\bf E_2}$} \label{app-E2}

\begin{center}
Triangle group type: 3,4,4; 4,4,6; 6\\
$Q^3$ is a complex reflection
\end{center}

\begin{center}
Lattice for $p=3,4,6,(12)$.
\end{center}

\begin{center}
Commensurability invariants:\\
\begin{tabular}{|c|c|c|c|c|c|}\hline
  $p$ & $\chi^{orb}$  & $\Q(\tr {\rm Ad}\ \Gamma)$ & CM field         & C? & A? \\ 
\hline
 3    &     1/4       & $\Q$                       & $\Q(i\sqrt{3})$  &  NC  &  A     \\
 4    &     17/32     & $\Q(\sqrt{3})$             & $\Q(\zeta_{12})$ &  NC  &  NA(1) \\
 6    &     3/4       & $\Q$                       & $\Q(i\sqrt{3})$  &  NC  &  A     \\
\hline
\end{tabular}
\begin{tabular}{:c:c:c:c:c:c:c:}\hdashline
  $p$ & $\chi^{orb}$  & $\Q(\tr {\rm Ad}\ \Gamma)$ & CM field         & C? & A? & Alt.\\ 
\hdashline
 12   &     1/4       & $\Q(\sqrt{3})$             & $\Q(\zeta_{12})$ & C  & A  &  ?   \\
\hdashline
\end{tabular}
\end{center}

\begin{center}
Presentations ($p=3,4,6$ only):
{\small
\begin{equation*}
\begin{array}{c}
\left\langle\,R_1, R_2, R_3\vphantom{(R_1R_2R_3R_2^{-1})^{\frac{6p}{p-6}}}\,\right.\left\vert\,  
     R_1^p,\, R_2^p,\, R_3^p,\, 
     (R_1R_2R_3)^6,\, \br_3(R_2,R_3),\,
     \br_4(R_3,R_1), \,(R_1R_3)^{\frac{4p}{p-4}},\, \br_4(R_1,R_2),\right.\\ 
     \qquad\qquad\left.\, (R_1R_2)^{\frac{4p}{p-4}},\, \br_4(R_1,R_2R_3R_2^{-1}),\, (R_1R_2R_3R_2^{-1})^{\frac{4p}{p-4}}, \br_6(R_3,R_1R_2R_1^{-1})\, (R_3R_1R_2R_1^{-1})^{\frac{3p}{p-3}} \,\right\rangle
\end{array}
\end{equation*}
}
\end{center}

\begin{center}
Combinatorics:\\
\begin{tabular}{|c|c|c|c|c|}
\hline
  Triangle & \#($P$-orb) & Top trunc.                                      & Top ideal \\
\hline
$[3]\ 1;\ 2,\ 3$ &  6  & $p=12$                                          & $p=6$\\                   
$[4]\ 23\bar2;\ 1,\ 3$ &  6  & $p=6,12$                                  & $p=4$\\                   
$[4]\ 3;\ 1,\ 2$ &  6  & $p=6,12$                                        & $p=4$\\
$[4]\ 2;\ 1,\ 23\bar2$ &  6  & $p=6,12$                                  & $p=4$\\
$[6]\ \bar313;\ 12\bar1,\ 3$ &  3  & $p=4,6,12$                          & $p=3$\\
\hline
\end{tabular}
\end{center}

\begin{figure}[htbp]
\begin{center}
\hfill
\begin{tabular}{c}
$p=3$\\
\end{tabular}
\hfill
\begin{minipage}[c]{0.18\textwidth}
  \centering
  \includegraphics[width = 0.8\textwidth]{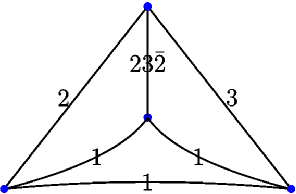}
  \end{minipage}\hfill
\begin{minipage}[c]{0.18\textwidth}
  \centering
  \includegraphics[width = 0.9\textwidth]{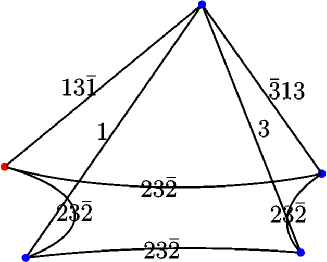}
\end{minipage}\hfill
\begin{minipage}[c]{0.18\textwidth}
  \centering
  \includegraphics[width = 0.9\textwidth]{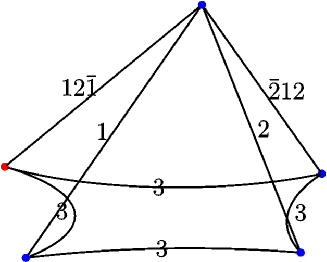}
\end{minipage}\hfill
\begin{minipage}[c]{0.18\textwidth}
  \centering
  \includegraphics[width = 0.9\textwidth]{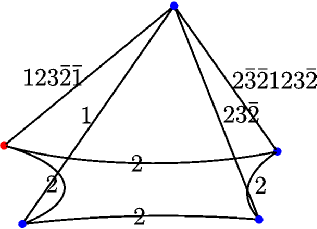}
\end{minipage}\hfill
\begin{minipage}[c]{0.18\textwidth}
  \centering
  \includegraphics[width = \textwidth]{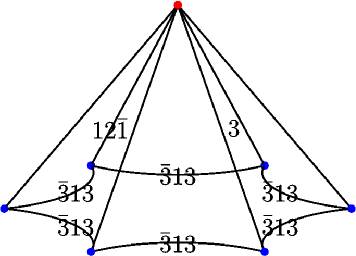}
\end{minipage}
\hfill\
\end{center}

\begin{center}
\hfill
\begin{tabular}{c}
$p=4$\\
\end{tabular}
\hfill
\begin{minipage}[c]{0.18\textwidth}
  \centering
  \includegraphics[width = 0.8\textwidth]{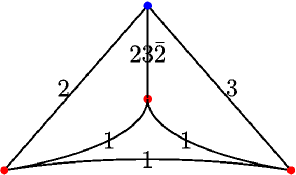}
  \end{minipage}\hfill
\begin{minipage}[c]{0.18\textwidth}
  \centering
  \includegraphics[width = 0.9\textwidth]{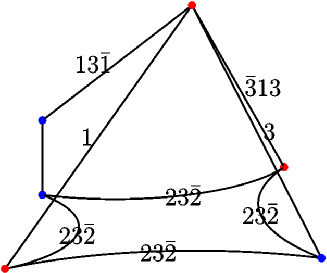}
\end{minipage}\hfill
\begin{minipage}[c]{0.18\textwidth}
  \centering
  \includegraphics[width = 0.9\textwidth]{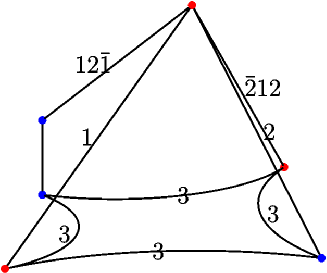}
\end{minipage}\hfill
\begin{minipage}[c]{0.18\textwidth}
  \centering
  \includegraphics[width = 0.9\textwidth]{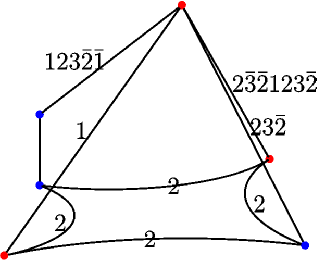}
\end{minipage}\hfill
\begin{minipage}[c]{0.18\textwidth}
  \centering
  \includegraphics[width = \textwidth]{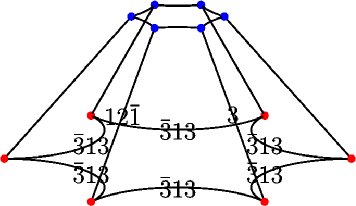}
\end{minipage}
\hfill\
\end{center}

\begin{center}
\hfill
\begin{tabular}{c}
$p=6$\\
\end{tabular}
\hfill
\begin{minipage}[c]{0.18\textwidth}
  \centering
  \includegraphics[width = 0.8\textwidth]{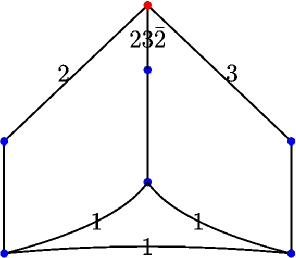}
  \end{minipage}\hfill
\begin{minipage}[c]{0.18\textwidth}
  \centering
  \includegraphics[width = 0.9\textwidth]{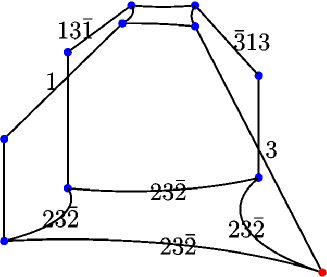}
\end{minipage}\hfill
\begin{minipage}[c]{0.18\textwidth}
  \centering
  \includegraphics[width = 0.9\textwidth]{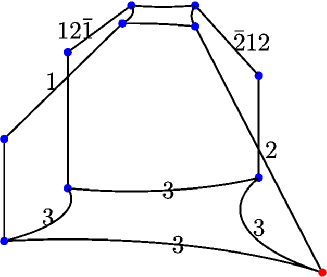}
\end{minipage}\hfill
\begin{minipage}[c]{0.18\textwidth}
  \centering
  \includegraphics[width = 0.9\textwidth]{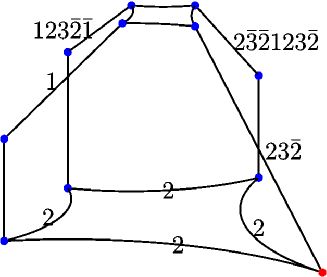}
\end{minipage}\hfill
\begin{minipage}[c]{0.18\textwidth}
  \centering
  \includegraphics[width = \textwidth]{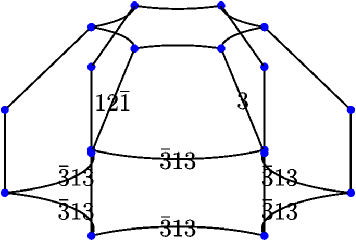}
\end{minipage}
\hfill\
\end{center}

\begin{center}
\hfill
\begin{tabular}{c}
$p=12$\\
\end{tabular}
\hfill
\begin{minipage}[c]{0.16\textwidth}
  \centering
  \includegraphics[width = 0.7\textwidth]{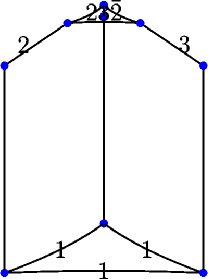}
  \end{minipage}\hfill
\begin{minipage}[c]{0.16\textwidth}
  \centering
  \includegraphics[width = 0.8\textwidth]{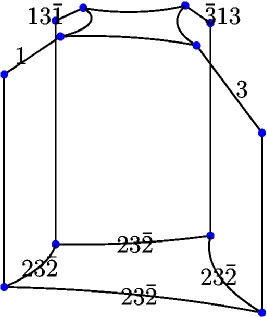}
\end{minipage}\hfill
\begin{minipage}[c]{0.16\textwidth}
  \centering
  \includegraphics[width = 0.8\textwidth]{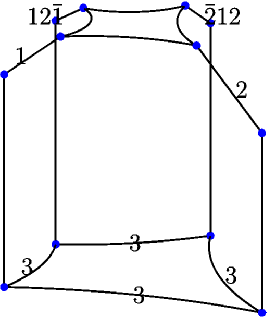}
\end{minipage}\hfill
\begin{minipage}[c]{0.16\textwidth}
  \centering
  \includegraphics[width = 0.8\textwidth]{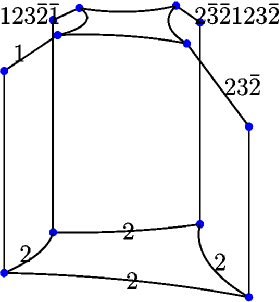}
\end{minipage}\hfill
\begin{minipage}[c]{0.16\textwidth}
  \centering
  \includegraphics[width = \textwidth]{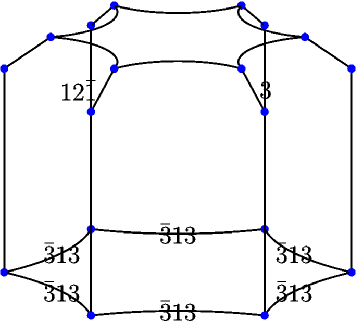}
\end{minipage}
\hfill\
\end{center}
\end{figure}

\begin{center}
Vertex stabilizers:\\
\begin{tabular}{|c|ccc|ccc|ccc|}
\hline
  $p$   &   Vertex             &  Order   &   Nature                 &    Vertex                       &   Order    & Nature  &    Vertex                       &   Order    & Nature\\ 
  \hline
   3    &   $p_{23}$              &   24     &    $G_4$ &&&&&&\\ 
 	&   $p_{1,23\bar2}$       &   72     &    $G_5$
        &   $p_{12}$              &   72     &    $G_5$
        &   $p_{13}$              &   72     &    $G_5$\\  
        &  $p_{2,123\bar2\bar1}$  & $\infty$ &    Cusp$^{(*)}$    &&&&&&\\
\hline
   4    &   $p_{23}$                 &   96     &    $G_8$ &&&&&&\\
 	&   $p_{1,23\bar2}$          & $\infty$ &    Cusp
        &   $p_{12}$                 & $\infty$ &    Cusp
        &   $p_{13}$                 & $\infty$ &    Cusp\\  
        & $p_{2,(123\bar2\bar1\cdot 2)^3}$ &   16     &    $\Z_4\times\Z_4$  &&&&&&\\
\hline
   6    & $p_{23}$ & $\infty$ & Cusp &&&&&&\\
 	&   $p_{1,(123\bar2)^2}$           & 36  &   $\Z_6\times\Z_6$ 
        &   $p_{1,(12)^2}$                 & 36  &   $\Z_6\times\Z_6$ 
        &   $p_{1,(13)^2}$                 & 36  &   $\Z_6\times\Z_6$ \\
        &   $p_{2,(123\bar2)^2}$           & 36  &   $\Z_6\times\Z_6$ 
        &   $p_{2,(12)^2}$                 & 36  &   $\Z_6\times\Z_6$ 
        &   $p_{3,(13)^2}$                 & 36  &   $\Z_6\times\Z_6$ \\
        & $p_{2,(123\bar2\bar1\cdot 2)^3}$ & 12  &   $\Z_6\times\Z_2$ &&&&&&\\
\hline
\end{tabular}
\end{center}

\begin{center}
Singular points:\\
  \begin{tabular}{|c|c|c|}
    \hline
    $p$ & Element &  Type\\
    \hline
     3,4,6  &  $Q$            & $\frac{1}{3}(1,1)$\\
    \hline
     4,6   & $Q\bar313Q\bar2$ & $A_1$\\
     4,6   & $\langle Q^{-1}2\bar3\bar2123\bar2,Q^3\rangle$  &  $\frac{1}{3}(1,1)$\\
    \hline
     6   & $R_2Q^5$         & $A_1$\\
     6   & $R_3Q^5$         & $A_1$\\
     6   & $23\bar2Q^5$     & $A_1$\\
    \hline
  \end{tabular}
\end{center}


\subsection{Thompson ${\bf \overline{H}_1}$} \label{app-H1}

\begin{center}
Triangle group type: 3,3,4; 7,7,7; 42\\
$Q^3$ is a complex reflection
\end{center}

\begin{center}
Lattice for $p=2,(7)$.
\end{center}

\begin{center}
Commensurability invariants:\\
\begin{tabular}{|c|c|c|c|c|c|}\hline
  $p$ & $\chi^{orb}$  & $\Q(\tr {\rm Ad}\ \Gamma)$ & CM field      & C? & A? \\ 
\hline
  2   &     1/49      & $\Q(\cos(2\pi/7))$         & $\Q(\zeta_7)$ & C  & A  \\
\hline
\end{tabular}
\begin{tabular}{:c:c:c:c:c:c:c:}\hdashline
  $p$ & $\chi^{orb}$  & $\Q(\tr {\rm Ad}\ \Gamma)$ & CM field      & C? & A? & Alt.\\ 
\hdashline
  7   &     1/49      & $\Q(\cos(2\pi/7))$         & $\Q(\zeta_7)$ & C  & A  & $\Gamma(7,9/14)$    \\
\hdashline
\end{tabular}
\end{center}

\begin{center}
Presentations ($p=2$ only):
\begin{eqnarray*}
&\left\langle\,R_1, R_2, R_3\vphantom{(R_1R_2R_3)^{42}}\,\right.\left\vert\,  
     R_1^2,\, R_2^2,\, R_3^2,\, 
     (R_1R_2R_3)^{42},\, (R_1R_2R_3R_2^{-1})^7\right.&\\ 
&\left.\br_3(R_2,R_3),\,\br_3(R_3,R_1),\, \br_4(R_1,R_2)\,\right\rangle&
\end{eqnarray*}
\end{center}

\begin{center}
Combinatorics:\\
\begin{tabular}{|c|c|c|c|c|}
\hline
  Triangle & \#($P$-orb) & Top trunc.                                      & Top ideal \\
\hline
$[3]\ 1;\ 2,\ 3$ &  42 & $p=7$                                           &\\                   
$[3]\ 23\bar2;\ 1,\ 3$ &  42  & $p=7$                                    &\\                   
$[4]\ 3;\ 1,\ 2$ &  42  &$p=7$                                           &\\
$[7]\ 2;\ 1,\ 23\bar2$ &  42  &$p=7$                                     &\\
$[14]\ \bar2123\bar2\bar12;\ \bar212,\ \bar312\bar13$ &  3  &$p=7$       &\\
\hline
\end{tabular}
\end{center}

\begin{figure}[htbp]

\begin{center}
\hfill
\begin{tabular}{c}
$p=2$\\
\end{tabular}
\hfill
\begin{minipage}[c]{0.18\textwidth}
  \centering
  \includegraphics[width = 0.7\textwidth]{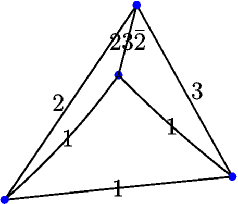}
  \end{minipage}\hfill
\begin{minipage}[c]{0.18\textwidth}
  \centering
  \includegraphics[width = 0.7\textwidth]{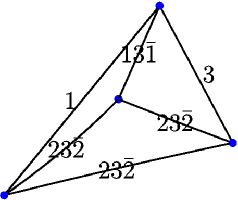}
\end{minipage}\hfill
\begin{minipage}[c]{0.18\textwidth}
  \centering
  \includegraphics[width = 0.95\textwidth]{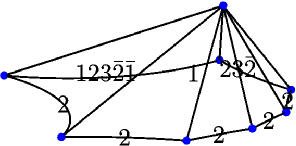}
\end{minipage}\hfill
\begin{minipage}[c]{0.18\textwidth}
  \centering
  \includegraphics[width = 0.95\textwidth]{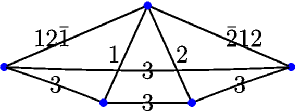}
\end{minipage}\hfill
\begin{minipage}[c]{0.18\textwidth}
  \centering
  \includegraphics[width = 0.9\textwidth]{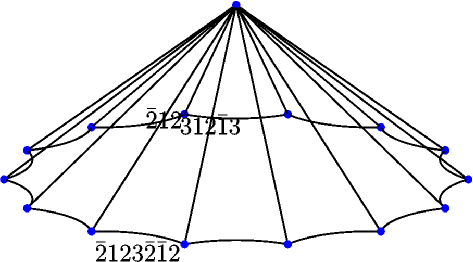}
\end{minipage}
\hfill\
\end{center}

\begin{center}
\hfill
\begin{tabular}{c}
$p=7$\\
\end{tabular}
\hfill
\begin{minipage}[c]{0.18\textwidth}
  \centering
  \includegraphics[width = 0.7\textwidth]{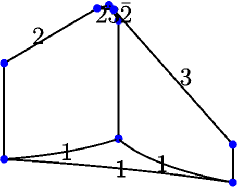}
  \end{minipage}\hfill
\begin{minipage}[c]{0.18\textwidth}
  \centering
  \includegraphics[width = 0.6\textwidth]{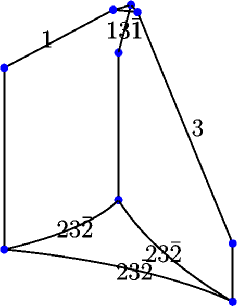}
\end{minipage}\hfill
\begin{minipage}[c]{0.18\textwidth}
  \centering
  \includegraphics[width = 0.8\textwidth]{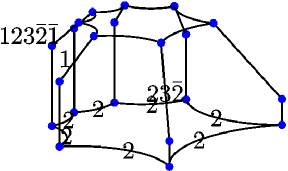}
\end{minipage}\hfill
\begin{minipage}[c]{0.18\textwidth}
  \centering
  \includegraphics[width = 0.8\textwidth]{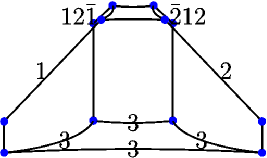}
\end{minipage}\hfill
\begin{minipage}[c]{0.18\textwidth}
  \centering
  \includegraphics[width = 0.8\textwidth]{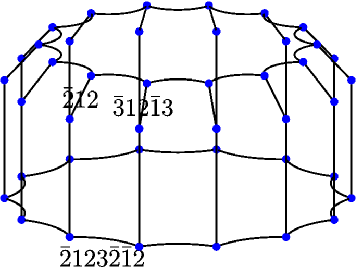}
\end{minipage}
\hfill\
\end{center}
\end{figure}

\begin{center}
Vertex stabilizers:\\
\begin{tabular}{|c|ccc|ccc|}
\hline
 $p$   &   Vertex             &  Order   &   Nature                 &    Vertex                       &   Order    & Nature \\ 
  \hline
  2    &   $p_{13}$              &   6      &    $G(3,3,2)$
       &   $p_{23}$              &   6      &    $G(3,3,2)$\\ 
       &   $p_{12}$              &   8      &    $G(4,4,2)$  &&&\\  
       &   $p_{1,23\bar2}$       &   14     &    $G(7,7,2)$  &&&\\
       &  $p_{2,123\bar212\bar3\bar2\bar1}$  & 392 &  $G(14,1,2)^{(*)}$    &&&\\
\hline
\end{tabular}
\end{center}

\begin{center}
Singular points:\\
  \begin{tabular}{|c|cccc|}
    \hline
    $p$ & Element & Order & Eigenvalues  &  Type\\
    \hline
    2  &  $Q$             &   42   &  $(\omega,\zeta_{42})$                  & $\frac{1}{3}(1,1)$\\
    \hline
  \end{tabular}
\end{center}


\subsection{Thompson ${\bf H_2}$} \label{app-H2}

\begin{center}
Triangle group type: 3,3,5; 5,5,5; 15\\
$Q^3$ is a complex reflection
\end{center}

\begin{center}
Lattice for $p=2,3,5,(10),(-5)
$.
\end{center}

\begin{center}
Commensurability invariants:\\
\begin{tabular}{|c|c|c|c|c|c|}\hline
  $p$ & $\chi^{orb}$  & $\Q(\tr {\rm Ad}\ \Gamma)$ & CM field         & C?  &  A?     \\ 
\hline
2     &    1/100      & $\Q(\sqrt{5})$             & $\Q(\zeta_5)$    & C   &  A      \\
3     &    26/75      & $\Q(\cos(2\pi/15))$        & $\Q(\zeta_{15})$ & C   &  NA(1)  \\
5     &    73/100     & $\Q(\sqrt{5})$             & $\Q(\zeta_5)$    & C   &  A      \\
\hline
\end{tabular}
\begin{tabular}{:c:c:c:c:c:c:c:}\hdashline
  $p$ & $\chi^{orb}$  & $\Q(\tr {\rm Ad}\ \Gamma)$ & CM field         & C? & A? & Alt.\\ 
\hdashline
  10  &    13/100     & $\Q(\sqrt{5})$             & $\Q(\zeta_5)$    & C  & A  & $\S(10,\sigma_{10})$      \\
\hdashline
&&&&&&\\[-13pt]
 -5   &     1/200     & $\Q(\sqrt{5})$             & $\Q(\zeta_5)$    & C  & A  & $\Gamma(5,7/10)$ \\
\hdashline
\end{tabular}
\end{center}

\begin{center}
Presentations ($p=2,3,5$ only):
\begin{eqnarray*}
&\left\langle\,R_1, R_2, R_3\vphantom{(R_1R_2R_3R_2^{-1})^{\frac{6p}{p-6}}}\,\right.\left\vert\,  
     R_1^p,\, R_2^p,\, R_3^p,\, 
     (R_1R_2R_3)^{15},\,\br_3(R_2,R_3),\,\br_3(R_3,R_1),\,  
     \right. & \\ 
&  \br_5(R_1,R_2),\,(R_1R_2)^{\frac{10p}{3p-10}},\, \br_5(R_1,R_2R_3R_2^{-1}),\, (R_1R_2R_3R_2^{-1})^{\frac{10p}{3p-10}},\\
     &\left. \, \br_{10}(R_3,R_1R_2R_1R_2^{-1}R_1^{-1}),\, (R_3R_1R_2R_1R_2^{-1}R_1^{-1})^{\frac{5p}{2p-5}}
\,\right\rangle&
\end{eqnarray*}
\end{center}

\begin{center}
Combinatorics:\\
\begin{tabular}{|c|c|c|c|c|}
\hline
  Triangle & \#($P$-orb) & Top trunc.                                      & Top ideal \\
\hline
$[3]\ 1;\ 2,\ 3$ & 15 & $p=10$                                           &\\                   
$[3]\ 23\bar2;\ 1,\ 3$ & 15 & $p=10$                                     &\\                   
$[5]\ 2;\ 1,\ 23\bar2$ & 15 & $p=5,10$                                   &\\                   
$[5]\ 3;\ 1,\ 2$ & 15 & $p=5,10$                                         &\\                   
$[10]\ (123)^2\bar212(\bar3\bar2\bar1)^2;\ 121\bar2\bar1,\ 3$ & 3 & $p=3,5,10$   &\\                   
\hline
\end{tabular}
\end{center}

\begin{figure}[htbp]
\begin{center}
\hfill
\begin{tabular}{c}
$p=2$\\
\end{tabular}
\hfill
\begin{minipage}[c]{0.16\textwidth}
  \centering
  \includegraphics[width = 0.7\textwidth]{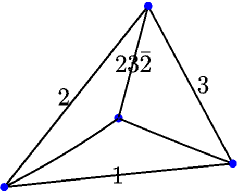}
  \end{minipage}\hfill
\begin{minipage}[c]{0.16\textwidth}
  \centering
  \includegraphics[width = 0.7\textwidth]{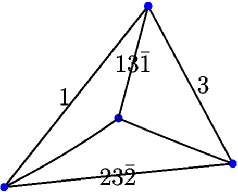}
\end{minipage}\hfill
\begin{minipage}[c]{0.16\textwidth}
  \centering
  \includegraphics[width = 0.9\textwidth]{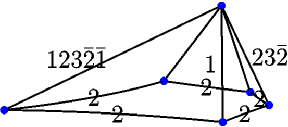}
\end{minipage}\hfill
\begin{minipage}[c]{0.16\textwidth}
  \centering
  \includegraphics[width = 0.9\textwidth]{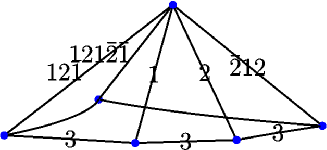}
\end{minipage}\hfill
\begin{minipage}[c]{0.16\textwidth}
  \centering
  \includegraphics[width = \textwidth]{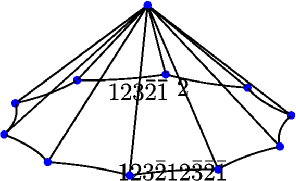}
\end{minipage}
\hfill\
\end{center}

\begin{center}
\hfill
\begin{tabular}{c}
$p=3$\\
\end{tabular}
\hfill
\begin{minipage}[c]{0.16\textwidth}
  \centering
  \includegraphics[width = 0.7\textwidth]{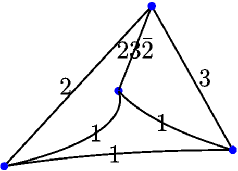}
  \end{minipage}\hfill
\begin{minipage}[c]{0.16\textwidth}
  \centering
  \includegraphics[width = 0.7\textwidth]{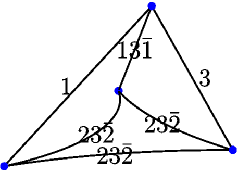}
\end{minipage}\hfill
\begin{minipage}[c]{0.16\textwidth}
  \centering
  \includegraphics[width = 0.9\textwidth]{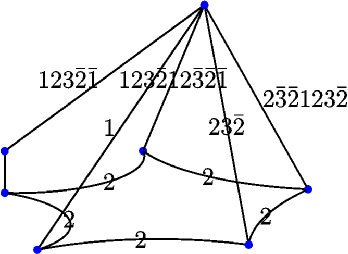}
\end{minipage}\hfill
\begin{minipage}[c]{0.16\textwidth}
  \centering
  \includegraphics[width = 0.9\textwidth]{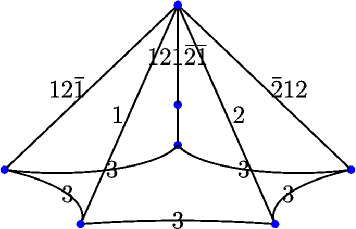}
\end{minipage}\hfill
\begin{minipage}[c]{0.16\textwidth}
  \centering
  \includegraphics[width = \textwidth]{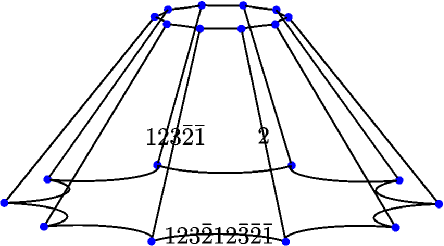}
\end{minipage}
\hfill\
\end{center}

\begin{center}
\hfill
\begin{tabular}{c}
$p=5$\\
\end{tabular}
\hfill
\begin{minipage}[c]{0.16\textwidth}  
  \centering
  \includegraphics[width = 0.6\textwidth]{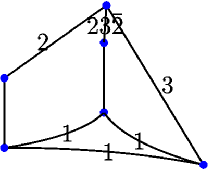}
  \end{minipage}\hfill
\begin{minipage}[c]{0.16\textwidth}
  \centering
  \includegraphics[width = 0.6\textwidth]{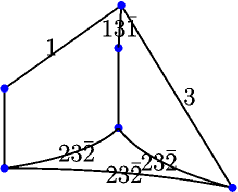}
\end{minipage}\hfill
\begin{minipage}[c]{0.16\textwidth}
  \centering
  \includegraphics[width = 0.9\textwidth]{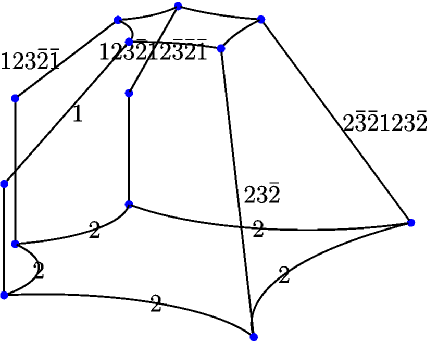}
\end{minipage}\hfill
\begin{minipage}[c]{0.16\textwidth}
  \centering
  \includegraphics[width = 0.9\textwidth]{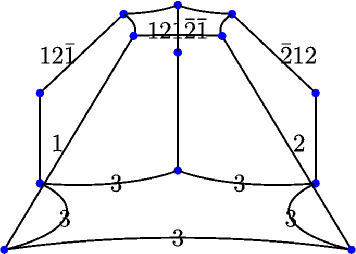}
\end{minipage}\hfill
\begin{minipage}[c]{0.16\textwidth}
  \centering
  \includegraphics[width = \textwidth]{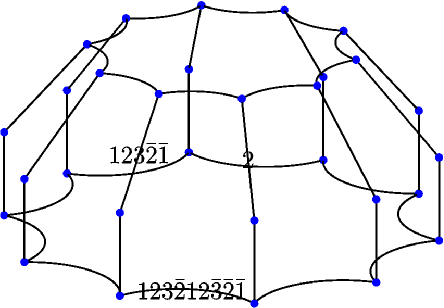}
\end{minipage}
\hfill\
\end{center}

\begin{center}
\hfill
\begin{tabular}{c}
$p=10$\\
\end{tabular}
\hfill
\begin{minipage}[c]{0.16\textwidth}
  \centering
  \includegraphics[width = 0.7\textwidth]{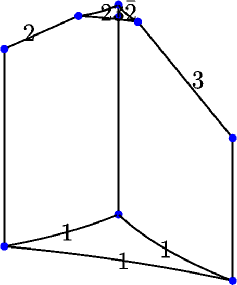}
  \end{minipage}\hfill
\begin{minipage}[c]{0.16\textwidth}
  \centering
  \includegraphics[width = 0.7\textwidth]{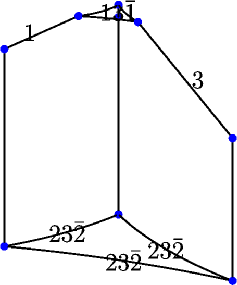}
\end{minipage}\hfill
\begin{minipage}[c]{0.16\textwidth}
  \centering
  \includegraphics[width = 0.9\textwidth]{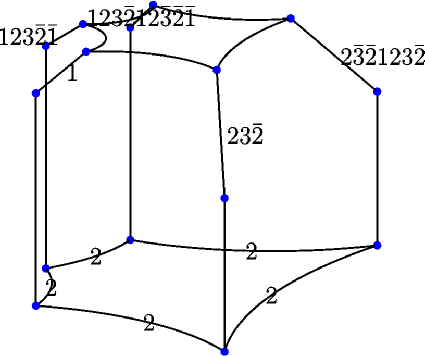}
\end{minipage}\hfill
\begin{minipage}[c]{0.16\textwidth}
  \centering
  \includegraphics[width = 0.9\textwidth]{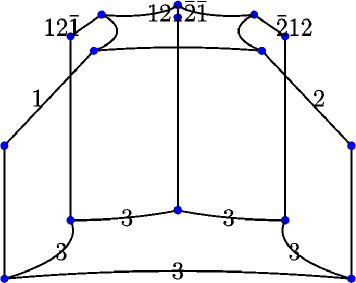}
\end{minipage}\hfill
\begin{minipage}[c]{0.16\textwidth}
  \centering
  \includegraphics[width = \textwidth]{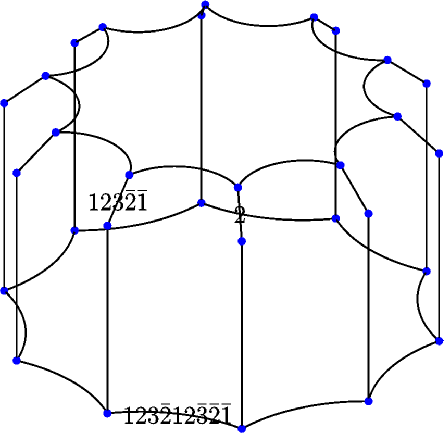}
\end{minipage}
\hfill\
\end{center}
\end{figure}

\begin{center}
Vertex stabilizers:\\
\begin{tabular}{|c|ccc|ccc|}
\hline
 $p$   &   Vertex             &  Order   &   Nature                 &    Vertex                       &   Order    & Nature \\ 
  \hline
   2   &   $p_{13}$              &   6      &    $G(3,3,2)$
       &   $p_{23}$              &   6      &    $G(3,3,2)$\\ 
       &   $p_{1,23\bar2}$       &   10     &    $G(5,5,2)$
       &   $p_{12}$              &   10     &    $G(5,5,2)$\\
       &  $p_{2,123\bar2\bar1}$  &   100    &    $G(10,2,2)^{(*)}$    &&&\\
\hline
   3   &   $p_{13}$              &   24      &    $G_4$
       &   $p_{23}$              &   24      &    $G_4$\\ 
       &   $p_{1,23\bar2}$       &   360     &    $G_{20}$
       &   $p_{12}$              &   360     &    $G_{20}$\\
       &   $p_{2,(2\cdot123\bar2\bar1)^5}$ &   45      &  $\Z_3\times\Z_{15}^{(*)}$  
       &   $p_{123\bar2\bar1,(2\cdot123\bar2\bar1)^5}$  &   45      &  $\Z_3\times\Z_{15}^{(*)}$  \\
\hline
  5    &   $p_{13}$                             &   600     &    $G_{16}$
       &   $p_{23}$                             &   600     &    $G_{16}$\\
       &   $p_{1,(123\bar2)^5}$                 &   10      &    $\Z_5\times\Z_2$
       &   $p_{2,(12)^5}$                       &   10      &    $\Z_5\times\Z_2$\\
       &   $p_{2,(\bar3\bar2\bar1)^2\bar2\bar12}$              &   25   &  $\Z_5\times\Z_5$  
       &   $p_{123\bar2\bar1,(\bar3\bar2\bar1)^2\bar2\bar12}$  &   25   &  $\Z_5\times\Z_5$\\
\hline
\end{tabular}
\end{center}

\begin{center}
Singular points\\
\begin{tabular}{|c|c|c|}
  \hline
  $p$ & Element &  Type\\
  \hline
   2,3,5  &  $Q$                & $\frac{1}{3}(1,1)$\\
  \hline
   5      & $23\bar2(Q\bar2)^2$ & $A_1$\\
   5      & $2(Q\bar3)^2$       & $A_1$\\
   5      & $2Q^{-1}$           & $\frac{1}{5}(1,2)$\\
   5      & $3Q^{-1}$           & $\frac{1}{5}(1,2)$\\
  \hline
\end{tabular}
\end{center}


\subsection{Mostow triangle groups} \label{sec:mostowinvariants}

In this section, we gather in the form of a table basic numerical
invariants for the Mostow triangle group.

In order to obtain the results below,
it is very useful to know that the group $\Gamma(p,t)$, generated by
$R_1$ and $J$ is always isomorphic to the hypergeometric monodromy
group $\Gamma_{\mu,\Sigma}$ for exponents
\begin{equation}\label{eq:converttomu}
\mu=\left(\frac{1}{2}-\frac{1}{p},\frac{1}{2}-\frac{1}{p},\frac{1}{2}-\frac{1}{p},
\frac{1}{4}+\frac{3}{2p}-\frac{t}{2},\frac{1}{4}+\frac{3}{2p}+\frac{t}{2}\right),
\end{equation}
and $\Sigma$ corresponding to permutations of the first three
weights. Moreover, if condition $\Sigma$-INT is satisfied but INT is
\emph{not} satisfied for the first three exponents (i.e. $p$ is odd),
then $\Gamma_{\mu,\Sigma}$ is the same as $\Gamma_\mu$.

From the hypergeometric exponents, one can easily read off the adjoint
trace field (which is the real subfield in the cyclotomic field
$\Q(\zeta_d)$, where $d$ is the least common denominator of the
exponents), see Lemma~(12.5) in~\cite{delignemostow}.

Presentations for various of these groups have been given in several
places,
including~\cite{mostowpacific},~\cite{dfp},~\cite{parkersurvey},~\cite{zhao}
for instance.  A unified presentation for all Deligne-Mostow groups
with three fold symmetry was given in~\cite{pasquinelli}. It is
straightforward to check that our presentation is equivalent to hers.

The non-arithmeticity index can be computed explicitly from the
hypergeometric weights, since Proposition~(12.7)
of~\cite{delignemostow} gives a formula for the signature of Galois
conjugates.

Finally, the volumes of Mostow lattices were tabulated by Sauter
in~\cite{sauter}; note that Sauter lists volumes, but volumes are
given by a universal constant ($8\pi^2/3$ if the holomorphic curvature
is normalized to be $-1$) times the orbifold Euler characteristic. A
lot of these volumes were also computed using different fundamental
domains, see~\cite{parkersurvey} for instance.

For some groups, namely $\Gamma(5,1/2)$, $\Gamma(7,3/14)$,
$\Gamma(9,1/18)$, the algorithm does not work quite as described in
section~\ref{sec:algo}, see section~\ref{sec:resultspoincare} for details.

Contrary to the previous sections of the Appendix, we do not list
vertex stabilizers and singular points, since these were already
described by Deligne and Mostow, see~\cite{delignemostowbook}.

\begin{center}
Triangle group type: 3,3,3; $k$,$k$,$k$; $2k$ 
\end{center}

\begin{center}
Presentations:
\begin{eqnarray*}
&\left\langle\,R_1, R_2, R_3, J\vphantom{(R_1R_2R_3R_2^{-1})^{\frac{6p}{p-6}}}\,\right.\left\vert\,  
     R_1^p,\, J^3,\, (R_1J)^{2k},\, JR_1J^{-1}=R_2,\, JR_2J^{-1}=R_3,\,  
     \right. & \\ 
& \left. \br_3(R_1,R_2),\,(R_1R_2)^{\frac{6p}{p-6}}, 
     (R_2R_1J)^{\frac{2kp}{(k-2)p-2k}}
\,\right\rangle&
\end{eqnarray*}
\end{center}


\begin{center}
Combinatorics:\\
\begin{tabular}{|c|c|c|}
\hline
   Triangle & \#(P-orbit) \\
\hline
 $[3]\ 1;\ 2,\ 3$ &  $2k$  \\                   
 $[k]\ 2;\ 1,\ 23\bar2$ &  2  \\
\hline
\end{tabular}
 \label{tab:rough_mostow}
\end{center}
\begin{rk} Note $2k$ stands for the order of $P=R_1J$. The second type of faces should be omitted when $k=2$, in which case
$1$ and $23\bar2$ commute, i.e. they braid with order $2$. The latter
  groups correspond to the Livn\'e family.
\end{rk}

\begin{center}
  Commensurability invariants:
\end{center}
\begin{center}
  $k=2$
\begin{tabular}{|c|c|c|c|c|c|c|c|c|}
\hline
  $p$ & $t$ & $o(P)$ & $\chi^{orb}$ & $\Q(\tr{\rm Ad}\ \Gamma)$  &  CM field        & C?  &  A?     \\ 
\hline
  $5$  & $7/10$      &  4       & 1/200  & $\Q(\sqrt{5})$        & $\Q(\zeta_5)$    & C   &  A      \\
  $6$  & $2/3$       &  4       & 1/72   & $\Q$                  & $\Q(\zeta_3)$    & NC  &  A      \\
  $7$  & $9/14$      &  4       & 1/49   & $\Q(\cos(2\pi/7))$    & $\Q(\zeta_7)$    & C   &  A      \\
  $8$  & $5/8$       &  4       & 3/128  & $\Q(\sqrt{2})$        & $\Q(\zeta_8)$    & C   &  A      \\
  $9$  & $11/18$     &  4       & 2/81   & $\Q(\cos(2\pi/9))$    & $\Q(\zeta_9)$    & C   &  NA(1)  \\
  $10$ & $3/5$       &  4       & 1/40   & $\Q(\sqrt{5})$        & $\Q(\zeta_5)$    & C   &  A      \\
  $12$ & $7/12$      &  4       & 7/288  & $\Q(\sqrt{3})$        & $\Q(\zeta_{12})$ & C   &  A      \\
  $18$ & $5/9$       &  4       & 13/648 & $\Q(\cos(2\pi/9))$    & $\Q(\zeta_9)$    & C   &  A      \\
\hline
\end{tabular}
\end{center}

\begin{center}
  $k=3$
\begin{tabular}{|c|c|c|c|c|c|c|c|c|}
\hline
  $p$ & $t$ & $o(P)$ & $\chi^{orb}$ & $\Q(\tr{\rm Ad}\ \Gamma)$  & CM field         & C?  &  A?     \\ 
\hline
  $4$  & $5/12$      &  6       & 1/72   & $\Q(\sqrt{3})$        & $\Q(\zeta_{12})$ & C   &  A     \\
  $5$  & $11/30$     &  6       & 8/225  & $\Q(\cos(2\pi/15))$   & $\Q(\zeta_{15})$ & C   &  A     \\
  $6$  & $1/3$       &  6       & 1/18   & $\Q$                  & $\Q(\zeta_3)$    & NC  &  A     \\
  $7$  & $13/42$     &  6       & 61/882 & $\Q(\cos(2\pi/21))$   & $\Q(\zeta_{21})$ & C   &  NA(2) \\
  $8$  & $7/24$      &  6       & 11/144 & $\Q(\cos(2\pi/24))$   & $\Q(\zeta_{24})$ & C   &  NA(1) \\
  $9$  & $5/18$      &  6       & 13/162 & $\Q(\cos(2\pi/9))$    & $\Q(\zeta_9)$    & C   &  A     \\
  $10$ & $4/15$      &  6       & 37/450 & $\Q(\cos(2\pi/15))$   & $\Q(\zeta_{15})$ & C   &  NA(2) \\
  $12$ & $1/4$       &  6       & 1/12   & $\Q(\sqrt{3})$        & $\Q(\zeta_{12})$ & C   &  A     \\
  $18$ & $2/9$       &  6       & 13/162 & $\Q(\cos(2\pi/9))$    & $\Q(\zeta_9)$    & C   &  A     \\
\hline
\end{tabular}
\end{center}

\begin{center}
  $k=4$
\begin{tabular}{|c|c|c|c|c|c|c|c|c|}
\hline
  $p$ & $t$ & $o(P)$ & $\chi^{orb}$ & $\Q(\tr{\rm Ad}\ \Gamma)$  & CM field         & C?  &  A?     \\ 
\hline
  $3$  & $1/3$       &  8       & 1/288  & $\Q(\sqrt{3})$        & $\Q(\zeta_{12})$ & C   &  A      \\
  $4$  & $1/4$       &  8       & 1/32   & $\Q$                  & $\Q(\zeta_{4})$  & NC  &  A      \\
  $5$  & $1/5$       &  8       & 23/400 & $\Q(\cos(2\pi/20))$   & $\Q(\zeta_{20})$ & C   &  NA(1)  \\
  $6$  & $1/6$       &  8       & 11/144 & $\Q(\sqrt{3})$        & $\Q(\zeta_{12})$ & NC  &  NA(1)  \\
  $8$  & $1/8$       &  8       & 3/32   & $\Q(\sqrt{2})$        & $\Q(\zeta_{8})$  & C   &  A      \\
  $12$ & $1/12$      &  8       & 7/72   & $\Q(\sqrt{3})$        & $\Q(\zeta_{12})$ & C   &  A      \\
\hline
\end{tabular}
\end{center}

\begin{center}
  $k=5$
\begin{tabular}{|c|c|c|c|c|c|c|c|}
\hline
  $p$  & $t$         & $o(P)$   & $\chi^{orb}$ & $\Q(\tr{\rm Ad}\ \Gamma)$ & CM field         & C?  &  A?     \\ 
\hline
  $3$  & $7/30$      &  10      & 2/225        & $\Q(\cos(2\pi/15))$       & $\Q(\zeta_{15})$ & C   &  A     \\
  $4$  & $3/20$      &  10      & 33/800       & $\Q(\cos(2\pi/20))$       & $\Q(\zeta_{20})$ & C   &  NA(2)  \\
  $5$  & $1/10$      &  10      & 13/200       & $\Q(\sqrt{5})$            & $\Q(\zeta_{5})$  & C   &  A      \\
  $10$ & $0$         &  10      & 1/10         & $\Q(\sqrt{5})$            & $\Q(\zeta_{5})$  & C   &  A      \\
\hline
\end{tabular}
\end{center}

\begin{center}
  $k=6$
\begin{tabular}{|c|c|c|c|c|c|c|c|}
\hline
  $p$ & $t$ & $o(P)$ & $\chi^{orb}$ & $\Q(\tr{\rm Ad}\ \Gamma)$  & CM field         & C?  &  A?     \\ 
\hline
  $3$  & $1/6$       &  12      & 1/72   & $\Q$                  & $\Q(\zeta_{3})$  & NC  &  A      \\
  $4$  & $1/12$      &  12      & 13/288 & $\Q(\sqrt{3})$        & $\Q(\zeta_{12})$ & C   &  NA(1)  \\
  $6$  & $0$         &  12      & 1/12   & $\Q$                  & $\Q(\zeta_{3})$  & NC  &  A      \\
\hline
\end{tabular}
\end{center}

\begin{center}
  $k=7$
\begin{tabular}{|c|c|c|c|c|c|c|c|}
\hline
  $p$  & $t$         & $o(P)$   & $\chi^{orb}$ & $\Q(\tr{\rm Ad}\ \Gamma)$  & CM field         &  C?  &  A?     \\ 
\hline
  $3$  & $5/42$      &  14      & 61/3528      & $\Q(\cos(2\pi/21))$        & $\Q(\zeta_{21})$ & C   &  NA(2) \\
\hline
\end{tabular}
\end{center}

\begin{center}
  $k=8$
\begin{tabular}{|c|c|c|c|c|c|c|c|}
\hline
  $p$  & $t$         & $o(P)$   & $\chi^{orb}$ & $\Q(\tr{\rm Ad}\ \Gamma)$  & CM field         & C?  &  A?     \\ 
\hline
  $3$  & $1/12$      &  16      & 11/576       & $\Q(\cos(2\pi/24))$        & $\Q(\zeta_{24})$ & C   &  NA(1)  \\
  $4$  & $0$         &  16      & 3/64         & $\Q(\sqrt{2})$             & $\Q(\zeta_{8})$  & C   &  A      \\
\hline
\end{tabular}
\end{center}

\begin{center}
  $k=9$
\begin{tabular}{|c|c|c|c|c|c|c|c|}
\hline
  $p$  & $t$         & $o(P)$   & $\chi^{orb}$ & $\Q(\tr{\rm Ad}\ \Gamma)$  & CM field        & C?  &  A?     \\ 
\hline
  $3$  & $1/18$      &  18      & 13/648       & $\Q(\cos(2\pi/9))$         & $\Q(\zeta_{9})$ & C   &  A      \\
\hline
\end{tabular}
\end{center}

\begin{center}
  $k=10$
\begin{tabular}{|c|c|c|c|c|c|c|c|}
\hline
  $p$ & $t$         & $o(P)$   & $\chi^{orb}$ & $\Q(\tr{\rm Ad}\ \Gamma)$  & CM field         & C?  &  A?     \\ 
\hline
  $3$ & $1/30$      &  20      & 37/1800      & $\Q(\cos(2\pi/15))$        & $\Q(\zeta_{15})$ & C   &  NA(2) \\
\hline
\end{tabular}
\end{center}

\begin{center}
  $k=12$
\begin{tabular}{|c|c|c|c|c|c|c|c|}
\hline
  $p$  & $t$   & $o(P)$ & $\chi^{orb}$ & $\Q(\tr{\rm Ad}\ \Gamma)$  & CM field         & C?  &  A?     \\ 
\hline
  $3$  & $0$   &  24    & 1/48         & $\Q(\sqrt{3})$             & $\Q(\zeta_{12})$ & C   &  A      \\
\hline
\end{tabular}
\end{center}

\begin{center}
  $k=5/2$
\begin{tabular}{:c:c:c:c:c:c:c:c:c:}
\hdashline
  $p$ & $t$   & $o(P)$ & $\chi^{orb}$ & $\Q(\tr{\rm Ad}\ \Gamma)$  & CM field        & C?  &  A? & Alt.\\ 
\hdashline
  $5$ & $1/2$ &  10    & 1/200        & $\Q(\sqrt{5})$             & $\Q(\zeta_{5})$ & C   &  A  &  $\Gamma(5,7/10)$          \\
\hdashline
\end{tabular}
\end{center}

\begin{center}
  $k=7/2$
\begin{tabular}{:c:c:c:c:c:c:c:c:c:}
\hdashline
  $p$ & $t$    & $o(P)$ & $\chi^{orb}$ & $\Q(\tr{\rm Ad}\ \Gamma)$  & CM field        & C?  &  A? & Alt.\\ 
\hdashline
  $7$ & $3/14$ &  14    & 1/49         & $\Q(\cos(2\pi/7))$         & $\Q(\zeta_{7})$ & C   &  A  & $\Gamma(7,9/14)$     \\
\hdashline
\end{tabular}
\end{center}

\begin{center}
  $k=9/2$
\begin{tabular}{:c:c:c:c:c:c:c:c:c:}
\hdashline
  $p$  & $t$    & $o(P)$ & $\chi^{orb}$ & $\Q(\tr{\rm Ad}\ \Gamma)$  & CM field        & C?  &  A?     & Alt.  \\ 
\hdashline
  $9$  & $1/18$ & 18     & 2/81         & $\Q(\cos(2\pi/9))$         & $\Q(\zeta_{9})$ & C   &  NA(1)  & $\Gamma(9,11/18)$  \\
\hdashline
\end{tabular}
\end{center}

\subsection{Deligne-Mostow groups without 3-fold symmetry}

In order to check that our lattices are not commensurable to any
Deligne-Mostow lattice, we also need to consider the handful of groups
in the Deligne-Mostow list whose hypergeometric exponents do not have
a 3-fold symmetry. In Table~\ref{tab:dm}, we list such non-arithmetic
groups, their orbifold Euler characteristics and the rough
commensurability invariants.

It turns out that the group $\Gamma_\mu$ with $\mu=(4,4,5,5,6)/12$ is
commensurable to the Mostow group $\Gamma(4,1/12)$, and the group with
$\mu=(6,6,9,9,10)/20$ is commensurable to $\Gamma(4,3/20)$,
see~\cite{delignemostowbook},~\cite{kappesmoller} or~\cite{mcmullengaussbonnet}.
\begin{table}[htbp]
\begin{tabular}{|c|c|c|c|c|c|}\hline
  $\mu$              & $\chi^{orb}(\Gamma_{\mu,\Sigma})$  & $\Q(\tr{\rm Ad}\ \Gamma)$   & CM field         & C? & A? \\ \hline

  $(4,4,5,5,6)/12$         & 13/96   & $\Q(\sqrt{3})$                                   & $\Q(\zeta_{12})$ & C  &  NA(1)          \\
  $(3,3,5,6,7)/12$         & 17/96   & $\Q(\sqrt{3})$                                   & $\Q(\zeta_{12})$ & NC &  NA(1)          \\
  $(6,6,9,9,10)/20$        & 99/800  & $\Q(\cos(2\pi/20))$                              & $\Q(\zeta_{20})$ & C  &  NA(2)          \\\hline
\end{tabular}
\caption{Invariants for non-arithmetic Deligne-Mostow groups
  $\Gamma_{\mu,\Sigma}$, such that $\mu$ has no 3-fold symmetry. Here
  $\Sigma$ stands for the full symmetry group of $\mu$.
}\label{tab:dm}
\end{table}

\subsection{Commensurability classes}

In this section we summarize the analysis of commensurability classes
of non-arithmetic lattices obtained in our paper, which brings to 22
the number of currently known non-arithmetic lattices in $\pu(2,1)$.

The result is given in Table~\ref{tab:commclasses}; groups in
different large boxes are in distinct commensurability classes, either
because they have different adjoint trace fields or because one is
cocompact and the other is not. Within a large box, we separate groups
by a solid line if they are known to be in distinct commensurability
classes.

\dashlinegap=2pt
\begin{table}
{
\small
\begin{tabular}{| c | c|c|c|c|c|c | c|c|c|c|c|c |}
\hline
$\Q(\tr{\rm Ad}\ \Gamma)$          &         \multicolumn{6}{c}{C}                                    &   \multicolumn{6}{|c|}{NC}  \\\hline\hline
$\Q(\sqrt{2})$                     &         \multicolumn{6}{c}{}                                     &   \multicolumn{6}{|c|}{$\S(4,\sigma_1)$} \\\hline
$\Q(\sqrt{3})$                     & \multicolumn{6}{c}{$\Gamma(4,\frac{1}{12})$, $\frac{1}{12}(4,4,5,5,6)$}     &   \multicolumn{2}{|c|}{$\Gamma(6,\frac{1}{6})$} & 
                                                                                                          \multicolumn{4}{c|}{$\frac{1}{12}(3,3,5,6,7)$,$\T(4,{\bf E_2})$}\\\hline
$\Q(\sqrt{5})$                     &         \multicolumn{6}{c}{}                                     &   \multicolumn{6}{|c|}{$\S(3,\sigma_5)$} \\\hline
$\Q(\sqrt{6})$                     &         \multicolumn{6}{c|}{}                                    &   \multicolumn{3}{c|}{\qquad$\S(3,\sigma_1)$\qquad \ } &
                                                                                                          \multicolumn{3}{c|}{$\S(6,\sigma_1)$}\\\hline
$\Q(\sqrt{7})$                     &         \multicolumn{6}{c|}{}                                    &   \multicolumn{6}{|c|}{$\S(4,\overline{\sigma}_4)$} \\\hline
$\Q(\sqrt{21})$                    &         \multicolumn{6}{c|}{}                                    &   \multicolumn{6}{|c|}{$\S(6,\overline{\sigma}_4)$} \\\hline
$\Q(\sqrt{2},\sqrt{3})$            &      \multicolumn{6}{c|}{}                                       &   \multicolumn{6}{|c|}{
                                                                                                            $\Gamma(3,1/12),\Gamma(8,\frac{7}{24})$}\\\hline
$\Q(\sqrt{2},\sqrt{7})$            & \multicolumn{6}{c|}{$\S(8,\overline{\sigma}_4)$}                 & \multicolumn{6}{c|}{} \\\hline
$\Q(\sqrt{3},\sqrt{5})$            &      \multicolumn{6}{c|}{}                                       &   \multicolumn{3}{c|}{\qquad$\S(4,\sigma_5)$\qquad \ } & 
                                                                                                          \multicolumn{3}{c|}{$\T(4,{\bf S_2})$}\\\hline
$\Q(\sqrt{3},\sqrt{7})$            &   \multicolumn{6}{c|}{$\S(12,\overline{\sigma}_4)$}              & \multicolumn{6}{c|}{}\\\hline
$\Q(\sqrt{\frac{5+\sqrt{5}}{2}})$  & \multicolumn{3}{c|}{$\Gamma(4,\frac{3}{20})$,$\frac{1}{20}(6,6,9,9,10)$}&    
                                     \multicolumn{3}{c|}{$\Gamma(5,\frac{1}{5})$}                             & \multicolumn{6}{c|}{}\\\hline
$\Q(\sqrt{\frac{5+\sqrt{5}}{14}})$        &       \multicolumn{6}{|c|}{$\S(5,\overline{\sigma}_4)$}           & \multicolumn{6}{|c|}{}\\\hline
$\Q(\cos\frac{2\pi}{9})$           & \multicolumn{6}{c|}{$\Gamma(9,11/18),\Gamma(9,1/18)$}    & \multicolumn{6}{c|}{}\\\hline
$\Q(\cos\frac{2\pi}{15})$          & \multicolumn{2}{c|}{$\Gamma(10,\frac{4}{15}),\Gamma(3,\frac{1}{30})$} &
                                          \multicolumn{2}{c|}{$\T(5,{\bf S_2})$} & 
                                          \multicolumn{2}{c|}{$\T(3,{\bf H_2})$}
                                                                 & \multicolumn{6}{c|}{}\\\hline
$\Q(\cos\frac{2\pi}{21})$          &  
                                          \multicolumn{6}{c|}{$\Gamma(3,\frac{5}{42}),\Gamma(7,\frac{13}{42})$} &\multicolumn{6}{c|}{}\\\hline

\end{tabular}
}
\caption{Commensurability classes of non-arithmetic lattice triangle groups.}\label{tab:commclasses}
\end{table}

\end{appendices}

\begin{flushleft}
  \textsc{Martin Deraux:\\
    Institut Fourier, Universit\'e Grenoble Alpes}\\
  \verb|martin.deraux@univ-grenoble-alpes.fr|
\end{flushleft}

\begin{flushleft}
  \textsc{John R. Parker:\\
    Department of Mathematical Sciences, Durham University}\\
    \verb|j.r.parker@durham.ac.uk|
\end{flushleft}

\begin{flushleft}
  \textsc{Julien Paupert:\\
    School of Mathematical and Statistical Sciences, Arizona State University}\\
       \verb|paupert@asu.edu|
\end{flushleft}

\end{document}